\pdfoutput=1
\documentclass[a4paper,reqno,11pt]{amsart}
\usepackage[margin=3cm]{geometry}

\newcommand{\ifarticle}[2]{
    \csname@ifclassloaded\endcsname{beamer}{#2}{#1}
}

\newcommand{\ifbook}[2]{
    \csname@ifclassloaded\endcsname{amsbook}{#1}{#2}
}

    \usepackage{xparse} 
    \usepackage{xspace} 
    \usepackage{etoolbox} 
    \usepackage{xpatch} 
    \usepackage{pgffor} 

    \usepackage{amsmath,amssymb,amsthm} 
    \allowdisplaybreaks[1]

    \usepackage{mathtools} 
    \usepackage{mathrsfs} 
    \usepackage{stmaryrd} 
    \usepackage{bbm} 
    \usepackage{quiver} 
    \usepackage{xcolor} 
    \usepackage{scalerel} 
    \usepackage{booktabs} 
    \usepackage{nameref} 
    \usepackage[british]{babel} 
    \usepackage{csquotes} 
    \usepackage[UKenglish]{isodate} 
    \usepackage{microtype} 
    \usepackage[splitrule]{footmisc} 

    \ifarticle{
        \PassOptionsToPackage{hyphens}{url}
        \usepackage[pdfusetitle,linktoc=all,colorlinks,citecolor=cyan,linkcolor=purple,urlcolor=blue]{hyperref} 
        \urlstyle{rm}

        \usepackage[nameinlink,noabbrev,capitalize]{cleveref} 
        \usepackage{footnotebackref} 

        \usepackage[shortlabels]{enumitem} 
        \setlist{topsep=2pt,itemsep=2pt,partopsep=2pt,parsep=2pt} 

        \setcounter{tocdepth}{1}
    }{}
    \usepackage[style=alphabetic,backref=true,backrefstyle=three,maxnames=9,minalphanames=3,maxalphanames=4]{biblatex} 
    \DeclareUnicodeCharacter{0301}{\TODO[Invalid symbol.]} 

    \DeclareFieldFormat*{citetitle}{\emph{#1}} 
    \DeclareCiteCommand{\citetitle}
        {\boolfalse{citetracker}%
        \boolfalse{pagetracker}%
        \usebibmacro{prenote}}
        {\ifciteindex
            {\indexfield{indextitle}}
            {}%
        \printtext[bibhyperref]{\printfield[citetitle]{labeltitle}}}
        {\multicitedelim}
        {\usebibmacro{postnote}}
    \DeclareCiteCommand*{\citeauthor}
        {\defcounter{maxnames}{99}%
        \defcounter{minnames}{99}%
        \defcounter{uniquename}{2}%
        \boolfalse{citetracker}%
        \boolfalse{pagetracker}%
        \usebibmacro{prenote}}
        {\ifciteindex{\indexnames{labelname}}{}%
        \printnames{labelname}}
        {\multicitedelim}
        {\usebibmacro{postnote}}

    \usepackage{calc} 
    \usepackage{tabu} 
    \usepackage{ebproof} 
    \usepackage{forest} 
    \usepackage{adjustbox} 
    \usepackage{float} 
    \usepackage{stfloats}

    \makeatletter
    \ifarticle{
        \xpretocmd{\@adminfootnotes}{\let\@makefntext\BHFN@OldMakefntext}{}{}
        \renewcommand\@makefntext[1]{%
        \@ifundefined{@makefnmark}
            {}
            {%
            \renewcommand\@makefnmark{%
            \mbox{%
                \textsuperscript{%
                \normalfont
                \hyperref[\BackrefFootnoteTag]{\@thefnmark}%
                }%
            }\,%
            }%
            \BHFN@OldMakefntext{#1}%
        }%
        }

        \LetLtxMacro{\BHFN@Old@footnotemark}{\@footnotemark}
        \renewcommand*{\@footnotemark}{%
            \refstepcounter{BackrefHyperFootnoteCounter}%
            \xdef\BackrefFootnoteTag{bhfn:\theBackrefHyperFootnoteCounter}%
            \label{\BackrefFootnoteTag}%
            \BHFN@Old@footnotemark
        }

        \makeatletter
        \def\paragraph{\@startsection{paragraph}{4}%
          \z@\z@{-\fontdimen2\font}%
          {\normalfont\bfseries}}
        \makeatother
    }{}
    \makeatother

    \ifarticle{
        \theoremstyle{plain}
        \ifbook{
            \newtheorem{theorem}{Theorem}[chapter]
        }{
            \newtheorem{theorem}{Theorem}[section]
        }
        \newtheorem{proposition}[theorem]{Proposition}
        \newtheorem{lemma}[theorem]{Lemma}
        \newtheorem{corollary}[theorem]{Corollary}

        \newtheorem*{theorem*}{Theorem}
        \newtheorem*{corollary*}{Corollary}

        \theoremstyle{definition}
        \newtheorem{definition}[theorem]{Definition}
        
        \newtheorem{example}[theorem]{Example}
        
        \newtheorem{notation}[theorem]{Notation}
        
        \newtheorem{remark}[theorem]{Remark}

        \newenvironment{sketch}{\proof}{\endproof}

        \Crefname{theoremenumi}{Theorem}{Theorems}
        \AtBeginEnvironment{theorem}{%
            \crefalias{enumi}{theoremenumi}%
            \setlist[enumerate,1]{
                ref={\csname thetheorem\endcsname.(\arabic*)}
            }%
            \crefalias{enumii}{theoremenumi}%
            \setlist[enumerate,2]{
                ref={\thetheorem.(\arabic*).(\alph*)}
            }%
        }

        \Crefname{propositionenumi}{Proposition}{Propositions}
        \AtBeginEnvironment{proposition}{%
            \crefalias{enumi}{propositionenumi}%
            \setlist[enumerate,1]{
                ref={\csname theproposition\endcsname.(\arabic*)}
            }%
            \crefalias{enumii}{propositionenumi}%
            \setlist[enumerate,2]{
                ref={\theproposition.(\arabic*).(\alph*)}
            }%
        }

        \Crefname{lemmaenumi}{Lemma}{Lemmas}
        \AtBeginEnvironment{lemma}{%
            \crefalias{enumi}{lemmaenumi}%
            \setlist[enumerate,1]{
                ref={\csname thelemma\endcsname.(\arabic*)}
            }%
            \crefalias{enumii}{lemmaenumi}%
            \setlist[enumerate,2]{
                ref={\thelemma.(\arabic*).(\alph*)}
            }%
        }

        \Crefname{corollaryenumi}{Corollary}{Corollaries}
        \AtBeginEnvironment{corollary}{%
            \crefalias{enumi}{corollaryenumi}%
            \setlist[enumerate,1]{
                ref={\csname thecorollary\endcsname.(\arabic*)}
            }%
            \crefalias{enumii}{corollaryenumi}%
            \setlist[enumerate,2]{
                ref={\thecorollary.(\arabic*).(\alph*)}
            }%
        }

        \Crefname{definitionenumi}{Definition}{Definitions}
        \AtBeginEnvironment{definition}{%
            \crefalias{enumi}{definitionenumi}%
            \setlist[enumerate,1]{
                ref={\csname thedefinition\endcsname.(\arabic*)}
            }%
            \crefalias{enumii}{definitionenumi}%
            \setlist[enumerate,2]{
                ref={\thedefinition.(\arabic*).(\alph*)}
            }%
        }

        \Crefname{exampleenumi}{Example}{Examples}
        \AtBeginEnvironment{example}{%
            \crefalias{enumi}{exampleenumi}%
            \setlist[enumerate,1]{
                ref={\csname theexample\endcsname.(\arabic*)}
            }%
            \crefalias{enumii}{exampleenumi}%
            \setlist[enumerate,2]{
                ref={\theexample.(\arabic*).(\alph*)}
            }%
        }

        \newcommand{\qedshift}{\vspace*{-\baselineskip}}

        \foreach \env in {definition,remark,example,notation}{
        \AtBeginEnvironment{\env}{%
          \pushQED{\qed}%
        }
        \AtEndEnvironment{\env}{\popQED\endexample}
        }
    }{}

    \ExplSyntaxOn
    \NewDocumentCommand{\mathcommand}{mO{0}m}
     {
      \exp_args:Nc \NewCommandCopy {khue_kept_\cs_to_str:N #1} { #1 }
      \exp_args:Nc \newcommand {khue_new_\cs_to_str:N #1}[#2]{#3}
      \DeclareDocumentCommand {#1} {}
       {
        \mode_if_math:TF
         {
          \use:c {khue_new_\cs_to_str:N #1}
         }
         {
          \use:c {khue_kept_\cs_to_str:N #1}
         }
       }
     }
    \ExplSyntaxOff

    \mathcommand{\h}{\textup{-}}

    \newcommand{\tx}{\mathrm}
    \mathcommand{\b}{\mathbf}
    \newcommand{\s}{\mathsf}
    \newcommand{\cl}{\mathcal}
    \mathcommand{\bb}{\mathbb}
    \DeclareMathAlphabet{\bbn}{U}{bbold}{m}{n}
    \newcommand{\dc}[1]{\TextOrMath{double category\xspace#1}{\b{\bb#1}}}
    \newcommand{\scr}{\mathscr}
    \mathcommand{\sf}{\mathsf}
    \mathcommand{\u}{\underline}

    \newcommand{\TODO}[1][TODO]{\textcolor{orange}{\textup{#1}}\xspace}

    \newcommand{\flip}[1]{\text{\rotatebox[origin=c]{-180}{$#1$}}}
    
    \newcommand{\datetoday}{\date{\cleanlookdateon\today}}

    \newcommand{\defeq}{\mathrel{:=}}
    \mathcommand{\d}{\mathbin{;}}
    \mathcommand{\c}{\circ}
    \newcommand{\ph}[1][]{{({-}_{#1})}}
    
    \newcommand{\iso}{\cong}
    
    \newcommand{\from}{\leftarrow}
    \newcommand{\xto}{\xrightarrow}
    
    \newcommand{\tto}{\Rightarrow}

    \newcommand{\ffto}{\hookrightarrow}
    
    \newcommand*\cocolon{%
            \nobreak
            \mskip6mu plus1mu
            \mathpunct{}%
            \nonscript
            \mkern-\thinmuskip
            {:}%
            \mskip2mu
            \relax
    }

    \makeatletter
    \def\slashedarrowfill@#1#2#3#4#5{%
    $\m@th\thickmuskip0mu\medmuskip\thickmuskip\thinmuskip\thickmuskip
    \relax#5#1\mkern-7mu%
    \cleaders\hbox{$#5\mkern-2mu#2\mkern-2mu$}\hfill
    \mathclap{#3}\mathclap{#2}%
    \cleaders\hbox{$#5\mkern-2mu#2\mkern-2mu$}\hfill
    \mkern-7mu#4$%
    }
    \def\rightslashedarrowfill@{%
    \slashedarrowfill@\relbar\relbar\mapstochar\rightarrow}
    \newcommand\xslashedrightarrow[2][]{%
    \ext@arrow 0055{\rightslashedarrowfill@}{#1}{#2}}
    \def\leftslashedarrowfill@{%
    \slashedarrowfill@\leftarrow\relbar\mapsfromchar\relbar}
    \newcommand\xslashedleftarrow[2][]{%
    \ext@arrow 0055{\leftslashedarrowfill@}{#1}{#2}}
    \makeatother

    \newcommand{\lto}{\xlto{}}

    \newcommand{\inv}{^{-1}}
    \newcommand{\op}{{}^\tx{op}}

    \newcommand{\tp}[1]{\langle#1\rangle}
    \newcommand{\unit}{{\tp{}}}
    
    \newcommand{\ob}[1]{|#1|}

    \DeclareFontFamily{U}{min}{}
    \DeclareFontShape{U}{min}{m}{n}{<-> udmj30}{}

    \mathcommand{\comma}{\downarrow}

    \newcommand{\copi}{\flip\pi}
    
    \newsavebox{\whitecircstar}\sbox{\whitecircstar}{\kern.075em\tikz{\node[draw, circle,line width=.36pt, inner sep=0]{$*$};}\kern.075em}
    
    \newsavebox{\blackcircstar}\sbox{\blackcircstar}{\kern.075em\tikz{\node[fill, circle, line width=.36pt, inner sep=0, text=white]{$*$};}\kern.075em}

    \newcommand{\pow}{\pitchfork}
    \newcommand{\copow}{\cdot}

    \makeatletter
    \def\widebreve{\mathpalette\wide@breve}
    \def\wide@breve#1#2{\sbox\z@{$#1#2$}%
         \mathop{\vbox{\m@th\ialign{##\crcr
    \kern0.08em\brevefill#1{0.8\wd\z@}\crcr\noalign{\nointerlineskip}%
                        $\hss#1#2\hss$\crcr}}}\limits}
    \def\brevefill#1#2{$\m@th\sbox\tw@{$#1($}%
      \hss\resizebox{#2}{\wd\tw@}{\rotatebox[origin=c]{90}{\upshape(}}\hss$}
    \makeatother

    \NewDocumentCommand{\jrule}{om}{%
        \IfNoValueTF{#1}
            {\textsc{#2}}
            {$#1$-\textsc{#2}}%
    }

    \newcommand{\N}{{\bb N}}
    
    \newcommand{\Set}{{\b{Set}}}
    
    \newcommand{\V}{{\bb V}} 
    \newcommand{\Cat}{\b{Cat}}
    \newcommand{\CAT}{\b{CAT}}

    \newcommand{\ff}{fully faithful}

    \newcommand{\ioo}{identity-on-objects}

    \newcommand{\eg}{e.g.\@\xspace}
    \newcommand{\ie}{i.e.\@\xspace}
    
    \newcommand{\cf}{cf.\@\xspace}
    
    \newcommand{\aka}{a.k.a.\@\xspace}
    \NewDocumentCommand{\etc}{t.}{etc.\@\xspace}
    \NewDocumentCommand{\ibid}{t.}{ibid.\@\xspace}
    \NewDocumentCommand{\loccit}{t.}{loc.\ cit.\@\xspace}

    \newcommand{\lfp}{locally finitely presentable}

\makeatletter

\define@key{beamerframe}{c}[true]{
    \beamer@frametopskip=0pt plus 1fill\relax%
    \beamer@framebottomskip=0pt plus 1fill\relax%
}

\patchcmd{\beamer@sectionintoc}{\vfill}{\vskip\itemsep}{}{}

\makeatother

\ifarticle{}{

  \colorlet{colour-bg}{black!85} 
  \definecolor{colour-primary}{HTML}{cc80ff} 
  \colorlet{colour-text}{black!10} 
  \colorlet{colour-subtle}{black!40} 
  \colorlet{colour-block-bg}{black!80} 
  \definecolor{colour-warning-bg}{HTML}{ffea80} 
  \definecolor{colour-warning-primary}{HTML}{e08152} 

  \hypersetup{linktoc=all,colorlinks, citecolor={colour-primary}, linkcolor={colour-primary}, urlcolor={colour-primary}} 
  \urlstyle{sf}

  \usepackage{changepage} 

  \usepackage{ragged2e} 

  \setbeamertemplate{section in toc}[sections numbered]

  \apptocmd{\frame}{}{\justifying}{}

  \usefonttheme[onlymath]{serif}

  \beamertemplatenavigationsymbolsempty
  \setbeamertemplate{footline}{
    \hfill%
    \usebeamercolor[fg]{page number in head/foot}%
    \usebeamerfont{page number in head/foot}%
    \setbeamertemplate{page number in head/foot}[pagenumber]%
    \usebeamertemplate*{page number in head/foot}\kern.2cm\vskip.2cm%
  }

  \setbeamertemplate{frametitle}[default][center]
  \addtobeamertemplate{frametitle}{\vspace*{1ex}}{\vspace*{1ex}}

  \setbeamertemplate{itemize item}{$\bullet$}

  \setbeamertemplate{theorems}[numbered]
  \newtheorem{proposition}[theorem]{\translate{Proposition}}

  \addtobeamertemplate{block begin}
  {}
  {\vspace{0ex} 
  \begin{adjustwidth}{1em}{1em} 
  \tikzcdset{background color=colour-block-bg}
  }
  \addtobeamertemplate{block end}{\end{adjustwidth}%
  \vspace{1ex}} 
  {}
  \renewenvironment<>{block}[1]{%
      \begin{actionenv}#2%
        \def\insertblocktitle{\begin{adjustwidth}{.8em}{.8em}#1\end{adjustwidth}}%
        \par%
        \usebeamertemplate{block begin}}
      {\par%
        \usebeamertemplate{block end}%
      \end{actionenv}}

  \addtobeamertemplate{block example begin}
  {}
  {\vspace{0ex} 
  \begin{adjustwidth}{1em}{1em} 
  \tikzcdset{background color=colour-block-bg}
  }
  \addtobeamertemplate{block example end}{\end{adjustwidth}%
  \vspace{1ex}} 
  {}
  \renewenvironment<>{exampleblock}[1]{%
      \begin{actionenv}#2%
          \def\insertblocktitle{\begin{adjustwidth}{.8em}{.8em}#1\end{adjustwidth}}%
          \par%
          \only<presentation>{
            \setbeamercolor{local structure}{parent=example text}}%
          \usebeamertemplate{block example begin}}
        {\par%
          \usebeamertemplate{block example end}%
        \end{actionenv}}

    \setbeamercolor{background canvas}{bg=colour-bg}
    \tikzcdset{background color=colour-bg}

    \setbeamercolor{block title}{fg=colour-primary,bg=colour-block-bg}
    \setbeamercolor{block title example}{fg=colour-primary,bg=colour-block-bg}
    \setbeamercolor{block body}{bg=colour-block-bg}
    \setbeamercolor{block body example}{bg=colour-block-bg}

    \setbeamercolor{title}{fg=colour-primary}
    \setbeamercolor{frametitle}{fg=colour-primary}
    \setbeamercolor{alerted text}{fg=colour-primary}

    \setbeamercolor{normal text}{fg=colour-text}
    \setbeamercolor{item}{fg=colour-text}
    \setbeamercolor{section in toc}{fg=colour-text}

    \setbeamercolor{page number in head/foot}{fg=colour-subtle}

    \setbeamercolor*{bibliography entry author}{fg=colour-text}
    \setbeamercolor*{bibliography entry note}{fg=colour-text}

}

\usepackage{comment}
\bibliography{references}

\makeatletter
\def\and{\unskip{ }\@@and{ }\ignorespaces}
\def\author@andify{%
  \nxandlist {\unskip ,\penalty-1 \space\ignorespaces}%
    {\unskip {} \@@and~}%
    {\unskip \penalty-2 \space \@@and~}%
}
\renewcommand{\andify}{%
  \nxandlist{\unskip, }{\unskip{} \@@and~}{\unskip \space \@@and~}}
\makeatother

\newcommand{\ct}{\scr}
\newcommand{\tc}{\scr}
\newcommand{\fc}{\bb}
\newcommand{\sk}{\cl}
\newcommand{\chord}[1]{{#1}^{+}}
\newcommand{\inchord}[1]{{#1}^{-}}

\newcommand{\FCAT}{{\F\h\CAT}}
\newcommand{\twoCAT}{{2\h\CAT}}
\newcommand{\SK}{{\h\b{SK}}}
\newcommand{\VSK}{{\V\SK}}
\newcommand{\FSK}{{\F\SK}}
\newcommand{\twoSK}{{2\SK}}
\newcommand{\Mod}{{\b{Mod}}}
\newcommand{\SET}{{\b{SET}}}

\newcommand{\fcat}[1]{{\b{\fc#1}}}
\newcommand{\FMonCat}[2]{\fcat{MonCat}_{{#1},{#2}}}
\newcommand{\FCartCat}[2]{\fcat{CartCat}_{{#1},{#2}}}
\newcommand{\FDblCat}[2]{\fcat{DblCat}_{{#1},{#2}}}
\newcommand{\FFib}[2]{\fcat{Fib}_{{#1},{#2}}}
\newcommand{\FOpfib}[2]{\fcat{Opfib}_{{#1},{#2}}}
\newcommand{\FDFib}{\fcat{DFib}}
\newcommand{\FDOpfib}{\fcat{DOpfib}}
\newcommand{\FFun}{\fcat{Fun}}
\newcommand{\FMod}{\fcat{Mod}}
\newcommand{\F}{{\scr F}}
\newcommand{\SMod}{{\b{\sk Mod}}}

\renewcommand{\lto}{\leadsto}
\renewcommand{\V}{{\scr V}}
\newcommand{\W}{{\scr W}}
\newcommand{\wb}{{\overline w}}
\newcommand{\wbp}{{\overline{w'}}}

\renewcommand{\ob}{{\tx{ob}}}
\newcommand{\el}{{\tx{el}}}
\newcommand{\U}{{\mathfrak U}}
\DeclareMathOperator*{\mlim}{mlim}
\newcommand{\allweights}{{\tx{all}}}
\newcommand{\FPCat}{{\b{FP}}}
\newcommand{\coeq}{{\tx{coeq}}}
\newcommand{\kerp}{{\tx{ker\;pair}}}

\title[Enhanced 2-categorical structures]{Enhanced 2-categorical structures,\\two-dimensional limit sketches \\and the symmetry of internalisation}
\author{Nathanael Arkor}
\address{Department of Software Science, Tallinn University of Technology, Estonia}
\author{John Bourke}
\address{Department of Mathematics and Statistics, Faculty of Science, Masaryk University, Czech Republic}
\author{Joanna Ko}
\address{Department of Mathematics and Statistics, Faculty of Science, Masaryk University, Czech Republic}
\thanks{The second-named and third-named authors acknowledge the support of the Grant Agency of the Czech Republic under the grant 22-02964S}
\subjclass{18C10, 18C30, 18C40, 18D20, 18M65, 18N10}
\datetoday

\begin{document}

\begin{abstract}
	Many structures of interest in two-dimensional category theory have aspects that are inherently strict. This strictness is not a limitation, but rather plays a fundamental role in the theory of such structures. For instance, a monoidal fibration is -- crucially -- a \emph{strict} monoidal functor, rather than a pseudo or lax monoidal functor. Other examples include monoidal double categories, double fibrations, and intercategories. We provide an explanation for this phenomenon from the perspective of enhanced 2-categories, which are 2-categories having a distinguished subclass of 1-cells representing the strict morphisms. As part of our development, we introduce enhanced 2-categorical limit sketches and explain how this setting addresses shortcomings in the theory of 2-categorical limit sketches. In particular, we establish the symmetry of internalisation for such structures, entailing, for instance, that a monoidal double category is equivalently a pseudomonoid in an enhanced 2-category of double categories, or a pseudocategory in an enhanced 2-category of monoidal categories.
\end{abstract}

\maketitle

\tableofcontents

\section{Introduction}

Category theory and its applications involve a wide range of categorical structures, such as monads, monoidal categories, fibrations, and double categories, to name but a few. These two-dimensional structures may be viewed as residing within the 2-category of categories, but may also be interpreted in other 2-categories to produce useful variations.
For instance, one may consider monads in any 2-category~\cite{street1972formal}, pseudomonoids in a 2-category with finite products~\cite{day1997monoidal}, and fibrations and pseudocategories in a 2-category with sufficient limits~\cite{street1974fibrations,martins2006pseudo}. This permits the unification of many different concepts, and can lead to simplifications of their theory.

Our motivation in this paper is to identify and resolve a limitation of the classical, purely 2-categorical approach to two-dimensional structure, and to show that our solution clarifies several perplexities that have recently been observed in the literature, particularly in connection with the theories of double categories and of fibrations. We shall explain the limitation by means of an example in two parts, each illustrating a different facet of the same problem.

The first issue naturally emerges when considering double categorical structures of recent interest, such as monoidal double categories \cite{shulman2010constructing} and double fibrations \cite{cruttwell2022double}. Let us focus on the case of monoidal double categories. A double category is a pseudocategory\footnotemark{} internal to the 2-category $\Cat$ of categories, functors, and natural transformations~\cite{grandis1999limits}.  Explicitly, it comprises categories $C_0$ and $C_1$, together with source and target functors $s$ and $t$, and functors $c$ and $i$ specifying the action of composition and identities, associative and unital up to coherent isomorphism.\footnotetext{We use \emph{double category} to mean a pseudo double category, rather than a strict double category.}%
\[\begin{tikzcd}
	{C_1 \times_{C_0} C_1} && {C_1} && {C_0}
	\arrow["c"{description}, from=1-1, to=1-3]
	\arrow["{\pi_1}", shift left=3, from=1-1, to=1-3]
	\arrow["{\pi_2}"', shift right=3, from=1-1, to=1-3]
	\arrow["s", shift left=3, from=1-3, to=1-5]
	\arrow["t"', shift right=3, from=1-3, to=1-5]
	\arrow["i"{description}, from=1-5, to=1-3]
\end{tikzcd}\]
From this perspective, it would be natural to hope that a monoidal double category would be a pseudocategory internal to the 2-category $\b{MonCat}_p$ of monoidal categories, pseudo (\aka strong) monoidal functors, and monoidal natural transformations.
This is almost true, but not quite. To capture monoidal double categories, it is necessary to impose the additional requirement that the source and target functors $s$ and $t$ are \emph{strict} monoidal. (This then implies that the pullback projections $\pi_1$ and $\pi_2$ are also strict monoidal functors; in fact, the requisite pullback $C_1 \times_{C_0} C_1$ may not even exist if $s$ and $t$ are not strict monoidal.)

This leads us to replace the diagram above by the following diagram,
\[\begin{tikzcd}
	{C_1 \times_{C_0} C_1} && {C_1} && {C_0}
	\arrow["c"{description}, squiggly, from=1-1, to=1-3]
	\arrow["{\pi_1}", shift left=3, from=1-1, to=1-3]
	\arrow["{\pi_2}"', shift right=3, from=1-1, to=1-3]
	\arrow["s", shift left=3, from=1-3, to=1-5]
	\arrow["t"', shift right=3, from=1-3, to=1-5]
	\arrow["i"{description}, squiggly, from=1-5, to=1-3]
\end{tikzcd}\]
in which the straight arrows ($\to$) indicate the morphisms intended to be \emph{strict}, whereas the squiggly arrows ($\rightsquigarrow$) indicate arbitrary, \emph{weak} morphisms. This is no longer a diagram in a 2-category, since it involves two kinds of morphism.
Instead, it lives in a structure known as an \emph{enhanced 2-category}~\cite{lack2012enhanced}, which is a 2-category with two classes of morphism, called \emph{tight} and \emph{loose}, in which the tight morphisms are a subclass of the loose morphisms.

This example, along with several others that we will discuss subsequently, suggest that we ought to define structures such as pseudocategories not in the context of 2-categories, but rather in the context of enhanced 2-categories. As expected, in this context, a monoidal double category is precisely a pseudocategory internal to an enhanced 2-category $\FMonCat s p$ of monoidal categories, in which the tight and loose morphisms are the strict and pseudo monoidal functors respectively.

The second issue, slightly more subtle, occurs even in more basic settings. There are various ways to make the concept of ``structure on a category'' precise. A traditional approach is through the use of 2-monads~\cite{blackwell1989two}: for instance, there is a 2-monad on $\Cat$ whose (strict) algebras are precisely (non-strict) monoidal categories. A more general approach is through the use of 2-sketches: for instance, whilst there is no 2-monad on $\Cat$ whose algebras (strict or otherwise) are pseudocategories, there is a 2-sketch whose (strict) models are precisely pseudocategories.

One advantage of this abstract approach over simply defining two-dimensional structures by hand is that the definitions of strict morphisms for structures defined by a 2-sketch are obtained for free, namely as the strict morphisms of models. Unfortunately, this is not the case for \emph{weak} morphisms. To illustrate the problem, consider the following diagram for pseudomonoids, which underlies a 2-sketch $\sk M$ -- \ie a 2-category $\tc M$ equipped with chosen cones (in this case, simply the product cones).
\[\begin{tikzcd}
	{M^2} && M && 1
	\arrow["u"{description}, from=1-5, to=1-3]
	\arrow["m"{description}, from=1-1, to=1-3]
	\arrow["{\pi_1}", shift left=3, from=1-1, to=1-3]
	\arrow["{\pi_2}"', shift right=3, from=1-1, to=1-3]
\end{tikzcd}\]
Models of the 2-sketch $\sk M$, which are 2-functors $\tc M \to \Cat$ sending cones in $\sk M$ to limit cones in $\Cat$, are precisely monoidal categories.

Now, suppose we are given a pair of monoidal categories $A$ and $B$, viewed as 2-functors $A, B \colon \tc M \to \Cat$. We would like to capture the notion of a lax monoidal functor $\phi \colon A \to B$ in terms of the 2-sketch $\sk M$. Given that a strict monoidal functor from $A$ to $B$ corresponds precisely to a 2-natural transformation from $A$ to $B$, our first guess might be that a lax monoidal functor should correspond to a lax natural transformation. This is almost true, but not quite. In fact, whilst a lax monoidal functor is lax natural in the unit and multiplication 1-cells $u$ and $m$,
\[
\begin{tikzcd}
	1 & 1 \\
	A & B
	\arrow[Rightarrow, no head, from=1-1, to=1-2]
	\arrow["{u_A}"', from=1-1, to=2-1]
	\arrow["\phi"', from=2-1, to=2-2]
	\arrow["{u_B}", from=1-2, to=2-2]
	\arrow["{\phi_u}"', shorten <=6pt, shorten >=6pt, Rightarrow, from=1-2, to=2-1]
\end{tikzcd}
\hspace{6em}
\begin{tikzcd}
	{A^2} & {B^2} \\
	A & B
	\arrow["\phi"', from=2-1, to=2-2]
	\arrow["{m_A}"', from=1-1, to=2-1]
	\arrow["{\phi^2}", from=1-1, to=1-2]
	\arrow["{m_B}", from=1-2, to=2-2]
	\arrow["{\phi_m}"', shorten <=6pt, shorten >=6pt, Rightarrow, from=1-2, to=2-1]
\end{tikzcd}
\]
it is \emph{strictly} natural in the product projections $\pi_1$ and $\pi_2$.
\[
\begin{tikzcd}
	{A^2} & {B^2} \\
	A & B
	\arrow["{\phi^2}", from=1-1, to=1-2]
	\arrow[""{name=0, anchor=center, inner sep=0}, "{(\pi_1)_A}"', from=1-1, to=2-1]
	\arrow[""{name=1, anchor=center, inner sep=0}, "{(\pi_1)_B}", from=1-2, to=2-2]
	\arrow["\phi"', from=2-1, to=2-2]
	\arrow["{=}"{description}, draw=none, from=0, to=1]
\end{tikzcd}
\hspace{4em}
\begin{tikzcd}
	{A^2} & {B^2} \\
	A & B
	\arrow["{\phi^2}", from=1-1, to=1-2]
	\arrow[""{name=0, anchor=center, inner sep=0}, "{(\pi_2)_A}"', from=1-1, to=2-1]
	\arrow[""{name=1, anchor=center, inner sep=0}, "{(\pi_2)_B}", from=1-2, to=2-2]
	\arrow["\phi"', from=2-1, to=2-2]
	\arrow["{=}"{description}, draw=none, from=0, to=1]
\end{tikzcd}
\]
Thus, it is necessary to treat the product projections $\pi_1$ and $\pi_2$ as special \emph{strict} morphisms. (Note that precisely the same problem occurs with pseudo and colax monoidal functors.) Consequently, we are led to consider the following refinement of the 2-sketch $\sk M$.
\[\begin{tikzcd}
	{M^2} && M && 1
	\arrow["u"{description}, squiggly, from=1-5, to=1-3]
	\arrow["m"{description}, squiggly, from=1-1, to=1-3]
	\arrow["{\pi_1}", shift left=3, from=1-1, to=1-3]
	\arrow["{\pi_2}"', shift right=3, from=1-1, to=1-3]
\end{tikzcd}\]
This is no longer a 2-sketch, but rather an \emph{enhanced 2-sketch}, the appropriate notion of sketch for enhanced 2-categories. A lax morphism of models for an enhanced 2-sketch is a lax natural transformation that restricts on the tight morphisms to a 2-natural transformation, thus resolving the previous mismatch between lax monoidal functors and lax morphisms of models.

The solution to both of the problems illustrated above is therefore to pass from 2-categories to enhanced 2-categories. The main aim of this paper is to introduce and study enhanced 2-categorical structures, both in the form of concrete examples and, more generally, in the form of enhanced limit 2-sketches. For instance, in \cref{F-categorical-structures}, we show that pseudocategories, pseudomonoids, and fibrations may naturally be defined in enhanced 2-categories with sufficient limits. These enhanced 2-categorical structures allow us to capture a wide range of examples that lie slightly outside the traditional 2-categorical setting, including the double categorical structures mentioned above, as well as various examples of a fibrational flavour, such as monoidal fibrations, tangent fibrations, and fibrewise opfibrations~\cite{shulman2008framed,cigoli2020fibered,cockett2018differential}.

One of the merits of enhanced 2-categories is that, like 2-categories, they are simply enriched categories for a certain base of enrichment $\F$. Consequently, many aspects of enhanced 2-category theory arise by appropriate instantiations of enriched category theory.  In particular, we introduce enhanced limit 2-sketches as $\F$-enriched limit sketches, before showing that the previous examples of enhanced 2-categorical structures are the models for respective enhanced 2-sketches.

As an application of the theory of such sketches, we prove a general \emph{symmetry of internalisation} result for the models of an enhanced 2-sketch that obtains, for instance, the fact that monoidal double categories can be presented either as pseudomonoids in an enhanced 2-category of double categories or as pseudocategories in an enhanced 2-category of monoidal categories. In particular, \cref{Mod-skew-symmetry-refined} establishes that, for enhanced limit 2-sketches $\sk S$ and $\sk T$ and a suitably complete enhanced 2-category $\fc C$,  there is an isomorphism of enhanced 2-categories of models
\[\FMod_\wb(\sk S, \FMod_w(\sk T, \fc C)) \iso \FMod_w(\sk T, \FMod_\wb(\sk S, \fc C))\]
where $w \in \{ s, p, l, c \}$ specifies a notion of morphism (strict, pseudo, lax, and colax respectively), and $\wb$ is its dual. In \cref{applications}, we then show that this symmetry recovers many known situations in which a given two-dimensional structure admits several different presentations, as well as producing new examples.

\subsection{Outline of the paper}

We now give a more detailed overview of the paper.

In \cref{notationalconventions} we establish our notational and foundational conventions, before turning in \cref{enrichedcategories} to the necessary background on enriched category theory.
Our focus here is on weighted limits, and we give examples from 2-category theory that will be used later in the paper.

\Cref{enhanced-2-category-theory} concerns enhanced 2-categories, which are categories enriched in $\F$.  We start by recalling the basic aspects of $\F$-categories and $\F$-weighted limits from \cite{lack2012enhanced}, before defining our three motivating examples of enhanced 2-categorical structures -- pseudomonoids, pseudocategories, and fibrations -- and explaining what concepts they capture in a range of enhanced 2-categories.

In \cref{sketches} we begin by recalling the notion of enriched limit sketch~\cite{kelly1982basic}, before specialising to the case of $\F$-enriched sketches. We introduce the enhanced 2-categories of models for an enhanced limit 2-sketch (\ie $\F$-enriched limit sketch) and examine them in each of our examples.

\Cref{multicategory-of-sketches} is the technical heart of the paper. Therein, we construct several multicategory structures on the category of limit $\F$-sketches, parameterised by various flavours of weakness. We show that each of these multicategories is representable and closed, thereby inducing closed monoidal structures on the category of limit $\F$-sketches.
In \cref{multiple-perspectives} we make use of this closed structure to establish symmetry of internalisation in \cref{Mod-skew-symmetry-refined}, before showing how this sheds light on a range of enhanced 2-categorical structures, such as monoidal double categories, double fibrations, and intercategories.

Finally, in \cref{2-sketches}, we relate limit $\F$-sketches to the better known limit 2-sketches. We show that the motivating examples of enhanced limit 2-sketches may be constructed universally from their underlying 2-sketches. We then explain how every algebraic 2-theory~\cite{power1999enriched}, and more generally every flexible-limit 2-sketch, induces an enhanced limit 2-sketch with the same 2-categories of models, establishing that the enhanced 2-categorical perspective strictly generalises the 2-categorical perspective. We conclude in \cref{future-directions} by mentioning possible future directions.

\subsection{Related work}
\label{related-work}
\subsubsection{Two-dimensional limit sketches}

Two-dimensional limit sketches occupy a curious position in the literature. Perhaps the earliest allusion to two-dimensional sketches is the work of \textcite{street1980cosmoi}, who defines a \emph{pseudo model} for a (one-dimensional) limit sketch\footnotemark{} to be a pseudofunctor sending the specified cones to pseudo bilimit cones~\cite[(8.5)]{street1980cosmoi}.
\footnotetext{\citeauthor{street1980cosmoi} calls limit sketches \emph{Gabriel theories}.}%
Enriched limit sketches were introduced in \cite[\S6.3]{kelly1982basic} and further developed in \cite{kelly1982structures}. Whilst the notion of limit 2-sketch is simply that of a $\CAT$-enriched limit sketch, it was already observed at this time that enriched category theory is not sufficient for the study of two-dimensional sketches, due to its inability to adequately capture weakness. However, \citeauthor{kelly1982basic} promised, in these and later works (\cf~\cites{dubuc1983presentation}[Remark~7.5]{bird1989flexible}), that a proper treatment would be forthcoming. Such a treatment never appeared. Later, a definition of PIE-limit 2-theory\footnotemark{} was given by \textcite[\S4]{power1995tricategories}, but morphisms of models were taken to be arbitrary pseudonatural transformations, and hence suffered from the second problem described above. \citeauthor{power1995tricategories} also introduced a bicategorical notion of theory (namely, a small bicategory with finite bilimits~\cite[\S5]{power1995tricategories}), which motivated a coherence theorem for finite bilimit preserving 2-functors~\cite[Theorem~6.7]{power1995tricategories}. More recently, \citeauthor{makkai2010two} considered in unpublished work a 2-categorical notion of limit theory (namely, a small 2-category with finite conical pseudolimits and powers by the interval category)~\cite{makkai2010two}, and \citeauthor{di2022biaccessible} studied a bicategorical notion of limit theory (namely, small 2-categories with finite bilimits)~\cite{di2022biaccessible}; both sets of authors make the same choice of morphisms of models as \textcite{power1995tricategories}.

A refined notion of weak morphism between models was considered by \textcite{lack2007lawvere} in unpublished work on algebraic 2-theories, corresponding exactly to the weak morphisms of algebras for a 2-monad. Their definition, whilst of an enhanced 2-categorical flavour, preceded the introduction of enhanced 2-categories in \cite{lack2012enhanced}. Subsequently, PIE-limit 2-theories were considered in \cite[\S9]{bourke2021accessible}, who extended the refined notion of morphism of models to this setting and suggested a connection with enhanced 2-categories; PIE-limit 2-theories are further generalised by the cloven flexible-limit 2-sketches we introduce in \cref{2-sketches}.
\footnotetext{A \emph{limit theory} is a limit sketch whose underlying category is closed under a class of limits, and whose cones comprise the limit cones.}

\subsubsection{Lax morphisms of one-dimensional sketches}

\label{lax-morphisms-via-2-categories}
A general notion of lax morphism between models of a (one-dimensional) sketch, relative to a 2-category, was introduced by \textcite[69]{bastiani1974multiple}.\footnote{Note that, counterintuitively, this notion is \emph{not} a special case of the notion of lax morphism relative to a double category defined \ibid \cite[68]{bastiani1974multiple}.} Let $\tc C$ be a 2-category and let $\sk S$ be a sketch. One may consider models of $\sk S$ in the category $\dc{Sq}(\tc C)_1$ of 1-cells and lax squares in $\tc C$. For instance, taking $\tc C = \Cat$ and $\sk S$ to be the sketch for categories, a model of $\sk S$ in $\dc{Sq}(\Cat)_1$ comprises a span $(s_C, \phi_s, s_D) \colon {\phi_0 \from \phi_1 \to \phi_0} \cocolon (t_C, \phi_t, t_D)$ together with morphisms $(i_C, \phi_i, i_D) \colon \phi_0 \to \phi_1$ and $(m_C, \phi_m, m_D) \colon \phi_2 \defeq \phi_1 \times_{\phi_0} \phi_1 \to \phi_1$ satisfying unitality and associativity axioms. Explicitly, these comprise natural transformations as follows.
\[
\begin{tikzcd}
	{C_1} & {D_1} \\
	{C_0} & {D_0}
	\arrow["{\phi_1}", from=1-1, to=1-2]
	\arrow["{s_C}"', from=1-1, to=2-1]
	\arrow["{\phi_s}"', shorten <=4pt, shorten >=4pt, Rightarrow, from=1-2, to=2-1]
	\arrow["{s_D}", from=1-2, to=2-2]
	\arrow["{\phi_0}"', from=2-1, to=2-2]
\end{tikzcd}
\hspace{2em}
\begin{tikzcd}
	{C_1} & {D_1} \\
	{C_0} & {D_0}
	\arrow["{\phi_1}", from=1-1, to=1-2]
	\arrow["{t_C}"', from=1-1, to=2-1]
	\arrow["{\phi_t}"', shorten <=4pt, shorten >=4pt, Rightarrow, from=1-2, to=2-1]
	\arrow["{t_D}", from=1-2, to=2-2]
	\arrow["{\phi_0}"', from=2-1, to=2-2]
\end{tikzcd}
\hspace{2em}
\begin{tikzcd}
	{C_0} & {D_0} \\
	{C_1} & {D_1}
	\arrow["{\phi_0}", from=1-1, to=1-2]
	\arrow["{i_C}"', from=1-1, to=2-1]
	\arrow["{\phi_i}"', shorten <=4pt, shorten >=4pt, Rightarrow, from=1-2, to=2-1]
	\arrow["{i_D}", from=1-2, to=2-2]
	\arrow["{\phi_1}"', from=2-1, to=2-2]
\end{tikzcd}
\hspace{2em}
\begin{tikzcd}
	{C_2} & {D_2} \\
	{C_1} & {D_1}
	\arrow["{\phi_2}", from=1-1, to=1-2]
	\arrow["{m_C}"', from=1-1, to=2-1]
	\arrow["{\phi_m}"', shorten <=4pt, shorten >=4pt, Rightarrow, from=1-2, to=2-1]
	\arrow["{m_D}", from=1-2, to=2-2]
	\arrow["{\phi_1}"', from=2-1, to=2-2]
\end{tikzcd}
\]
\textcite{bastiani1974multiple} then require each such natural transformation to be the identity if it is induced by a morphism in $\ct S$ in the image of a diagram over which there exists a cone in $\sk S$: in the example above, these are precisely $\phi_s$ and $\phi_t$ (because the object of composable morphisms in $\sk S$ is a pullback over the source and target morphisms). Consequently, the lax morphisms of \cite{bastiani1974multiple} are a special case of the lax morphisms of models of a (one-dimensional) enhanced sketch in our sense, in which the tight morphisms are precisely those over which there exists a cone. We shall exhibit a precise connection to the work of \cite{bastiani1974multiple} in \cref{(co)free-F-sketches}.

\subsubsection{Enhanced 2-categories versus double categories}

The notion of enhanced 2-category is equivalent to that of a strict double category equipped with a functorial choice of companions in the sense of \cite{grandis2004adjoint}. It is therefore reasonable to imagine that our development might be carried out double categorically, rather than enhanced 2-categorically. However, it should be noted, as observed in \cite[\S1]{lack2012enhanced}, that the notion of double limit introduced in \cite{grandis1999limits} does not suffice to capture weighted limits in enhanced 2-category theory: whilst one does recover the appropriate universal property with respect to tight morphisms, a double limit does not give an appropriate universal property with respect to loose morphisms. It is likely that one could translate the enhanced 2-categorical universal property into the language of double categories with companions, and thereby reformulate our development in this language. However, we shall not explore this direction.

Our work should also be contrasted with the theory of cartesian double theories developed in \cite{lambert2024cartesian,patterson2024products} which concerns structures definable using double categorical finite product theories. In our setting, the two classes of morphisms in each example tend to comprise a strict and weak notion of homomorphism of structure respectively (\eg strict monoidal functors and lax monoidal functors). In contrast, in the double categorical setting, the two classes of morphisms in each example tend to comprise a notion of (strict or weak) homomorphism, together with a notion of bimodule of structure (\eg lax monoidal functors and monoidal distributors).

\subsubsection{Enhanced limit sketches for enhanced categories}
\label{enhanced-categories}

Whilst our interest lies primarily in two-dimensional structure, enhanced categories are also of interest in the one-dimensional setting (\cf the $\b{Subset}$-categories of \cite{power2002premonoidal}). It does not appear that enhanced sketches have been considered even in this setting. Consequently, our results are of interest even to those concerned solely with one-dimensional structures.

\subsection{Acknowledgements}

The authors thank Christina Vasilakopoulou who, at \emph{Category Theory 2023}, posed the question to N.~Arkor of how to relate the different perspectives on structured double categories. This led him to discuss this question with his colleagues J.~Bourke and J.~Ko at Masaryk University who had, serendipitously, already begun studying enhanced limit 2-sketches. Combining their efforts led to the present paper. The authors also thank John Power for providing historical context for the study of two-dimensional limit theories.

\section{Notational conventions}
\label{notationalconventions}

We begin by establishing conventions for notation and size.

\subsection{Notation}

We shall use $A, B$, \etc for objects; $\ct S, \ct T$, \etc for categories, $\V$-enriched categories, and 2-categories; $\fc S$, $\fc T$, \etc for $\F$-categories; and $\sk S, \sk T$, \etc{} for sketches. We use boldface to denote named 2-categories, such as $\Cat$.

\subsection{Size conventions}\label{sect:size}

For the sake of readability, we will not emphasise size concerns in the paper proper, and so explain our approach to size here. Some of our main results, for instance \cref{Mod-skew-symmetry-refined}, concern the structure of various closed multicategories of $\F$-enriched categories and sketches (for a certain monoidal category $\F$ discussed in \cref{enhanced-2-category-theory}). For these multicategories to be closed, it is necessary to restrict to \emph{small} $\F$-categories and sketches, since enriched functor categories $[\fc A, \fc B]$ generally do not exist unless $\fc A$ is small. This means that we require our main examples of $\F$-categories -- such as the $\F$-categories of categories, monoidal categories, double categories, and fibrations (\cref{F-categories}) -- to be small.

This is achieved by positing the existence of a Grothendieck universe $\U$. We denote by $\SET$ the large category of all sets, and by $\Set$ the small category of $\U$-small sets. Then $\CAT$ denotes the large (2-)category of small categories (\ie categories internal to $\SET$), and $\Cat$ denotes the small (2-)category of $\U$-small categories (\ie categories internal to $\Set$). Our aforementioned examples of $\F$-categories are then taken to comprise $\U$-small categories, $\U$-small monoidal categories, and so on.

\begin{remark}
	An alternative approach to size, not requiring the existence of a Grothendieck universe, would be to work only with \emph{partially} closed multicategories of locally small \mbox{$\F$-categories} and locally small $\F$-sketches. Whilst this approach would yield the same applications, it has a number of drawbacks: for instance, in addition to making theorem statements more convoluted (requiring additional assumptions constraining the size of certain categories), tensor products of locally small $\F$-categories (\cref{tensor-product}), and consequently also of locally small $\F$-sketches (\cref{tensor-of-sketches}), do not exist in general, as they fail to remain locally small.
\end{remark}

\section{Background on weighted limits}
\label{enrichedcategories}

In this section we recall the basics of weighted limits in enriched category theory, giving some examples from 2-category theory.
Following \textcite{kelly1982basic}, we take $\V$ to be a complete and cocomplete symmetric closed monoidal category. In this setting, we have access to all of the standard constructions of enriched category theory, such as the construction of enriched functor categories and of free $\V$-categories, and may view $\V$ itself as a $\V$-category. The two bases of enrichment of interest in this paper are $\V = \CAT$ and $\V = \F$, the latter of which will be the focus of \cref{enhanced-2-category-theory}.

A \emph{weight}\footnotemark{} is a $\V$-functor $W \colon \ct J \to \V$ with $\ct J$ small.
\footnotetext{Note that weighted limits are called \emph{indexed limits} by \textcite{kelly1982basic}, and weights are called \emph{indexing types}.}%
Given a $\V$-functor $D \colon \ct J \to \ct C$, a \emph{$W$-weighted cone over $D$} \footnotemark{} comprises an object $X \in \ct C$ together with a $\V$-natural transformation $\gamma \colon W \tto \ct C(X, D{-})$. It is a \emph{limit cone} when, for each object $C \in \ct C$, the induced morphism in $\V$
\[\ct C(C,X) \to [\ct J,\V](W,\ct C(C,D{-}))\]
is invertible, in which case it is said to exhibit $X$ as the \emph{$W$-weighted limit $\{W,D\}$ of $D$}.%
\footnotetext{Called a \emph{$(W, D)$-cylinder} by \textcite{kelly1982basic}.}%

\begin{example}[Conical limits]
	The conical limit of an (unenriched) functor $D \colon \ct J \to \ct C_0$ to the underlying category of a $\V$-category $\ct C$ comprises an (unenriched) cone $\gamma \colon \Delta(X) \tto D$ for which, for each object $C \in \ct C$, the induced morphism $\ct C(C, X) \to \lim \ct C(X, D{-})$ in $\V$ is invertible. This implies that $\gamma$ is a limit cone in $\ct C_0$, but is generally a stronger property: in particular, when we speak of \emph{products} or \emph{pullbacks} in a $\V$-category, we mean so in this stronger sense.

	Conical limits are a simple kind of weighted limit. Denoting by $\ct J^*$ the free $\V$-category on $\ct J$, the conical limit of $D$ is equivalent to the limit of the induced $\V$-functor $\ct J^* \to \ct C$, weighted by the constant weight $\Delta(I) \colon \ct J^* \to \V$, where $I$ is the unit of $\V$~\cite[\S3.8]{kelly1982basic}.
\end{example}

\begin{example}[Powers]
	Another simple class of weighted limits are powers. Given objects $V \in \V$ and $X \in \ct C$, the power $V \pow X$ is an object equipped with a morphism $V \to \ct C(V \pow X, X)$ in $\V$, such that, for each object $C \in \ct C$, the induced morphism $\ct C(C, V \pow X) \to \V(V, \ct C(C, X))$ in $\V$ is invertible.

	Powers are weighted limits whose weights have domain the unit $\V$-category (equivalently the free $\V$-category on the terminal category)~\cite[\S3.7]{kelly1982basic}.
\end{example}

\begin{example}[2-categorical limits]
	\label{2-categorical-limits}
	For $\V$ the cartesian closed category $\CAT$, a \mbox{$\V$-enriched} category is a 2-category. There are many interesting weighted limits (a.k.a. $2$-limits) in 2-category theory, a few of which we describe below.  Before doing so, we make a few general observations specific to the 2-categorical case.

	First, the requirement that the morphism $\tc C(C,X) \to [\tc J, \CAT](W, \tc C(C,D{-}))$ is invertible -- \ie is an isomorphism of categories -- endows each weighted limit with both a one-dimensional universal property (corresponding to the bijectivity of the isomorphism on objects) and a two-dimensional universal property (corresponding to the bijectivity of the isomorphism on morphisms).

	Second, a weighted cone $\gamma \colon W \tto \tc C(X, D{-})$ induces, for each $J \in \tc J$ and $Y \in W(J)$, a 1-cell $\gamma_{J, Y} \colon X \to D(J)$ in $\tc C$ which we refer to as a \emph{cone projection}.

	Conical limits that will feature in the examples of $\CAT$-enriched sketches that we give later include products and pullbacks. An important example of a non-conical weighted limit is the comma object, whose weight is described in \cite[Example 3), p.\ 167]{street1974elementary} and which we describe in elementary terms. Given a cospan $f \colon X \to Y \from Z \cocolon g$, the comma object $f \comma g$ comes equipped with a weighted cone shaped as below left.
	\[
	\begin{tikzcd}
		{f \comma g} & X \\
		Z & Y
		\arrow["{{\pi_1}}", from=1-1, to=1-2]
		\arrow["{{\pi_2}}"', from=1-1, to=2-1]
		\arrow["\varpi"', shorten <=9pt, shorten >=9pt, Rightarrow, from=1-2, to=2-1]
		\arrow["f", from=1-2, to=2-2]
		\arrow["g"', from=2-1, to=2-2]
	\end{tikzcd}
	\hspace{2cm}
	\begin{tikzcd}
		C & X \\
		Z & Y
		\arrow["f", from=1-2, to=2-2]
		\arrow["g"', from=2-1, to=2-2]
		\arrow["{p_1}", from=1-1, to=1-2]
		\arrow["{p_2}"', from=1-1, to=2-1]
		\arrow["\theta"', shorten <=6pt, shorten >=6pt, Rightarrow, from=1-2, to=2-1]
	\end{tikzcd}\]
	The 1-dimensional universal property says that, given a weighted cone as above right,
	there exists a unique 1-cell $\tp{p_1, \theta, p_2} \colon C \to f \comma g$ such that $\pi_1 \tp{p_1, \theta, p_2} = p_1$, $\pi_2 \tp{p_1, \theta, p_2} = p_2$, and $\varpi \tp{p_1, \theta, p_2} = \theta$. Its 2-dimensional universal property says that, given a further weighted cone as below left, and 2-cells $\gamma_1$ and $\gamma_2$ satisfying the equation below right,
	\[\begin{tikzcd}
		C & X \\
		Z & Y
		\arrow["{q_1}", from=1-1, to=1-2]
		\arrow["{q_2}"', from=1-1, to=2-1]
		\arrow["g"', from=2-1, to=2-2]
		\arrow["f", from=1-2, to=2-2]
		\arrow["\phi"', shorten <=6pt, shorten >=6pt, Rightarrow, from=1-2, to=2-1]
	\end{tikzcd}
	\hspace{1.5cm}
	\begin{tikzcd}
		C & X \\
		Z & Y
		\arrow["{p_1}", from=1-1, to=1-2]
		\arrow[""{name=0, anchor=center, inner sep=0}, "{p_2}"{description}, from=1-1, to=2-1]
		\arrow[""{name=1, anchor=center, inner sep=0}, "{q_2}"', curve={height=24pt}, from=1-1, to=2-1]
		\arrow["\theta"', shorten <=6pt, shorten >=6pt, Rightarrow, from=1-2, to=2-1]
		\arrow["f", from=1-2, to=2-2]
		\arrow["g"', from=2-1, to=2-2]
		\arrow["{\gamma_2}"', shorten <=5pt, shorten >=5pt, Rightarrow, from=0, to=1]
	\end{tikzcd}
	\quad = \quad
	\begin{tikzcd}
		C & X \\
		Z & Y
		\arrow[""{name=0, anchor=center, inner sep=0}, "{p_1}", curve={height=-18pt}, from=1-1, to=1-2]
		\arrow[""{name=1, anchor=center, inner sep=0}, "{q_1}"{description}, from=1-1, to=1-2]
		\arrow["{q_2}"', from=1-1, to=2-1]
		\arrow["\phi"', shorten <=6pt, shorten >=6pt, Rightarrow, from=1-2, to=2-1]
		\arrow["f", from=1-2, to=2-2]
		\arrow["g"', from=2-1, to=2-2]
		\arrow["{\gamma_1}", shorten <=2pt, shorten >=2pt, Rightarrow, from=0, to=1]
	\end{tikzcd}
	\]
	there exists a unique 2-cell
	\[\begin{tikzcd}
		C && {f \comma g}
		\arrow[""{name=0, anchor=center, inner sep=0}, "{\tp{p_1, \theta, p_2}}", curve={height=-24pt}, from=1-1, to=1-3]
		\arrow[""{name=0p, anchor=center, inner sep=0}, phantom, from=1-1, to=1-3, start anchor=center, end anchor=center, curve={height=-24pt}]
		\arrow[""{name=1, anchor=center, inner sep=0}, "{\tp{q_1, \phi, q_2}}"', curve={height=24pt}, from=1-1, to=1-3]
		\arrow[""{name=1p, anchor=center, inner sep=0}, phantom, from=1-1, to=1-3, start anchor=center, end anchor=center, curve={height=24pt}]
		\arrow["{\tp{\gamma_1, \gamma_2}}"{description}, shorten <=6pt, shorten >=6pt, Rightarrow, from=0p, to=1p]
	\end{tikzcd}\]
	such that $\pi_1 \tp{\gamma_1, \gamma_2} = \gamma_1$ and $\pi_2 \tp{\gamma_1, \gamma_2} = \gamma_2$.

	Special cases of comma objects are the lax limit $f \comma Y \defeq f \comma 1_Y$ of a 1-cell $f \colon X \to Y$ and the colax (or oplax) limit $Y \comma g \defeq 1_Y \comma g$ of a 1-cell $g \colon Z \to Y$. Being special cases, these come equipped with universal weighted cones of the shapes depicted below,
	\[
	\begin{tikzcd}
		{f \comma Y} & X \\
		& Y
		\arrow["{\pi_1}", from=1-1, to=1-2]
		\arrow[""{name=0, anchor=center, inner sep=0}, "{\pi_2}"', from=1-1, to=2-2]
		\arrow["f", from=1-2, to=2-2]
		\arrow["\varpi"', shorten >=3pt, Rightarrow, from=1-2, to=0]
	\end{tikzcd}
	\hspace{2cm}
	\begin{tikzcd}
		{Y \comma g} \\
		Z & Y
		\arrow["{\pi_2}"', from=1-1, to=2-1]
		\arrow[""{name=0, anchor=center, inner sep=0}, "{\pi_1}", from=1-1, to=2-2]
		\arrow["g"', from=2-1, to=2-2]
		\arrow["\varpi"', shorten <=3pt, Rightarrow, from=0, to=2-1]
	\end{tikzcd}
	\]
	though it is often useful to view them not merely as special cases of comma objects, but as separate kinds of weighted limits, since many 2-categories admit either lax or colax morphisms of 1-cells, but not both, and, in particular, not general comma objects (\cf~\cite{lack2005limits}).

	Another important special case is the power $\b 2 \pow X$ of an object $X$ by the interval category $\b 2 \defeq \{ 0 \to 1 \}$, which is equivalently the comma object $1_X \comma 1_X$. Its universal weighted cone is specified simply by a 2-cell as below.
	\[\begin{tikzcd}
		{\b 2 \pow X} && X
		\arrow[""{name=0, anchor=center, inner sep=0}, "{\pi_1}", curve={height=-18pt}, from=1-1, to=1-3]
		\arrow[""{name=0p, anchor=center, inner sep=0}, phantom, from=1-1, to=1-3, start anchor=center, end anchor=center, curve={height=-18pt}]
		\arrow[""{name=1, anchor=center, inner sep=0}, "{\pi_2}"', curve={height=18pt}, from=1-1, to=1-3]
		\arrow[""{name=1p, anchor=center, inner sep=0}, phantom, from=1-1, to=1-3, start anchor=center, end anchor=center, curve={height=18pt}]
		\arrow["\varpi"', shorten <=7pt, shorten >=7pt, Rightarrow, from=0p, to=1p]
	\end{tikzcd}\qedshift\]
\end{example}

\section{Enhanced 2-categories and enhanced 2-categorical structures}
\label{enhanced-2-category-theory}

Enhanced 2-categories, \aka $\F$-categories, were introduced by \citeauthor{lack2012enhanced} in \cite{lack2012enhanced}. In this section we describe the basics of $\F$-enriched category theory and recall a characterisation of $\F$-weighted limits. We then describe a range of structures definable in the enhanced 2-categorical context.

\subsection{\texorpdfstring{$\F$}{F}-categories}
\label{F-categories}

\begin{definition}[{\cite[\S3.1]{lack2012enhanced}}]
	$\F$ is the full subcategory of the arrow category $\CAT^\to$ spanned by the full embeddings, \ie the \ff{} and injective-on-object functors.
\end{definition}

In other words, an object of $\F$ is a full embedding
\begin{equation*}
	A \colon A_\tau \ffto A_\lambda
\end{equation*}
and a morphism $f \colon A \to B$ in $\F$ comprises a pair of functors $f_\tau \colon A_\tau \to B_\tau$ and ${f_\lambda \colon A_\lambda \to B_\lambda}$ rendering the following square commutative.
\begin{equation}
	\label{tightness-preserving-functor}
	\begin{tikzcd}
		{A_\tau} & {A_\lambda} \\
		{B_\tau} & {B_\lambda}
		\arrow["A", hook, from=1-1, to=1-2]
		\arrow["{f_\tau}"', from=1-1, to=2-1]
		\arrow["{f_\lambda}", from=1-2, to=2-2]
		\arrow["B"', hook, from=2-1, to=2-2]
	\end{tikzcd}
\end{equation}
We call $A_\tau$ the \emph{tight} part of $A$, and $A_\lambda$ the \emph{loose} part of $A$, and similarly for $f$.

It is established in \cite[\S3.1]{lack2012enhanced} that $\F$ is cartesian closed, complete and cocomplete and so permits the use of the general theory of enriched categories recalled in the previous section.

An \emph{enhanced 2-category} is an $\F$-category, \ie a category enriched in $\F$. Explicitly, an enhanced 2-category $\fc A$ comprises an \ioo{}, faithful, and locally \ff{} 2-functor $\fc A \colon \fc A_\tau \to \fc A_\lambda$. We may view this data as comprising a 2-category $\fc A_\lambda$ whose 1-cells we call \emph{loose morphisms}, together with a wide and locally full sub-2-category $\fc A_\tau$ whose 1-cells we call \emph{tight morphisms}. It is this latter perspective that motivates the terminology ``enhanced 2-category''.

Under this interpretation, an \emph{enhanced 2-functor} (that is, an $\F$-functor) $F \colon \fc A \to \fc B$ is a $2$-functor $F_{\lambda} \colon \fc A_{\lambda} \to \fc B_{\lambda}$ that preserves tightness; similarly, an \emph{enhanced 2-natural transformation} (that is, an $\F$-natural transformation) is a 2-natural transformation between the 2-categories of loose morphisms, whose components are tight. Henceforth, we shall typically use the prefix \emph{$\F$-} rather than \emph{enhanced 2-} for conciseness.

\begin{notation}
	We shall write $A \to B$ for a tight morphism from $A$ to $B$; and $A \leadsto B$ for a loose morphism from $A$ to $B$.
\end{notation}

\begin{example}[2-categories]
	\label{chord-and-inchord}
	Every 2-category $\tc K$ gives rise to two $\F$-categories: a \emph{chordate} $\F$-category $\chord{\tc K}$, in which every morphism is tight; and an \emph{inchordate} $\F$-category $\inchord{\tc K}$, in which only the identities are tight.  These assignments extend to the right and left adjoints of an adjoint triple
	\[\begin{tikzcd}
		{\FCAT_0} && {\twoCAT_0}
		\arrow[""{name=0, anchor=center, inner sep=0}, "{(-)_\lambda}"{description}, from=1-1, to=1-3]
		\arrow[""{name=1, anchor=center, inner sep=0}, "{\inchord\ph}"', shift right=3, curve={height=6pt}, from=1-3, to=1-1]
		\arrow[""{name=2, anchor=center, inner sep=0}, "{\chord\ph}", shift left=3, curve={height=-6pt}, from=1-3, to=1-1]
		\arrow["\dashv"{anchor=center, rotate=-90}, draw=none, from=0, to=2]
		\arrow["\dashv"{anchor=center, rotate=-90}, draw=none, from=1, to=0]
	\end{tikzcd}\]
	where $\ph_\lambda \colon \FCAT_0 \to \twoCAT_0$ is the functor sending an $\F$-category to its underlying 2-category of loose morphisms~\cite[Example~3.2]{lack2012enhanced}. Furthermore, in contrast to the leftmost adjunction, the rightmost adjunction extends to a 2-adjunction between $\FCAT$ and $\twoCAT$.
\end{example}

\begin{example}[$\bbn 1$]
	The terminal $\F$-category, denoted $\bbn 1$, is given by $\chord{\b 1} \iso \inchord{\b 1}$, where $\b 1$ is the terminal 2-category. It has a single object, a single morphism (which is necessarily tight), and a single 2-cell.
\end{example}

\begin{example}[$\fc F$]
	\label{eg:F}
	Since $\F$ is closed, it is itself canonically equipped with the structure of an $\F$-category, which we denote by $\fc F$. Explicitly, $\fc F$ is the $\F$-category whose objects are the full embeddings $A_\tau \hookrightarrow A_\lambda$; whose loose morphisms are functors $A_\lambda \to B_\lambda$; and whose 2-cells are natural transformations therebetween. The tight morphisms are the morphisms of $\F$ described in \eqref{tightness-preserving-functor}.
\end{example}

We will be interested in several examples of $\F$-categories whose objects are structured categories, whose tight morphisms preserve the structure strictly, and whose loose morphisms preserve the structure weakly (\eg up to isomorphism, or (co)laxly). We introduce the following notation accordingly.

\begin{notation}
	Define $\W \defeq \{ s, p, l, c \}$ (standing for \emph{strict}, \emph{pseudo}, \emph{lax}, and \emph{colax} respectively) to be the set of \emph{weaknesses}. By convention, an \emph{$l$-cell} is a 2-cell of the following shape:
	\[\begin{tikzcd}
		\cdot & \cdot \\
		\cdot & \cdot
		\arrow[from=1-1, to=1-2]
		\arrow[from=1-1, to=2-1]
		\arrow[from=1-2, to=2-2]
		\arrow[from=2-1, to=2-2]
		\arrow[shorten <=6pt, shorten >=6pt, Rightarrow, from=1-2, to=2-1]
	\end{tikzcd}\]
	Conversely, a \emph{$c$-cell} is a 2-cell of the following shape:
	\[\begin{tikzcd}
		\cdot & \cdot \\
		\cdot & \cdot
		\arrow[from=1-1, to=1-2]
		\arrow[from=1-1, to=2-1]
		\arrow[from=1-2, to=2-2]
		\arrow[from=2-1, to=2-2]
		\arrow[shorten <=6pt, shorten >=6pt, Rightarrow, from=2-1, to=1-2]
	\end{tikzcd}\]
	A \emph{$p$-cell} is an invertible $l$-cell. A \emph{$s$-cell} is an identity 2-cell. Define an involution:
	\[\overline\ph \defeq \{ s \mapsto s,
	p \mapsto p, l \mapsto c, c \mapsto l \} \colon \W \to \W\]
	We shall occasionally use that $\W$ is equipped with a partial order in which $s \leq p$, $p \leq l$, $p \leq c$.
\end{notation}

Note that we could have alternatively defined a $p$-cell to be an invertible $c$-cell, or defined two variants of invertible 2-cell depending on the primary direction. However, there is a canonical bijection between the two notions, so we shall simply take care when considering $\overline p$ to take inverses where appropriate.

\begin{example}[Categories]
	\label{ex:Cat+}
	A basic but central example is the (small) chordate $\F$-category $\chord{\Cat}$ of categories, functors and natural transformations. (We remind the reader that, according to our size conventions discussed in \cref{sect:size}, the objects of $\chord{\Cat}$ are \mbox{\emph{$\U$-small}} categories, and that the same is true of the categorical structures in the following examples.)
\end{example}

\begin{example}[Monoidal categories]
	\label{ex:FMonCat}
	For each pair $w' \leq w \in \W$ of weaknesses, there is an $\F$-category $\FMonCat {w'} w$ whose objects are monoidal categories, whose loose morphisms are $w$-weak monoidal functors, in which a morphism is tight if it is $w'$-weak, and whose 2-cells are monoidal natural transformations. For instance, when $w' = s$ and $w = p$, $\FMonCat s p$ is the $\F$-category whose loose morphisms are pseudo monoidal functors and whose tight morphisms are strict monoidal functors.

	By restricting to cartesian monoidal categories, we obtain $\F$-categories $\FCartCat {w'} w$.
\end{example}

\begin{example}[Double categories]
	\label{ex:FDblCat}
	For each pair $w' \leq w \in \W$ of weaknesses, there is an $\F$-category $\FDblCat {w'} w$ whose objects are pseudo double categories, whose loose morphisms are $w$-weak double functors, in which a morphism is tight if it is $w'$-weak, and whose 2-cells are natural transformations \cite[\S1.4]{grandis1999limits} (not to be confused with internal natural transformations in $\Cat$).
\end{example}

\begin{example}[Fibrations]
	\label{ex:FFib}
	For each pair $w' \leq w \in \W$ of weaknesses, there is an \mbox{$\F$-category} $\FFib {w'} w$ whose objects are (cloven) fibrations over arbitrary bases, whose loose morphisms are $w$-morphisms of fibrations and whose tight morphisms are $w'$-morphisms. Here, the pseudo morphisms of fibrations are the usual morphisms of fibrations: that is, commutative squares
	\[\begin{tikzcd}
		E & {E'} \\
		B & {B'}
		\arrow["{p'}", from=1-2, to=2-2]
		\arrow["e", from=1-1, to=1-2]
		\arrow["p"', from=1-1, to=2-1]
		\arrow["b"', from=2-1, to=2-2]
	\end{tikzcd}\]
	for which the top functor preserves cartesian morphisms. Such a pseudo morphism is \emph{strict} when the top functor preserves the cleavage. The \emph{lax} morphisms of fibrations coincide with the pseudo morphisms, whilst a \emph{colax} morphism of fibrations is simply a commutative square satisfying no additional compatibility condition.%
	\footnote{The reason for this choice of the lax and colax morphisms corresponds to the fact that fibrations are pseudo algebras for a colax-idempotent 2-monad~\cite{street1974fibrations}. For such a 2-monad, the lax and pseudo morphisms coincide.}
	In each $\F$-category $\FFib {w'} w$, the 2-cells are pairs of natural transformations $e \tto e'$ and $b \tto b'$ satisfying the evident compatibility condition. Dually, there are analogously defined $\F$-categories $\FOpfib {w'} w$ of opfibrations. Restricting to the discrete fibrations and discrete opfibrations, we respectively obtain full sub-$\F$-categories $\FDFib$ and $\FDOpfib$ -- note that each $w$-morphism of (op)fibrations between discrete (op)fibrations is strict, so these $\F$-categories are chordate.

	Each of the above $\F$-categories admits a sub-$\F$-category obtained by fixing the base category $B$. In particular, the $\F$-category $\FFib {w'} w(B)$ has objects the fibrations over $B$, morphisms (both tight and loose) those in $\FFib {w'} w$ having $b = 1_B$ and 2-cells those of $\FFib {w'} w$ having $b \Rightarrow b'$ the identity. Applying the same restrictions in the other examples yields $\F$-categories $\FOpfib {w'} w(B)$, $\FDFib(B)$, and $\FDOpfib(B)$.
\end{example}

\subsection{Enhanced 2-categorical limits}
\label{F-limits}

We now turn to weighted limits in $\F$-categories. We begin by looking at the simple case of tight limits, since these are the only limits that are required in our main examples.

\begin{example}[Tight limits]
	\label{tight-limits}
	Tight limits are simply 2-limits of diagrams of tight morphisms which also satisfy a universal property with respect to the loose morphisms. To be precise, let $\fc C$ be an $\F$-category, let $W \colon \tc J \to \CAT$ be a $\CAT$-weight, and let $D \colon \tc J \to \fc C_\tau$ be a 2-functor. We say that a 2-limit $\{W,D\}$ is a \emph{tight limit} if it is preserved by the inclusion 2-functor $\fc C_\tau \to \fc C_\lambda$. These are precisely $\chord W$-weighted limits where $\chord W \colon \chord{\tc J} \to \fc F$ is the $\F$-weight whose value at $J \in \fc J$ is the identity on $W(J)$.

	Particular cases of interest for us will be tight products, tight pullbacks, and tight comma objects. For example, the case of tight pullbacks is depicted below.
	\[\begin{tikzcd}
		C \\
		& {X \times_Y Z} & X \\
		& Z & Y
		\arrow["{\tp{p_1, p_2}}"{description, pos=0.4}, squiggly, from=1-1, to=2-2]
		\arrow["{p_1}", curve={height=-12pt}, squiggly, from=1-1, to=2-3]
		\arrow["{p_2}"', curve={height=12pt}, squiggly, from=1-1, to=3-2]
		\arrow["{\pi_1}", from=2-2, to=2-3]
		\arrow["{\pi_2}"', from=2-2, to=3-2]
		\arrow["f", from=2-3, to=3-3]
		\arrow["g"', from=3-2, to=3-3]
	\end{tikzcd}\]
	This is a commutative square of tight morphisms, satisfying the usual universal property of the 2-categorical pullback in $\fc C_\lambda$, and such that the induced map $\tp{p_1, p_2}$ to the pullback is tight if and only if both $p_1$ and $p_2$ are tight.
\end{example}

\begin{remark}[Enhanced 2-limits as 2-limits with properties]
	\label{enhanced-2-limits-as-2-limits}
	Tight limits are simply 2-limits satisfying two additional properties. The same is true of general $\F$-weighted limits, as we now explain (though we note that such limits will not be required until \cref{2-sketches}).

	First, observe that an $\F$-weight $W \colon \fc J \to \fc F$ has an underlying $\CAT$-weight ${W_\lambda \colon \fc J_{\lambda} \to \CAT}$, which sends $J \in \fc J$ to the loose part $W(J)_\lambda$ of the full embedding $W(J)_\tau \to W(J)_\lambda$. Furthermore, each $W$-weighted cone $\gamma \colon W \tto \fc J(X, D{-})$ has an underlying $W_{\lambda}$-weighted cone $\gamma_{\lambda} \colon W_{\lambda} \to \fc J_{\lambda}(X, D_{\lambda}-)$ with the property that its cone projections are tight at those $Y \in W(J)_{\tau} \subseteq W(J)_{\lambda}$; $W$-weighted cones amount to $W_{\lambda}$-weighted cones with this property.

	Accordingly, \cite[Proposition~3.6]{lack2012enhanced} establishes that the $W$-weighted limit of an \mbox{$\F$-functor} $D \colon \fc J \to \fc C$ is simply the $W_\lambda$-weighted 2-limit of $D_\lambda \colon \fc J_{\lambda} \to \fc C_{\lambda}$ with the additional properties that
	\begin{enumerate}
		\item for each $J \in \fc J$ and $Y \in W_\tau(J)$, the cone projection $\gamma_{J, Y} \colon \{ W_\lambda, D_\lambda \} \lto DJ$ is tight;
		\item a loose morphism $f \colon A \lto \{ W_\lambda, D_\lambda \}$ is tight if and only if $\gamma_{J, Y} \circ f \colon A \lto DJ$ is tight for each $Y \in W_\tau(J)$.
	\end{enumerate}
	In this case, one says that the cone projections $\gamma_{J, Y}$ \emph{jointly detect tightness.}
\end{remark}

\subsection{Enhanced 2-categorical structures}
\label{F-categorical-structures}

Many kinds of categorical structures -- for instance, pseudomonoids, pseudocategories, and fibrations -- can be defined internally to a sufficiently complete 2-category $\tc C$, with the classical notions recovered when $\tc C = \CAT$. In this section, we \emph{enhance} these structures to $\F$-categorical structures. These $\F$-categorical definitions subsume the 2-categorical ones, whilst also capturing many structures that lie outside the scope of 2-category theory. In \cref{F-sketches}, we will see how each of these structures is captured by the notion of an \emph{enhanced 2-sketch}, which will also automatically provide a notion of morphism for such $\F$-categorical structures.

\begin{definition}[Pseudomonoids]
	\label{skew-monoid}
	A \emph{pseudomonoid} in an $\F$-category $\fc C$ is a pseudomonoid in the 2-category $\fc C_\lambda$ in the sense of \textcite[\S3]{day1997monoidal} for which the specified products are tight products. Explicitly, a pseudomonoid comprises an object $M$ equipped with loose morphisms $\otimes \colon M^2 \lto M$ and $I \colon 1 \lto M$, and invertible 2-cells
	\[\begin{tikzcd}
		{M^3} & {M^2} \\
		{M^2} & M
		\arrow["{\otimes M}", squiggly, from=1-1, to=1-2]
		\arrow["{M\otimes}"', squiggly, from=1-1, to=2-1]
		\arrow["\alpha"{description}, shorten <=6pt, shorten >=6pt, Rightarrow, from=1-2, to=2-1]
		\arrow["\otimes", squiggly, from=1-2, to=2-2]
		\arrow["\otimes"', squiggly, from=2-1, to=2-2]
	\end{tikzcd}
	\hspace{4em}
	\begin{tikzcd}
		M & {M^2} & M \\
		& M
		\arrow["IM", squiggly, from=1-1, to=1-2]
		\arrow[""{name=0, anchor=center, inner sep=0}, "M"', from=1-1, to=2-2]
		\arrow[""{name=0p, anchor=center, inner sep=0}, phantom, from=1-1, to=2-2, start anchor=center, end anchor=center]
		\arrow[""{name=1, anchor=center, inner sep=0}, "\otimes"{description}, squiggly, from=1-2, to=2-2]
		\arrow[""{name=1p, anchor=center, inner sep=0}, phantom, from=1-2, to=2-2, start anchor=center, end anchor=center]
		\arrow[""{name=1p, anchor=center, inner sep=0}, phantom, from=1-2, to=2-2, start anchor=center, end anchor=center]
		\arrow["MI"', squiggly, from=1-3, to=1-2]
		\arrow[""{name=2, anchor=center, inner sep=0}, "M", from=1-3, to=2-2]
		\arrow[""{name=2p, anchor=center, inner sep=0}, phantom, from=1-3, to=2-2, start anchor=center, end anchor=center]
		\arrow["\lambda"', shift right=2, shorten <=5pt, shorten >=5pt, Rightarrow, from=1p, to=0p]
		\arrow["\rho"', shift right=2, shorten <=5pt, shorten >=5pt, Rightarrow, from=2p, to=1p]
	\end{tikzcd}\]
	satisfying the following two equations.
	\begin{equation}\label{eq:pseudomon1}
	\begin{tikzcd}[column sep=small]
		& {M^3} && {M^2} \\
		{M^4} && {M^2} && M \\
		& {M^3} && {M^2}
		\arrow["mM", squiggly, from=1-2, to=1-4]
		\arrow["Mm"{description}, squiggly, from=1-2, to=2-3]
		\arrow["{=}"{description}, draw=none, from=1-2, to=3-2]
		\arrow["\alpha", shorten <=6pt, shorten >=6pt, Rightarrow, from=1-4, to=2-3]
		\arrow["m", squiggly, from=1-4, to=2-5]
		\arrow["{mM^2}", squiggly, from=2-1, to=1-2]
		\arrow["{M^2m}"', squiggly, from=2-1, to=3-2]
		\arrow["m"{description}, squiggly, from=2-3, to=2-5]
		\arrow["\alpha", shorten <=6pt, shorten >=6pt, Rightarrow, from=2-3, to=3-4]
		\arrow["mM"{description}, squiggly, from=3-2, to=2-3]
		\arrow["Mm"', squiggly, from=3-2, to=3-4]
		\arrow["m"', squiggly, from=3-4, to=2-5]
	\end{tikzcd}
	\quad = \quad
	\begin{tikzcd}[column sep=small]
		& {M^3} && {M^2} \\
		{M^4} && {M^3} && M \\
		& {M^3} && {M^2}
		\arrow["mM", squiggly, from=1-2, to=1-4]
		\arrow["{\alpha M}", shorten <=6pt, shorten >=6pt, Rightarrow, from=1-2, to=2-3]
		\arrow["m", squiggly, from=1-4, to=2-5]
		\arrow["\alpha", shorten <=16pt, shorten >=16pt, Rightarrow, from=1-4, to=3-4]
		\arrow["{mM^2}", squiggly, from=2-1, to=1-2]
		\arrow["MmM"{description}, squiggly, from=2-1, to=2-3]
		\arrow["{M^2m}"', squiggly, from=2-1, to=3-2]
		\arrow["mM"{description}, squiggly, from=2-3, to=1-4]
		\arrow["{\alpha M}", shorten <=6pt, shorten >=6pt, Rightarrow, from=2-3, to=3-2]
		\arrow["Mm"{description}, squiggly, from=2-3, to=3-4]
		\arrow["Mm"', squiggly, from=3-2, to=3-4]
		\arrow["m"', squiggly, from=3-4, to=2-5]
	\end{tikzcd}
	\end{equation}
	\begin{equation}\label{eq:pseudomon2}
	\begin{tikzcd}
		&&& {M^2} \\
		{M^2} && {M^3} && M \\
		&&& {M^2}
		\arrow["m", squiggly, from=1-4, to=2-5]
		\arrow["\alpha", shorten <=13pt, shorten >=13pt, Rightarrow, from=1-4, to=3-4]
		\arrow[""{name=0, anchor=center, inner sep=0}, "{1_{M^2}}", curve={height=-24pt}, from=2-1, to=1-4]
		\arrow["MiM"{description}, squiggly, from=2-1, to=2-3]
		\arrow[""{name=1, anchor=center, inner sep=0}, "{1_{M^2}}"', curve={height=24pt}, from=2-1, to=3-4]
		\arrow["mM"{description}, squiggly, from=2-3, to=1-4]
		\arrow["Mm"{description}, squiggly, from=2-3, to=3-4]
		\arrow["m"', squiggly, from=3-4, to=2-5]
		\arrow["{\rho M}"', shorten <=8pt, shorten >=8pt, Rightarrow, from=0, to=2-3]
		\arrow["{M\lambda}"', shorten <=8pt, shorten >=8pt, Rightarrow, from=2-3, to=1]
	\end{tikzcd}
	\quad = \quad 1_m
	\end{equation}
	A pseudomonoid is \emph{cartesian} if ${\otimes} \colon M \times M \lto M$ is right adjoint to the diagonal and $I \colon 1 \lto M$ is right adjoint to the unique morphism to the terminal object.
\end{definition}

\begin{example}
	\label{pseudomonoid-examples}
	Specialising the definition of pseudomonoid (\cref{skew-monoid}) to different $\F$-categories recovers various concepts of interest in the literature.
	\begin{center}
	\begin{tabu}{ccc}
		$\F$-category & Pseudomonoid & Reference \\
		\hline
		$\chord\Cat$ &  Monoidal cat. & \cite[\S3]{day1997monoidal} \\
		$\FMonCat s p$ & Braided monoidal cat. & \cite[\S3]{joyal1986braided} \\
		$\FMonCat s l$ & Duoidal cat. & \cite[Definition~6.1]{aguiar2010monoidal} \\
		$\FMonCat s c$ & Duoidal cat. & \cite[Proposition~6.73]{aguiar2010monoidal} \\
		$\FDblCat s c$ & Colax double cat. & \cites[\S5.5]{grandis2004adjoint}[Definition~4.1]{gambino2024monoidal} \\
		$\FFib s p$ & Monoidal fibration & \cites[Definition~12.1]{shulman2008framed}[Proposition~3.1]{moeller2020monoidal}
	\end{tabu}
	\end{center}
	Similarly, specialising the definition of cartesian pseudomonoid recovers other concepts of interest.
	\begin{center}
	\begin{tabu}{ccc}
		$\F$-category & Cartesian pseudomonoid & Reference \\
		\hline
		$\FDblCat s l$ & Precartesian double cat. & \cite[Definition~4.1.1]{aleiferi2018cartesian} \\
		$\FDblCat s p$ & Cartesian double cat. & \cite[Definition~4.2.1]{aleiferi2018cartesian}
	\end{tabu}
	\end{center}
\end{example}

\begin{remark}
	\label{skew-examples}
	There are many variations on the notion of pseudomonoid that one may consider. For instance, a \emph{left-skew monoid} in an $\F$-category is defined in the same fashion as a pseudomonoid except that the 2-cells $\alpha$, $\lambda$ and $\rho$ are not required to be invertible, and there are three additional equations; in a \emph{right-skew monoid}, these 2-cells have their orientations reversed (see \cites{szlachanyi2012skew}[\S4]{lack2012skew}). For instance, right-skew monoids in the \mbox{$\F$-category} $\FFib p {c}(B)$ are closely related to the lax monoidal fibrations of \textcite{zawadowski2009lax} (though the former are subject to the three additional equations aforementioned, whilst the latter satisfies a further normality condition on the unit).
\end{remark}

\begin{definition}[Pseudocategories]
	\label{pseudocategory}
	A \emph{pseudocategory} in an $\F$-category $\fc C$ is a pseudocategory in the 2-category $\fc C_\lambda$ in the sense of \textcite[\S1]{martins2006pseudo} for which the source and target morphisms are tight, and for which the specified pullbacks are tight pullbacks.

	To begin with, a pseudocategory comprises objects $C_0$ and $C_1$ equipped with tight morphisms $s, t \colon C_1 \to C_0$.  We require that the tight pullback $C_2 \defeq C_1 \times_{C_0} C_1$ depicted below exists,
	\[\begin{tikzcd}
		{C_2} & {C_1} \\
		{C_1} & {C_0}
		\arrow["{\pi^1_1}", from=1-1, to=1-2]
		\arrow["{\pi^1_2}"', from=1-1, to=2-1]
		\arrow["\lrcorner"{anchor=center, pos=0.125}, draw=none, from=1-1, to=2-2]
		\arrow["t", from=1-2, to=2-2]
		\arrow["s"', from=2-1, to=2-2]
	\end{tikzcd}\]
	and also that the iterated tight pullbacks \[C_3 \defeq C_1 \times_{C_{0}} C_1 \times_{C_{0}} C_1 \hspace{1.5cm} C_4 \defeq C_1 \times_{C_{0}} C_1 \times_{C_{0}} C_1 \times_{C_{0}} C_1\] of the source and target morphisms exist.

	Furthermore, a pseudocategory comprises a loose morphism $i \colon C_0 \lto C_1$ satisfying the reflexivity equations $s \circ i = 1_{C_{0}} = t \circ i$, and a loose morphism
	$c \colon C_2 \lto C_1$ satisfying the compatibility equations $s \circ c = s \circ \pi^1_1$ and $t \circ c = t \circ \pi^1_2$ together with
	invertible 2-cells
	\[\begin{tikzcd}
		{C_3} & {C_2} \\
		{C_2} & {C_1}
		\arrow["{\langle c,1_{C_1}\rangle}", squiggly, from=1-1, to=1-2]
		\arrow["{\langle 1_{C_1},c\rangle}"', squiggly, from=1-1, to=2-1]
		\arrow["\alpha"{description}, shorten <=6pt, shorten >=6pt, Rightarrow, from=1-2, to=2-1]
		\arrow["c", squiggly, from=1-2, to=2-2]
		\arrow["c"', squiggly, from=2-1, to=2-2]
	\end{tikzcd}
	\hspace{3cm}
	\begin{tikzcd}
		{C_1} & {C_2} & {C_1} \\
		& {C_1}
		\arrow["{\langle i,1_{C_1}\rangle}", squiggly, from=1-1, to=1-2]
		\arrow[""{name=0, anchor=center, inner sep=0}, "{1_{C_1}}"', from=1-1, to=2-2]
		\arrow[""{name=0p, anchor=center, inner sep=0}, phantom, from=1-1, to=2-2, start anchor=center, end anchor=center]
		\arrow[""{name=1, anchor=center, inner sep=0}, "c"{description}, squiggly, from=1-2, to=2-2]
		\arrow[""{name=1p, anchor=center, inner sep=0}, phantom, from=1-2, to=2-2, start anchor=center, end anchor=center]
		\arrow[""{name=1p, anchor=center, inner sep=0}, phantom, from=1-2, to=2-2, start anchor=center, end anchor=center]
		\arrow["{\langle 1_{C_1},i\rangle}"', squiggly, from=1-3, to=1-2]
		\arrow[""{name=2, anchor=center, inner sep=0}, "{1_{C_1}}", from=1-3, to=2-2]
		\arrow[""{name=2p, anchor=center, inner sep=0}, phantom, from=1-3, to=2-2, start anchor=center, end anchor=center]
		\arrow["\lambda"', shift right=2, shorten <=5pt, shorten >=5pt, Rightarrow, from=1p, to=0p]
		\arrow["\rho"', shift right=2, shorten <=5pt, shorten >=5pt, Rightarrow, from=2p, to=1p]
	\end{tikzcd}\]
	satisfying the following two equations.\footnotemark{}
	\footnotetext{We have reversed the direction of $\alpha$ and $\rho$ in \cref{pseudocategory} in relation to \cite{martins2006pseudo} for consistency with the definition of left-skew monoid (\cref{skew-monoid}).}
	\begin{equation}\label{eq:pseudocat1}
	\begin{tikzcd}[column sep=small,row sep=large]
		& {C_3} && {C_2} \\
		{C_4} && {C_2} && {C_1} \\
		& {C_3} && {C_2}
		\arrow["{\langle c,1_{C_1}\rangle}", squiggly, from=1-2, to=1-4]
		\arrow["{\langle 1_{C_1},c\rangle}"{description}, squiggly, from=1-2, to=2-3]
		\arrow["{=}"{description}, draw=none, from=1-2, to=3-2]
		\arrow["\alpha", shorten <=6pt, shorten >=6pt, Rightarrow, from=1-4, to=2-3]
		\arrow["c", squiggly, from=1-4, to=2-5]
		\arrow["{\langle c,1_{C_1},1_{C_1}\rangle}", squiggly, from=2-1, to=1-2]
		\arrow["{\langle 1_{C_1},1_{C_1},c\rangle}"', squiggly, from=2-1, to=3-2]
		\arrow["c"{description}, squiggly, from=2-3, to=2-5]
		\arrow["\alpha", shorten <=6pt, shorten >=6pt, Rightarrow, from=2-3, to=3-4]
		\arrow["{\langle c,1_{C_1}\rangle}"{description}, squiggly, from=3-2, to=2-3]
		\arrow["{\langle 1_{C_1},c\rangle}"', squiggly, from=3-2, to=3-4]
		\arrow["c"', squiggly, from=3-4, to=2-5]
	\end{tikzcd}
	=\hspace{-1em}
	\begin{tikzcd}[row sep=large]
		& {C_3} && {C_2} \\
		{C_4} && {C_3} && {C_1} \\
		& {C_3} && {C_2}
		\arrow["{\langle c,1_{C_1}\rangle}", squiggly, from=1-2, to=1-4]
		\arrow["{\alpha 1_{C_1}}", shorten <=6pt, shorten >=6pt, Rightarrow, from=1-2, to=2-3]
		\arrow["c", squiggly, from=1-4, to=2-5]
		\arrow["\alpha", shorten <=13pt, shorten >=13pt, Rightarrow, from=1-4, to=3-4]
		\arrow["{\langle c,1_{C_1},1_{C_1}\rangle}", squiggly, from=2-1, to=1-2]
		\arrow["{\langle 1_{C_1},c,1_{C_1}\rangle}"{description}, squiggly, from=2-1, to=2-3]
		\arrow["{\langle 1_{C_1},1_{C_1},c\rangle}"', squiggly, from=2-1, to=3-2]
		\arrow["{\langle c,1_{C_1}\rangle}"{description}, squiggly, from=2-3, to=1-4]
		\arrow["{1_{C_1} \alpha}", shorten <=6pt, shorten >=6pt, Rightarrow, from=2-3, to=3-2]
		\arrow["{\langle 1_{C_1},c\rangle}"{description}, squiggly, from=2-3, to=3-4]
		\arrow["{\langle 1_{C_1},c\rangle}"', squiggly, from=3-2, to=3-4]
		\arrow["c"', squiggly, from=3-4, to=2-5]
	\end{tikzcd}
	\end{equation}
	\begin{equation}\label{eq:pseudocat2}
	\begin{tikzcd}
		&&& {C_2} \\
		{C_2} && {C_3} && {C_1} \\
		&&& {C_2}
		\arrow["c", squiggly, from=1-4, to=2-5]
		\arrow["\alpha", shorten <=13pt, shorten >=13pt, Rightarrow, from=1-4, to=3-4]
		\arrow[""{name=0, anchor=center, inner sep=0}, "{1_{C_2}}", curve={height=-24pt}, from=2-1, to=1-4]
		\arrow["{\langle 1_{C_1},c,1_{C_1} \rangle}"{description}, squiggly, from=2-1, to=2-3]
		\arrow[""{name=1, anchor=center, inner sep=0}, "{1_{C_2}}"', curve={height=24pt}, from=2-1, to=3-4]
		\arrow["{\langle c,1_{C_1} \rangle}"{description}, squiggly, from=2-3, to=1-4]
		\arrow["{\langle 1_{C_1},c \rangle}"{description}, squiggly, from=2-3, to=3-4]
		\arrow["c"', squiggly, from=3-4, to=2-5]
		\arrow["{\langle \rho, 1_{C_1}\rangle}"', shorten <=8pt, shorten >=8pt, Rightarrow, from=0, to=2-3]
		\arrow["{\langle 1_{C_1}, \lambda \rangle}"', shorten <=8pt, shorten >=8pt, Rightarrow, from=2-3, to=1]
	\end{tikzcd}
	\quad = \quad 1_c
	\qedshift
	\end{equation}
\end{definition}

\begin{example}
	\label{pseudocategory-examples}
	Specialising the definition of pseudocategory (\cref{pseudocategory}) to different $\F$-categories recovers various concepts of interest in the literature.
	\begin{center}
		\begin{tabu}{ccX[c]}
			$\F$-category & Pseudocategory & Reference \\
			\hline
			$\chord\Cat$ & Pseudo double cat. & \cite[\S7.1]{grandis1999limits} \\
			$\FMonCat s p$ & Monoidal double cat. & \cite[Remark~2.12]{shulman2010constructing} \\
			$\FDblCat s l$ & Horizontal intercategory & \cite[Definition~3.2]{grandis2015intercategories} \\
			$\FDblCat s c$ & Vertical intercategory & \cite[Definition~3.3]{grandis2015intercategories} \\
			$\FFib s p$ & Double fibration & \cite[Definition~2.1]{cruttwell2022double} \\
			$\FOpfib s p$ & Double opfibration & \cite[Remark~2.12]{cruttwell2022double} \\
			$\FDFib$ & Discrete double fibration & \cites[Definition~3.43]{cruttwell2022double}[Proposition~2.2.5]{lambert2021discrete} \\
			$T\h\fcat{Alg}_{s, c}$ & Colaxly $T$-structured pseudocategory & \cite[Definition~5.14]{capucci2024contextads}
		\end{tabu}
	\end{center}
	In the final example, $T$ is a 2-monad on a 2-category, and $T\h\fcat{Alg}_{s, c}$ is the $\F$-category whose objects are strict $T$-algebras, whose loose morphisms are the colax morphisms of $T$-algebras, in which a morphism is tight if it is strict, and whose 2-cells are $T$-algebra transformations~\cite[Example~3.3]{lack2012enhanced}.
\end{example}

Before defining fibrations in $\F$-categories, let us recall from \cite[\S3.1]{loregian2020categorical} that a morphism ${p \colon A \to B}$ in a 2-category $\tc C$ is said to be a fibration when, for all $X \in \tc C$, the functor $\tc C(X,p) \colon \tc C(X,A) \to \tc C(X,B)$ is a fibration in $\Cat$, and, for all morphisms $f \colon Y \to X$, the commutative square
\[\begin{tikzcd}
	{\tc C(X,A)} & {\tc C(Y,A)} \\
	{\tc C(X,B)} & {\tc C(Y,B)}
	\arrow["{\tc C(X,p)}"', from=1-1, to=2-1]
	\arrow["{\tc C(f,A)}", from=1-1, to=1-2]
	\arrow["{\tc C(Y,p)}", from=1-2, to=2-2]
	\arrow["{\tc C(f,B)}"', from=2-1, to=2-2]
\end{tikzcd}\]
is a pseudo morphism of fibrations. A more restrictive notion is that of a \emph{discrete fibration} in a 2-category: $p$ is said to be a discrete fibration just when each $\tc C(X,p) \colon \tc C(X,A) \to \tc C(X,B)$ is a discrete fibration of categories. Opfibrations and discrete opfibrations in $\tc C$ are defined analogously.

\begin{definition}
	\label{fibration}
	A \emph{fibration}/\emph{discrete fibration} in an $\F$-category $\fc C$ is a tight morphism $p \colon A \to B$ which is a fibration/discrete fibration in the 2-category $\fc C_\lambda$. Likewise, an \emph{opfibration}/\emph{discrete opfibration} in $\fc C$ is a tight morphism which is an opfibration/discrete opfibration in the 2-category $\fc C_\lambda$.
\end{definition}

\begin{example}
	\label{fibration-examples}
	Specialising the definition of (op)fibration (\cref{fibration}) to different $\F$-categories recovers several known flavours of (op)fibration.
	\begin{center}
	\begin{tabu}{cccX[c]}
		$\F$-category & (Op)fibration & Reference \\
		\hline
		$\chord\Cat$ & (Op)fibration & \cite[Proposition~3.6]{gray1966fibred} \\
		$\FMonCat s p$ & Monoidal (op)fibration & See (1) below. \\
		$\FDblCat s p$ & Double (op)fibration & \cite[Corollary~4.7]{cruttwell2022double} \\
		$\fcat{TanCat}_{s, p}$ & Tangent (op)fibration & See (2) below. \\
		$\FFib p c(B)$ & Fiberwise (op)fibration in $\b{Fib}(B)$ & See (3) below.
	\end{tabu}
	\end{center}
	\begin{enumerate}
		\item The fact that fibrations in $\FMonCat s p$ are precisely monoidal fibrations does not seem to have been noted before, but will be an instance of \cref{Mod-skew-symmetry-tight}, given that monoidal fibrations are pseudomonoids in $\FFib s p$.
		\item Here, $\fcat{TanCat}_{s, p}$ is the $\F$-category of tangent categories arising from the 2-category denoted by $\b{TNGCAT}_{\iso, \tx{pp}}$ in \cite[Lemma~3.25]{lanfranchi2024grothendieck} by equipping as the tight morphisms the strict morphisms of tangent categories~\cite{cockett2014differential}. \cite[Theorem~3.26]{lanfranchi2024grothendieck} then states that fibrations in $\fcat{TanCat}_{s, p}$ are precisely tangent fibrations in the sense of \cite[Definition~5.2]{cockett2018differential}.
		\item A fibrewise opfibration over a category $B$ is simply a commutative triangle
		\[\begin{tikzcd}
			E && {E'} \\
			& B
			\arrow["e", from=1-1, to=1-3]
			\arrow["p"', from=1-1, to=2-2]
			\arrow["{p'}", from=1-3, to=2-2]
		\end{tikzcd}\]
		in which, for each $b \in B$, the restriction $e_b \colon E_b \to E'_b$ of $e$ to the fibre over $b$ is an opfibration~\cite[Definition~2.1]{cigoli2020fibered}. The main case of interest is when $p$ and $p'$ are fibrations and $e$ is a pseudo morphism of fibrations: in \cite{cigoli2020fibered}, these are simply called \emph{fibrewise opfibrations in $\b{Fib}(B)$}. Due to the lack of any conditions relating the opcartesian liftings and the fibrations, they are not necessarily opfibrations in the 2-category $\b{Fib}(B)$ of fibrations and pseudo morphisms over $B$. Instead, it follows from \cite[Propositions~2.5 \& 2.7]{cigoli2020fibered} that they are precisely opfibrations in the $\F$-category $\FFib p c(B)$.
		\qedhere
	\end{enumerate}
\end{example}

\begin{remark}
	The reader may note that several of the previous examples may be presented in two ways: either as a $X$ in an $\F$-category of $Y$s, or as a $Y$ in an $\F$-category of $X$s. This is not a coincidence: that this is a general phenomenon will be shown in \cref{multiple-perspectives}.
\end{remark}

\section{Enhanced limit 2-sketches}
\label{sketches}

In the previous section, we examined several $\F$-categorical structures. In this section, we will show that they each arise as instances of the general theory of \emph{enhanced limit 2-sketches}. To begin, we recall the notion of $\V$-enriched limit sketch, following \textcite{kelly1982basic}, before specialising to the case of $\F$-enriched sketches. We then introduce the $\F$-categories of models of an $\F$-sketch, and investigate them in the cases of the $\F$-sketches for pseudomonoids, pseudocategories, and fibrations.

\subsection{Enriched limit sketches}

Limit sketches, introduced by \textcite{ehresmann1968esquisses}, provide a simple category-theoretic method for describing essentially algebraic structures. A limit sketch $\sk S$ comprises a small category $\ct S$ together with a collection of cones that specify the shapes of limits. Enriched limit sketches, introduced in \cite[Chapter~6]{kelly1982basic}\footnotemark{} and recalled below, naturally extend the concept of limit sketch to enriched category theory, the only significant change being the use of weighted cones rather than ordinary cones.
\footnotetext{In introducing sketches, \citeauthor{kelly1982basic} additionally assumes that $\V$ is locally bounded, but this is not necessary for the basic definitions presented here. However, we note that $\F$ is \lfp{}, hence locally bounded, so the complete theory \ibid indeed applies.}%

\begin{definition}[{\cite[\S6.3]{kelly1982basic}}]
	\label{sketch}
	A \emph{limit $\V$-sketch} $\sk S$ comprises a small $\V$-category $\ct S$ equipped with a class $\Gamma$ of weighted cones:
	\begin{equation}
		\label{cone}
		\Big\{ \big(W_i \colon \ct J_i \to \V , \quad D_i \colon \ct J_i \to \ct S , \quad X_i \in \ct S , \quad \gamma_i \colon W_i \tto \ct S(X_i, D_i{-}) \big) \Big\}_{i \in \Gamma}
	\end{equation}
	A morphism $F \colon \sk S \to \sk T$ of $\V$-sketches is a $\V$-functor $M \colon \ct S \to \ct T$ for which, given each weighted cone $(W_i, D_i, X_i, \gamma_i)$ in $\Gamma_{\sk S}$, the weighted cone $(W_i, M \c D_i, M(X_i), M_{X_i, D_i{-}} \c \gamma_i)$ belongs to $\Gamma_{\sk T}$. Limit $\V$-sketches, their morphisms, and arbitrary $\V$-natural transformations form a 2-category $\VSK$.

	For a class $\Psi$ of weights, a \emph{$\Psi$-limit $\V$-sketch} is one for which each weight $W_i$ belongs to $\Psi$. We denote by $\Psi\SK$ the full sub-2-category of $\VSK$ spanned by the $\Psi$-limit $\V$-sketches.
\end{definition}

\begin{remark}
	One may also consider the notion of \emph{mixed $\V$-sketch}, which permits the specification of weighted cocones as well as weighted cones. However, herein, we restrict our attention to limit $\V$-sketches. Consequently, we shall often drop the prefix \emph{limit} and simply refer to limit $\V$-sketches as $\V$-\emph{sketches}.
\end{remark}

\begin{definition}
	\label{model}
	Let $\sk S$ and $\sk C$ be $\V$-sketches. A \emph{model} for $\sk S$ in $\sk C$ is a sketch morphism $M \colon \sk S \to \sk C$. We denote by $\Mod(\sk S, \sk C)$ the full sub-$\V$-category of the $\V$-functor $\V$-category $[\ct S, \ct C]$ spanned by the models of $\sk S$ in $\sk C$.
\end{definition}

Each small $\V$-category $\ct C$ can be viewed in a canonical way as a $\V$-sketch, whose weighted cones comprise all the weighted limit cones in $\ct C$. It is consequently common to consider small $\V$-categories implicitly as $\V$-sketches and, for instance, to talk of a \emph{model of a $\V$-sketch in a $\V$-category}.

\begin{definition}[{\cite[\S6.3]{kelly1982basic}}]
	Let $\Psi$ be a class of $\V$-weights. A \emph{$\Psi$-limit theory} is the $\Psi$-limit sketch associated to a small $\Psi$-complete $\V$-category.
\end{definition}

We will use the same notation for a $\Psi$-complete $\V$-category and its associated $\Psi$-limit sketch. We shall need a basic observation for \cref{2-sketches} regarding the 2-functoriality of forming 2-categories of $\V$-sketches.

\begin{definition}[{\cite[Theorem~2.2]{eilenberg1966closed}}]
	Denote by $\b{CLMONCAT}_l$ the 2-category of closed monoidal categories, lax monoidal functors, and monoidal natural transformations\footnote{Note that lax monoidal functors and monoidal natural transformations automatically cohere appropriately with the closed structure~\cite[Theorem~5.1]{eilenberg1966closed}.}.
\end{definition}

The assignment, for each closed monoidal category $\V$, to the 2-category of $\V$-categories extends to a 2-functor $\ph\h\CAT \colon \b{CLMONCAT}_l \to 2\h\CAT$~\cite[Proposition~6.3]{eilenberg1966closed}.

\begin{proposition}[Change of base]
	\label{change-of-base}
	The 2-functor $\ph\h\CAT \colon \b{CLMONCAT}_l \to 2\h\CAT$ lifts to a 2-functor $\ph\SK \colon \b{CLMONCAT}_l \to 2\h\CAT$ assigning to each $\V$ the 2-category of $\V$-sketches.
\end{proposition}

\begin{proof}
	Let $F \colon \V \to \W$ be a lax monoidal functor between monoidal categories. Observe that there is a $\W$-functor $F^\W \colon F\h\CAT(\V) \to \W$ given on objects by $F$ and on hom-objects by the canonical morphism $F(\V(X, Y)) \to \W(FX, FY)$ induced by the lax monoidal structure of $F$. Consequently, each $\V$-weight $W \colon \ct J \to \V$ induces a $\W$-weight $F^\W \c F\h\CAT(W) \colon F\h\CAT(\ct J) \to \W$. Now let $\sk S$ be a $\V$-sketch. The $\W$-category $F\h\CAT(\ct S)$ may be equipped with the structure of a $\W$-sketch: for each $\V$-weighted cone $(W_i, D_i, X_i, \gamma_i)$ as in \eqref{cone} in $\sk S$, we define a $\W$-weighted cone as follows.
	\[\big( F^\W \c F\h\CAT(W_i), F\h\CAT(D_i), X_i, F^\W \c F\h\CAT(\gamma_i) \big)\]
	It is then clear that each morphism of $\V$-sketches induces a morphism of $\W$-sketches. This assignment defines a 2-functor $F\SK \colon \V\SK \to \W\SK$. Finally, given a monoidal transformation $\phi \colon F \tto G \colon \V \to \W$, it is straightforward to check that the 2-natural transformation $\phi\h\CAT \colon F\h\CAT \tto G\h\CAT$ restricts to a 2-natural transformation $F\SK \tto G\SK$.
\end{proof}

\begin{corollary}
	\label{restricted-change-of-base}
	Let $F \colon \V \to \W$ be a lax monoidal functor between closed monoidal categories, and let $\Psi$ be a class of $\V$-weights. There is an induced 2-functor ${\Psi\SK \to \Psi_F\SK}$, where $\Psi_F \defeq \{ F^\W \c F\h\CAT(W) \mid W \in \Psi \}$.
\end{corollary}

\begin{proof}
	Immediate from the construction of \cref{change-of-base}.
\end{proof}

\subsection{\texorpdfstring{$\F$}{F}-sketches and their \texorpdfstring{$\F$}{F}-categories of models}
\label{F-sketches}

\begin{definition}
	An \emph{enhanced limit 2-sketch} (or simply an \emph{enhanced 2-sketch}) is a limit $\F$-sketch in the sense of \cref{sketch}. We shall denote the underlying $\F$-category of an $\F$-sketch $\sk S$ by $\fc S$.
\end{definition}

We now turn to the $\F$-categories of models of $\F$-sketches. These will be defined as full sub-$\F$-categories of certain functor $\F$-categories, which we therefore introduce first.

\begin{definition}[{\cf{}~\cite[\S4.1]{lack2012enhanced}}]
	\label{functor-F-category}
	Let $w' \leq w \in \W$ be a pair of weaknesses, and let $\fc S$ and $\fc C$ be small $\F$-categories. Denote by $\FFun_{w', w}(\fc S, \fc C)$ the small $\F$-category of $\F$-functors and \emph{loose $(w', w)$-natural transformations}, defined as follows.
	\begin{enumerate}
		\item Objects are $\F$-functors $\fc S \to \fc C$.
		\item Loose morphisms $\phi \colon M \lto N$ are $w$-natural
		transformations $\phi \colon M \tto N$, which comprise $w$-natural families
		$\{ \phi_S \colon M(S) \lto N(S) \}_{S \in \fc S}$, such that the
		restricted family
		\[\begin{tikzcd}
			{\fc S_\tau} & {\fc S_\lambda} & {\fc C_\lambda}
			\arrow[""{name=0, anchor=center, inner sep=0}, "{M_\lambda}",
			shift left, curve={height=-6pt}, from=1-2, to=1-3]
			\arrow[""{name=1, anchor=center, inner sep=0}, "{N_\lambda}"',
			shift right, curve={height=6pt}, from=1-2, to=1-3]
			\arrow["\fc S", from=1-1, to=1-2]
			\arrow["\phi", shorten <=2pt, shorten >=2pt, Rightarrow,
			from=0, to=1]
		\end{tikzcd}\]
		is $w'$-natural. Explicitly, these comprise families of $w$-cells,
		\[\begin{tikzcd}
			{M(S)} & {N(S)} \\
			{M(S')} & {N(S')}
			\arrow["{\phi_S}", squiggly, from=1-1, to=1-2]
			\arrow["{\phi_{S'}}"', squiggly, from=2-1, to=2-2]
			\arrow[""{name=0, anchor=center, inner sep=0}, "{M(s)}"',
			squiggly, from=1-1, to=2-1]
			\arrow[""{name=1, anchor=center, inner sep=0}, "{N(s)}",
			squiggly, from=1-2, to=2-2]
			\arrow["{\phi_s}"{description}, draw=none, from=0, to=1]
		\end{tikzcd}\]
		such that $\phi_s$ is a $w'$-cell if $s \colon S \lto S'$ is tight.
		\item A morphism $\phi \colon M \lto N$ is tight when each component $\phi_S
		\colon M(S) \lto N(S)$ is tight and each $\phi_s$ is a $w'$-cell.
		\item 2-cells are modifications.
	\end{enumerate}
	Denote by $\FFun_w(\fc S, \fc C) \defeq \FFun_{s, w}(\fc S, \fc C)$ the $\F$-category of $\F$-functors and \emph{loose \mbox{$w$-natural} transformations}.
\end{definition}

In particular, $\FFun_s(\fc S, \fc C)$ is the usual $\F$-category of $\F$-functors in the sense of enriched category theory~\cite[Chapter~2]{kelly1982basic}. However, it is the $\F$-categories $\FFun_p(\fc S, \fc C)$, $\FFun_l(\fc S, \fc C)$, and $\FFun_c(\fc S, \fc C)$ of loose pseudo, lax, and colax natural transformations that are of primary interest here.

\begin{definition}
	\label{F-category-of-models}
	Let $\sk S$ and $\sk C$ be $\F$-sketches and let $w' \leq w \in \W$ be a pair of weaknesses. Define $\FMod_{w', w}(\sk S, \sk C) \hookrightarrow \FFun_{w', w}(\fc S, \fc C)$ to be the full sub-$\F$-category spanned by the models of $\sk S$ in $\sk C$. Define $\FMod_w(\sk S, \sk C) \defeq \FMod_{s, w}(\sk S, \sk C)$.
\end{definition}

Note that, under this definition, $\FMod_s(\sk S, \sk C) = \FMod(\sk S, \sk C)$ as defined in \cref{model} for general $\V$. It is, however, the $\F$-categories $\FMod_p(\sk S, \sk C)$, $\FMod_l(\sk S, \sk C)$ and $\FMod_c(\sk S, \sk C)$ of pseudo, lax, and colax morphisms that are of primary interest here.

\subsection{Examples of $\F$-sketches}
\label{examples-of-F-sketches}

We shall exhibit each of the $\F$-categorical structures in \cref{F-categorical-structures} as models of $\F$-sketches. This automatically provides us with appropriate notions of morphism for such structures, as well as facilitating the application of the general theory of $\F$-sketches that we will develop in the subsequent sections. We recall that, since $\F$ is cocomplete, the category of $\F$-categories is monadic over the category of $\F$-graphs~\cite[Theorem~2.13]{wolff1974v}: this permits us to speak of $\F$-categories freely generated by $\F$-graphs.

\begin{remark}
	In many examples of $\F$-sketches, the choice of tight morphisms is, in a certain sense, canonical. A categorical justification for this will be given in \cref{(co)free-F-sketches}, where we will show that each of following examples may be constructed freely from a limit 2-sketch.

	In the meantime, it may prove helpful to give a syntactic intuition for the choice of tight morphisms. Many sketches admit natural presentations as dependently typed algebraic theories~\cite{cartmell1986generalised}. From this perspective, the tight morphisms are precisely those encoding the type dependencies (\ie the \emph{display maps}~\cite{taylor1987recursive}). For instance, in the definition of a (pseudo)category in an $\F$-category (\cref{pseudocategory}), it is the source and target morphisms that are tight. This corresponds to the fact that, in the usual presentation of a category in terms of \eqref{Ob} a collection $\s{Ob}$ of objects and, \eqref{Hom} for each pair of objects $A$ and $B$, a collection $\s{Hom}(A, B)$ of morphisms, the source and target morphisms describe the type dependency of $\s{Hom}$.
	\begin{align}
		\label{Ob} \vdash \s{Ob}\text{ type} \\
		\label{Hom} A : \s{Ob}, B : \s{Ob} & \vdash \s{Hom}(A, B)\text{ type}
	\end{align}

	This intuition matches the choice of tight morphisms in \citeauthor{bastiani1974multiple}'s~\cite{bastiani1974multiple} approach to lax morphisms of sketches mentioned in \cref{lax-morphisms-via-2-categories}, since, in a model of dependent type theory, it is precisely the display maps over which one may take limits.
\end{remark}

\begin{example}[Pseudomonoid sketch]
	\label{sketch-pseudomonoid}
	The $\F$-sketch $\sk M$ for pseudomonoids (\cref{skew-monoid}) is a tight product sketch.  It contains objects $1$, $M$, $M^2$, $M^3$, $M^4$ and, as an $\F$-category, is freely generated by these objects together with
	\begin{itemize}
		\item tight morphisms $\pi^i_j \colon M^i \to M$ for $1 < i \leq 4$ and $1 \leq j \leq i$;
		\item a tight morphism $\unit \colon M \to 1$;
		\item loose morphisms ${\otimes} \colon M^2 \lto M$, $I \colon 1 \lto M$;
		\item further loose morphisms and a pair of invertible 2-cells as in the following diagrams
		\[
		\begin{tikzcd}[sep=large]
			{M^3} & {M^2} \\
			{M^2} & M
			\arrow["{M {\otimes}}"', squiggly, from=1-1, to=2-1]
			\arrow["\otimes"', squiggly, from=2-1, to=2-2]
			\arrow["\otimes", squiggly, from=1-2, to=2-2]
			\arrow["{{\otimes} M}", squiggly, from=1-1, to=1-2]
			\arrow["\alpha"{description}, shorten <=6pt, shorten >=6pt, Rightarrow, from=1-2, to=2-1]
		\end{tikzcd}
		\hspace{4em}
		\begin{tikzcd}[sep=large]
			M & {M^2} & M \\
			& M
			\arrow[""{name=0, anchor=center, inner sep=0}, "\otimes"{description}, squiggly, from=1-2, to=2-2]
			\arrow[""{name=0p, anchor=center, inner sep=0}, phantom, from=1-2, to=2-2, start anchor=center, end anchor=center]
			\arrow[""{name=0p, anchor=center, inner sep=0}, phantom, from=1-2, to=2-2, start anchor=center, end anchor=center]
			\arrow[""{name=1, anchor=center, inner sep=0}, "{1_M}", from=1-3, to=2-2]
			\arrow[""{name=1p, anchor=center, inner sep=0}, phantom, from=1-3, to=2-2, start anchor=center, end anchor=center]
			\arrow[""{name=2, anchor=center, inner sep=0}, "{1_M}"', from=1-1, to=2-2]
			\arrow[""{name=2p, anchor=center, inner sep=0}, phantom, from=1-1, to=2-2, start anchor=center, end anchor=center]
			\arrow["{M I}"', squiggly, from=1-3, to=1-2]
			\arrow["{I M}", squiggly, from=1-1, to=1-2]
			\arrow["\rho"', shift right=2, shorten <=5pt, shorten >=5pt, Rightarrow, from=1p, to=0p]
			\arrow["\lambda"', shift right=2, shorten <=5pt, shorten >=5pt, Rightarrow, from=0p, to=2p]
		\end{tikzcd}
		\]
		together with all of the further morphisms and 2-cells in \eqref{eq:pseudomon1} and \eqref{eq:pseudomon2} subject to those equations holding;
		\item the equations asserting that the morphisms and 2-cells with codomain $M^i$ for $1 < i \leq 4$ behave as expected with respect to the product projections $\pi^i_j$ for $1 \leq j \leq i$ (these are left to the reader).
	\end{itemize}
	The four specified cones are the tight product projections $\pi^i_j \colon M^i \to M$ for $1 < i \leq 4$ and $1 \leq j \in \leq i$ and an empty tight terminal object cone for $1$.

	For each weakness $w \in \W$, we have $\FMod_w(\sk M, \chord{\Cat}) \iso \FMonCat s w$ (\cref{ex:FMonCat}).
\end{example}

\begin{example}[Pseudocategory sketch]
	\label{sketch-pseudocategory}
	The $\F$-sketch $\sk C$ for pseudocategories (\cref{pseudocategory}) is a tight pullback sketch.  It contains objects $C_0$, $C_1$, $C_2$, $C_3$, $C_4$ and, as an $\F$-category, is freely generated by these objects together with
	\begin{itemize}
		\item tight morphisms $s,t \colon C_1 \to C_0$;
		\item tight morphisms $\pi^i_1, \pi^i_2 \colon C_{i+1} \to C_i$ for $1 \leq i \leq 3$ rendering commutative the following diagrams;
		\begin{equation}\label{eq:pullback}
		\begin{tikzcd}
			{C_{i+2}} & {C_{i+1}} \\
			{C_{i+1}} & {C_{i}}
			\arrow["{\pi^{i+1}_2}", from=1-1, to=1-2]
			\arrow["{\pi^{i+1}_1}"', from=1-1, to=2-1]
			\arrow["{\pi^i_1}", from=1-2, to=2-2]
			\arrow["{\pi^i_2}"', from=2-1, to=2-2]
		\end{tikzcd}
		\end{equation}
		\item loose morphisms $c \colon C_2 \lto C_1$ and $i \colon C_0 \lto C_1$ satisfying compatibility conditions with composition $s \circ c = s \circ \pi^1_1$ and $t \circ c = t \circ \pi^1_2$, and with identities $s \circ i = 1_{C_{0}} = t \circ i$;
		\item further morphisms and a pair of invertible 2-cells as in the following diagrams
		\[
		\begin{tikzcd}
			{C_3} & {C_2} \\
			{C_2} & {C_1}
			\arrow["{\langle c,1_{C_1}\rangle}", squiggly, from=1-1, to=1-2]
			\arrow["{\langle 1_{C_1},c\rangle}"', squiggly, from=1-1, to=2-1]
			\arrow["\alpha"', shorten <=6pt, shorten >=6pt, Rightarrow, from=1-2, to=2-1]
			\arrow["c", squiggly, from=1-2, to=2-2]
			\arrow["c"', squiggly, from=2-1, to=2-2]
		\end{tikzcd}
		\hspace{4em}
		\begin{tikzcd}
			{C_1} & {C_2} & {C_1} \\
			& {C_1}
			\arrow["{\langle i,1_{C_1}\rangle}", squiggly, from=1-1, to=1-2]
			\arrow[""{name=0, anchor=center, inner sep=0}, "{1_{C_1}}"', from=1-1, to=2-2]
			\arrow[""{name=0p, anchor=center, inner sep=0}, phantom, from=1-1, to=2-2, start anchor=center, end anchor=center]
			\arrow[""{name=1, anchor=center, inner sep=0}, "c"{description}, squiggly, from=1-2, to=2-2]
			\arrow[""{name=1p, anchor=center, inner sep=0}, phantom, from=1-2, to=2-2, start anchor=center, end anchor=center]
			\arrow[""{name=1p, anchor=center, inner sep=0}, phantom, from=1-2, to=2-2, start anchor=center, end anchor=center]
			\arrow["{\langle 1_{C_1},i\rangle}"', squiggly, from=1-3, to=1-2]
			\arrow[""{name=2, anchor=center, inner sep=0}, "{1_{C_1}}", from=1-3, to=2-2]
			\arrow[""{name=2p, anchor=center, inner sep=0}, phantom, from=1-3, to=2-2, start anchor=center, end anchor=center]
			\arrow["\lambda"', shift right=2, shorten <=5pt, shorten >=5pt, Rightarrow, from=1p, to=0p]
			\arrow["\rho"', shift right=2, shorten <=5pt, shorten >=5pt, Rightarrow, from=2p, to=1p]
		\end{tikzcd}\]
		together with all of the further morphisms and 2-cells in \eqref{eq:pseudocat1} and \eqref{eq:pseudocat2} subject to those equations holding;
		\item the equations asserting that the morphisms and $2$-cells with codomain $C_i$ for $2 \leq i \leq 3$ behave as expected under postcomposition by the pullback projections $\pi^i_1$ and $\pi^i_2$ (these are left to the reader).
	\end{itemize}
	There are three tight pullback cones, depicted in the commutative squares \eqref{eq:pullback}.

	For each weakness $w \in \W$, we have $\FMod_w(\sk C, \chord{\Cat}) \iso \FDblCat s w$ (\cref{ex:FDblCat}). Furthermore, for $w \ge p$, we have that $\FMod_{p, w}(\sk C, \chord{\Cat})$ is the $\F$-category of double categories, \emph{$w$-weak wobbly double functors} in the sense of \cite[Definition~2.1]{pare2015wobbly}, and natural transformations.
\end{example}

\begin{example}[Models in sketches]
	Every $\F$-category $\fc D$ admitting tight pullbacks may be viewed as a tight pullback $\F$-sketch $\sk D$. Models of $\sk C$ (\cref{sketch-pseudocategory}) in the $\F$-sketch $\sk D$ are then precisely pseudocategories in $\fc D$ in the sense of \cref{pseudocategory}. However, more generally, we can consider models of $\sk C$ in an $\F$-sketch $\sk D$ in which the cones are not limiting. We give two examples to illustrate the value in considering this more general notion. In the following, an \emph{enhanced category} is an enhanced 2-category whose underlying 2-category of loose morphisms is locally discrete; and the $\F$-sketch for \emph{categories} is obtained from the sketch for pseudocategories (\cref{pseudocategory}) by requiring $\alpha$, $\lambda$, $\rho$ to be identities.
	\begin{itemize}
		\item Denote by $\fcat{Set_{\tx{par}}}$ the enhanced category whose objects are sets, whose loose morphisms are partial functions, and whose tight morphisms are total functions. $\fcat{Set_{\tx{par}}}$ does not admit tight pullbacks. However, we may equip it with the structure of a tight pullback sketch $\sk P$ in which the cones are the (typically non-limiting) cones for pullbacks of functions.  A \emph{precategory} in the sense of \cite[Definition~3.1]{mateus1999precategories} is then precisely a category in $\sk P$ for which $i \colon C_0 \lto C_1$ is tight. Similarly, an unbiased precategory (\ie with $n$-ary, rather than binary, composition operations) in $\sk P$ is a \emph{paracategory} in the sense of \cite[\S2]{hermida2003paracategories}.
		\item Denote by $\fcat{Man}$ the enhanced category whose objects are smooth manifolds, whose loose morphisms are smooth functions, and whose tight morphisms are surjective submersions. We may equip $\fcat{Man}$ with the structure of a tight pullback sketch $\sk S$ in which the cones are the cones for pullbacks of surjective submersions. A category in $\sk S$ is then a \emph{differentiable category} in the sense of \cite[Definition~III.1.1]{mackenzie1987lie}; a model for the evident $\F$-sketch for groupoids is a \emph{differentiable groupoid}, \aka{} a \emph{Lie groupoid}.
		\qedhere
	\end{itemize}
\end{example}

\begin{example}[Fibration sketch]
	\label{sketch-fibration}
	Before describing the $\F$-sketch for fibrations,  let us first recall an algebraic description of fibrations in a sufficently complete 2-category $\tc C$. Consider a 1-cell $p \colon E \to B$ in $\tc C$ and suppose that both the power $\b 2 \pow E$ and colax limit $B \comma p$ exist. Then the composite 2-cell depicted below left induces a unique 1-cell $t \colon \b 2 \pow E \to B \comma p$ for which the equality below holds.
	\begin{equation}
	\label{fibration-equation}
	\begin{tikzcd}[row sep=2.6em]
		& {\b 2 \pow E} \\
		\\
		E && B
		\arrow[""{name=0, anchor=center, inner sep=0}, "{\pi_1}"', curve={height=18pt}, from=1-2, to=3-1]
		\arrow[""{name=1, anchor=center, inner sep=0}, "{\pi_2}"{description}, curve={height=-18pt}, from=1-2, to=3-1]
		\arrow[""{name=2, anchor=center, inner sep=0}, "{\pi_2p}", from=1-2, to=3-3]
		\arrow[""{name=3, anchor=center, inner sep=0}, "p"', from=3-1, to=3-3]
		\arrow["\varpi", shorten <=7pt, shorten >=7pt, Rightarrow, from=0, to=1]
		\arrow["{=}"{description}, draw=none, from=3, to=2]
	\end{tikzcd}
	\hspace{.6cm}
	=
	\hspace{.6cm}
	\begin{tikzcd}
		& {\b 2 \pow E} \\
		& {B \comma p} \\
		E && B
		\arrow["t"{description}, from=1-2, to=2-2]
		\arrow["{\pi^p_1}"', from=2-2, to=3-1]
		\arrow["{\pi^p_2}", from=2-2, to=3-3]
		\arrow[""{name=0, anchor=center, inner sep=0}, "p"', from=3-1, to=3-3]
		\arrow["{\varpi_p}"', shorten <=3pt, shorten >=5pt, Rightarrow, from=2-2, to=0]
	\end{tikzcd}
	\end{equation}
	By \cite[Theorem~3.1.3]{loregian2020categorical}, $p$ is a fibration precisely when $t$ admits a right adjoint ${r \colon B \comma p \to \b 2 \pow E}$ with identity counit $t r = 1$.%
	\footnote{This is a minor variation of \cite[Proposition~9]{street1974fibrations}.}
	The intuition is that, when $\tc C = \Cat$, the morphism $r(f \colon b \to pa) \colon b' \to a$ is the cartesian lifting of $f$: that it is a lifting of $f$ is captured by the fact that the counit is the identity, whilst the universality property of the cartesian lifting is captured by adjointness. Discrete fibrations in $\tc C$ can be characterised as those morphisms for which $t$ is invertible.

	Now suppose we are given a tight morphism $p \colon E \to B$ in an $\F$-category $\fc C$ admitting colax limits of tight arrows. We can form $\b 2 \pow E$ and $B \comma p$ and, by the characterisation of fibrations above in the 2-category $\fc C_\lambda$, the morphism $p$ is a fibration precisely when $t$ admits a right adjoint $r \colon B \comma p \lto \b 2 \pow E$ in $\fc C_\lambda$ with identity counit. Accordingly, the $\F$-sketch $\sk F$ for fibrations is a tight comma object sketch. As an $\F$-category, it is freely generated by objects, morphisms and 2-cells as below,
	\[
	\begin{tikzcd}[column sep=large,row sep=3.7em]
		{\b 2 \pow E} \\
		{B\comma p} & {\b 2 \pow E}
		\arrow["t"', squiggly, from=1-1, to=2-1]
		\arrow[""{name=0, anchor=center, inner sep=0}, "{1_{\b 2 \pow E}}", from=1-1, to=2-2]
		\arrow["r"', squiggly, from=2-1, to=2-2]
		\arrow["\eta"{description}, shorten <=4pt, Rightarrow, from=0, to=2-1]
	\end{tikzcd}
	\hspace{.8cm}
	\begin{tikzcd}
		& {\b 2 \pow E} \\
		\\
		E && B
		\arrow[""{name=0, anchor=center, inner sep=0}, "{\pi_1}"', curve={height=18pt}, from=1-2, to=3-1]
		\arrow[""{name=1, anchor=center, inner sep=0}, "{\pi_2}"{description}, curve={height=-18pt}, from=1-2, to=3-1]
		\arrow[""{name=2, anchor=center, inner sep=0}, "{\pi_2p}", from=1-2, to=3-3]
		\arrow[""{name=3, anchor=center, inner sep=0}, "p"', from=3-1, to=3-3]
		\arrow["\varpi", shorten <=7pt, shorten >=7pt, Rightarrow, from=0, to=1]
		\arrow["{=}"{description}, draw=none, from=3, to=2]
	\end{tikzcd}
	\hspace{.8cm}
	\begin{tikzcd}[row sep=3.4em]
		& {B \comma p} \\
		E && B
		\arrow["{\pi^p_1}"', from=1-2, to=2-1]
		\arrow["{\pi^p_2}", from=1-2, to=2-3]
		\arrow[""{name=0, anchor=center, inner sep=0}, "p"', from=2-1, to=2-3]
		\arrow["{\varpi_p}"{description}, shorten <=6pt, shorten >=6pt, Rightarrow, from=1-2, to=0]
	\end{tikzcd}
	\]
	satisfying the equation \eqref{fibration-equation}, $t r = 1$, and the triangle identities $t \eta = 1$ and $\eta t = 1$. The two specified weighted cones are $(\pi_1,\varpi,\pi_2)$ for the tight power by $\b 2$, and $(\pi^p_1,\varpi_p,\pi^p_2)$ for the colax limit of $p$.%
	\footnote{Note that in an $\F$-category with colax limits of tight morphisms, the induced morphism $t$ is necessarily tight (since $p$ is), but in the $\F$-sketch $\sk F$ we do not require $t$ to be tight. The reason for this will be justified in \cref{(co)free-F-sketches}, where we will show that this choice is determined by a universal construction. See \cref{regular-category-sketch} for a detailed discussion of a related situation.}

	For each $w \in \W$, we have $\FMod_w(\sk F, \chord{\Cat}) \iso \FFib s w$ (\cref{ex:FFib}). There is an analogous $\F$-sketch for discrete fibrations, in which we take $\eta$ to be an identity 2-cell. $\F$-sketches for opfibrations and discrete opfibrations are defined in the same way, but replacing the colax limit $B \comma p$ by the lax limit $p \comma B$.
\end{example}

\section{Multicategories of enhanced 2-sketches}
\label{multicategory-of-sketches}

In this section, we define multicategories $\FSK_{w', w}$ of $\F$-sketches and investigate their properties. In particular, we show that, for each pair $w' \leq w \in \W$ of weaknesses, the multicategory $\FSK_{w', w}$ is both left- and right-closed and representable, and so induces a left- and right-closed monoidal structure $\otimes_{w', w}$ on the category $\FSK$ of $\F$-sketches. Closure of $\FSK_{w', w}$ leads to our main applications of the theory of $\F$-sketches in \cref{multiple-perspectives}, namely in establishing the symmetry of internalisation.

Before proceeding with the definitions and their study, let us provide some context for our motivation in studying multicategories of $\F$-sketches. Recall that a morphism of $\F$-sketches $\sk S \to \sk C$ provides an interpretation of the structure of the $\F$-sketch $\sk S$ in the $\F$-sketch $\sk C$. More generally, a \emph{multi}morphism of $\F$-sketches $(\sk S_1, \ldots, \sk S_n) \to \sk C$ provides a interpretation in $\sk C$ of each of the structures of the sketches $\sk S_1$ through $\sk S_n$ simultaneously, which is coherent in the sense that the operations of each sketch are homomorphisms for each of the others. Closure means that such multimorphisms are in natural bijection with multimorphisms
$(\sk S_2, \ldots, \sk S_n) \to [\sk S_1, \sk C]_l$ and with multimorphisms
$(\sk S_1, \ldots, \sk S_{n - 1}) \to [\sk S_n, \sk C]_r$,
where $[{-}, {-}]_l$ and $[{-}, {-}]_r$ denote internal hom objects. The existence of such closed structure therefore allows us to view models of an $\F$-sketch $\sk S$ in the $\F$-category of models of another $\F$-sketch $\sk T$ as coherent models of both $\sk S$ and $\sk T$. We shall elaborate on this perspective in \cref{multiple-perspectives}.

\subsection{Background on multicategories}

We assume familiarity with the definition of multicategory~\cite[103]{lambek1969deductive}, but recall some important properties. First, multicategories may possess internal homs: in the absence of symmetry, there are two variants.

\begin{definition}[{\cite[\S4]{lambek1989multicategories}}]
	A multicategory $\ct M$ is \emph{left-closed} if, for each pair of objects $X, Z \in \ct M$, there is an object $[X, Z]_l \in \ct M$ and, for each $Y_1, \ldots, Y_n \in \ct M$, an isomorphism
	\[\ct M(X, Y_1, \ldots, Y_n; Z) \iso \ct M(Y_1, \ldots, Y_n; [X, Z]_l)\]
	natural in $X, Y_1, \ldots, Y_n, Z$. $\ct M$ is \emph{right-closed} if, for each pair of objects $Y, Z \in \ct M$, there is an object $[Y, Z]_r \in \ct M$ and, for each $X_1, \ldots, X_n \in \ct M$, an isomorphism
	\[\ct M(X_1, \ldots X_n, Y; Z) \iso \ct M(X_1, \ldots, X_n; [Y, Z]_r)\]
	natural in $X_1, \ldots, X_n, Y, Z$. $\ct M$ is \emph{closed}\footnotemark{} if it is left- and right-closed.
	\footnotetext{\citeauthor{lambek1989multicategories} uses \emph{biclosed} where we used \emph{closed}. However, the prefix \emph{bi-} is redundant; in the absence of symmetry, our usage of \emph{closed} is unambiguous. In the presence of symmetry, the left and right internal homs are canonically isomorphic (see below).}
\end{definition}

Every monoidal category induces a multicategory whose multimorphisms ${X_1, \ldots, X_n \to Y}$ are given by morphisms $X_1 \otimes \cdots \otimes X_n \to Y$. Conversely, it is possible to identify those multicategories that arise from monoidal categories in this way, via a universal property.

\begin{definition}[{\cites[\S4]{lambek1989multicategories}{hermida2000representable}}]
	\label{representable}
	A multicategory $\ct M$ is \emph{representable} if it admits
	\begin{itemize}
		\item a \emph{unit} object $I$ and nullary multimorphism ${} \to I$ for which the following function induced by precomposition is invertible for each $X_1, \ldots, X_n, Y \in \ct M$;
		\[\ct M(X_1, \ldots, X_i, I, X_{i + 1}, \ldots, X_n; Y) \to \ct M(X_1, \ldots, X_i, X_{i + 1}, \ldots, X_n; Y)\]
		\item for each pair of objects $A, B \in \ct M$, a \emph{tensor product} object $A \otimes B$ and a binary multimorphism $A, B \to A \otimes B$ for which the following function induced by precomposition is invertible for each $X_1, \ldots, X_n, Y \in \ct M$.
		\[\ct M(X_1, \ldots, X_i, A \otimes B, X_{i + 1}, \ldots, X_n; Y) \to \ct M(X_1, \ldots, X_i, A, B, X_{i + 1}, \ldots, X_n; Y) \qedhere\]
	\end{itemize}
\end{definition}

If a (closed) multicategory is representable, its underlying category of unary multimorphisms obtains a (closed) monoidal structure; conversely every (closed) monoidal category is canonically equipped with the structure of a (closed) representable multicategory. In fact, the concepts of (closed) monoidal category and of representable (closed) multicategory are essentially equivalent~\cite[\S9]{hermida2000representable}.

\begin{remark}
	\label{representability-simplified}
	Note that our chosen presentation of representability in \cref{representable} is a \emph{biased} form of representability (following \citeauthor{lambek1989multicategories}'s formulation), rather than \citeauthor{hermida2000representable}'s unbiased formulation, in which one asks for $n$-ary tensor products for each $n \in \N$. That these formulations are equivalent follows essentially as in \cite[Proposition~4.5]{bourke2018skew}. Furthermore, for \emph{closed} multicategories, it suffices in \cref{representable} to take $n = 0$~\cite[Lemma~3.7]{weber2013free}.
\end{remark}

Finally, there is a notion of symmetry for multicategories, reflecting that for monoidal categories.

\begin{definition}[{\cite[\S2.3]{baez1998higher}}]
	A \emph{symmetric multicategory} is a multicategory equipped with, for each list of objects $X_1, \ldots, X_n, Y \in \ct M$ and each permutation $\varrho$ of $\{ 1, \ldots, n \}$, a (necessarily invertible) function
	\[\varrho\ph \colon \ct M(X_1, \ldots, X_n; Y) \to \ct M(X_{\varrho 1}, \ldots, X_{\varrho n}; Y)\]
	such that $\varrho\ph$ is natural and compatible with pre- and postcomposition.
\end{definition}

As for the nonsymmetric case, the notions of symmetric (closed) monoidal categories and of representable symmetric (closed) multicategories are essentially equivalent~\cite[\S3.4]{weber2013free}. If $X$ and $Y$ are objects in a symmetric multicategory $\ct M$, then a left-hom $[X, Y]_l$ is automatically also a right-hom $[X, Y]_r$, and vice versa. Accordingly, in a closed symmetric multicategory, we denote the inner homs simply by $[X, Y]$.

In a closed \emph{symmetric} multicategory, we have a natural isomorphism $[X, [Y, Z]] \iso [Y, [X, Z]]$, corresponding to the process of uncurrying, applying a symmetry, then currying. Without symmetry, we cannot permute variables exactly in this manner, because we have two notions of inner hom. However, we are able to permute \emph{up to a twist}.

\begin{proposition}
	\label{closed-multicategory-skew-symmetry}
	Let $\ct M$ be a closed multicategory. For each $X, Y, Z \in \ct M$, there is an isomorphism $[X, [Y, Z]_\ell]_r \iso [Y, [X, Z]_r]_\ell$, natural in $X, Y, Z$.
\end{proposition}

\begin{proof}
	We have the following chain of isomorphisms, natural in $W, X, Y, Z$.
	\begin{align*}
		\ct M(W; [X, [Y, Z]_\ell]_r) & \iso \ct M(W, X; [Y, Z]_\ell) \\
		& \iso \ct M(Y, W, X; Z) \\
		& \iso \ct M(Y, W; [X, Z]_r) \\
		& \iso \ct M(W; [Y, [X, Z]_r]_\ell)
		\qedhere
	\end{align*}
\end{proof}

Similarly, we may internalise the tensor--hom adjunction, which takes two forms: one for the left-hom, and one for the right-hom.

\begin{proposition}
	\label{closed-multicategory-tensor}
	Let $\ct M$ be a multicategory. If $\ct M$ is right-closed then, for each ${X, Y, Z \in \ct M}$ for which a tensor product $X \otimes Y$ exists, there is an isomorphism $[X \otimes Y, Z]_r \iso [X, [Y, Z]_r]_r$, natural in $X, Y, Z$. Dually, if $\ct M$ is left-closed, then, for each such $X, Y, Z \in \ct M$, there is an isomorphism $[X \otimes Y, Z]_l \iso [Y, [X, Z]_l]_l$, natural in $X, Y, Z$.
\end{proposition}

\begin{proof}
	If $\ct M$ is right-closed, then we have the following chain of isomorphisms, natural in $W, X, Y, Z$. The case in which $\ct M$ is left-closed follows dually.
	\begin{align*}
		\ct M(W; [X \otimes Y, Z]_r) & \iso \ct M(W, X \otimes Y; Z) \\
		& \iso \ct M(W, X, Y; Z) \\
		& \iso \ct M(W, X; [Y, Z]_r) \\
		& \iso \ct M(W; [X, [Y, Z]_r]_r)
		\qedhere
	\end{align*}
\end{proof}

\Cref{closed-multicategory-tensor} reveals a shortcoming of nonsymmetric tensor products: we are unable to relate the tensor product to the twisted symmetry of \cref{closed-multicategory-skew-symmetry}. However, in \cref{applications}, we will see an example of a situation in which this mismatch leads to interesting behaviour.

\subsection{The multicategory of \texorpdfstring{$\F$}{F}-categories}

We construct our multicategories of \mbox{$\F$-sketches} in two stages, starting by constructing multicategories $\FCAT_{w', w}$ of $\F$-categories. These enhance the classical multicategories of 2-categories corresponding to the different flavours of the tensor product of 2-categories, as first studied by \textcite[Theorem~1,4.14]{gray1974formal}. Whilst in 2-category theory, there is a $w$-tensor product for each weakness $w \in \W$, in $\F$-category theory we obtain a $(w', w)$-tensor product for each pair $w' \leq w \in \W$, corresponding to the choice of tight morphisms as a subclass of the loose morphisms. Rather than constructing these closed monoidal structures directly as \citeauthor{gray1974formal} does, we follow \textcite[\S1.3]{verity1992enriched} (\cf~\cite[Theorem~{I,4.7}]{gray1974formal}) in working multicategorically from the start, which avoids the complexity of defining and proving the coherence of tensor products of 2-categories. In fact, the closed multicategories we construct are refinements of the one constructed by \citeauthor{verity1992enriched}, which means properties like unitality and associativity of composition of multimorphisms follow automatically.

We begin by defining a notion of multimorphism of $\F$-categories, parametrised by $w' \leq w \in \W$.

\begin{definition}
	\label{F-multifunctor}
	Let $w' \leq w \in \W$ be weaknesses and let $\fc S_1, \ldots, \fc S_n, \fc C$ be small \mbox{$\F$-categories}. A \emph{$(w', w)$-natural $\F$-multifunctor} $M$ from $\fc S_1, \ldots, \fc S_n$ to $\fc C$, for $n > 0$, comprises the following data.
	\begin{enumerate}
		\item \label{separate-functoriality} For each $S_1 \in \fc S_1, \ldots, S_n \in \fc S_n$, an object $M(S_1, \ldots, S_n) \in \fc C$, such that, for each $1 \leq i \leq n$, the assignment $M(S_1, \ldots, {-}_i, \ldots, S_n)$ is equipped with the structure of an $\F$-functor $\fc S_i \to \fc C$.
		\item \label{w-cell-data} For each $S_1 \in \fc S_1, \ldots, S_n \in \fc S_n$ and pair of loose morphisms $s_i \colon S_i \lto S_i'$ in $\fc S_i$ and $s_j \colon S_j \lto S_j'$ in $\fc S_j$ for $1 \leq i < j \leq n$, a $w$-cell in $\fc C$
		\[\begin{tikzcd}[column sep=8em,row sep=large]
			{M(S_1, \ldots, S_i, \ldots, S_j, \ldots, S_n)} & {M(S_1, \ldots, S_i', \ldots, S_j, \ldots, S_n)} \\
			{M(S_1, \ldots, S_i, \ldots, S_j', \ldots, S_n)} & {M(S_1, \ldots, S_i', \ldots, S_j', \ldots, S_n)}
			\arrow["{M(S_1, \ldots, s_i, \ldots, S_j, \ldots, S_n)}", squiggly, from=1-1, to=1-2]
			\arrow[""{name=0, anchor=center, inner sep=0}, "{M(S_1, \ldots, S_i, \ldots, s_j, \ldots, S_n)}"', squiggly, from=1-1, to=2-1]
			\arrow["{M(S_1, \ldots, s_i, \ldots, S_j', \ldots, S_n)}"', squiggly, from=2-1, to=2-2]
			\arrow[""{name=1, anchor=center, inner sep=0}, "{M(S_1, \ldots, S_i', \ldots, s_j, \ldots, S_n)}", squiggly, from=1-2, to=2-2]
			\arrow["{M(S_1, \ldots, s_i, \ldots, s_j, \ldots, S_n)}"{description}, draw=none, from=1, to=0]
		\end{tikzcd}\]
		such that
		\begin{enumerate}
			\item $M(S_1, \ldots, s_i, \ldots, s_j, \ldots, S_n)$ is a $w'$-cell if either $s_i$ or $s_j$ is tight;
			\item the assignment $M(S_1, \ldots, s_i, \ldots, {-}_j, \ldots, S_n)$ is equipped with the structure of a loose $w$-natural transformation
			\[M(S_1, \ldots, S_i, \ldots, {-}_j, \ldots, S_n) \tto M(S_1, \ldots, S_i', \ldots, {-}_j, \ldots, S_n)\]
			\item the assignment $M(S_1, \ldots, {-}_i, \ldots, s_j, \ldots, S_n)$ is equipped with the structure of a loose $\wb$-natural transformation
			\[M(S_1, \ldots, {-}_i, \ldots, S_j, \ldots, S_n) \tto M(S_1, \ldots, {-}_i, \ldots, S_j', \ldots, S_n)\]
			\item for each loose morphism $s_i \colon S_i \lto S_i'$ in $\fc S_i$, $s_j \colon S_j \lto S_j'$ in $\fc S_j$, and $s_k \colon S_k \lto S_k'$ in $\fc S_k$ for $1 \leq i < j < k \leq n$, the evident cubical compatibility condition holds~\cites[Definition~1,4.6]{gray1974formal}[Definition~1.3.3]{verity1992enriched}.
		\end{enumerate}
	\end{enumerate}
	A nullary $(w', w)$-natural $\F$-multifunctor to $\fc C$ is simply an object of $\fc C$.
\end{definition}

When the $\F$-categories in question are chordate, $(s,l)$-natural $\F$-multifunctors coincide with the multimorphisms of \cite[Definition~1.3.3]{verity1992enriched} (when restricted to 2-categories).

\begin{proposition}
	$\F$-categories and $(w', w)$-natural $\F$-multifunctors form a multicategory structure $\FCAT_{w', w}$ on $\FCAT$, for each pair of weaknesses $w' \leq w \in \W$.
\end{proposition}

\begin{proof}
	Identities and composites of $(w',l)$-natural $\F$-multifunctors are defined as in \cite[Definition~1.3.3]{verity1992enriched}. It is therefore necessary only to check that identities and composites respect the $\F$-functorial structure of the object assignments, and restrict to $w$-cells appropriately on the 2-cell components, both of which follow easily by expanding the definitions. This defines a multicategory $\FCAT_{w', l}$. The case of $(w', c)$-natural $\F$-multifunctors follows dually. The cases $w = s$ and $w = p$ are straightforward restrictions of the above cases. Finally, it is immediate that the underlying category of unary multimorphisms is simply $\FCAT$.
\end{proof}

\begin{proposition}
	\label{multicategory-symmetry}
	$\FCAT_{w', w}$ is symmetric for $w' \leq w \leq p$.
\end{proposition}

\begin{proof}
	Symmetry follows as for the symmetry of the multicategory of bicategories and pseudonatural cubical multifunctors in \cite[\S1.3]{verity1992enriched}, the potential noninvertibility of the 2-cell components of an $\F$-multifunctor being the only obstruction to symmetry.
\end{proof}

We shall now show that $\FCAT_{w', w}$ is closed. Since $\FCAT_{w', w}$ is not symmetric for general $w' \leq w \in \W$, we must exhibit both left- and right-closed structure.

\begin{proposition}
	\label{FCAT-is-closed}
	For each pair of weaknesses $w' \leq w \in \W$, the $\F$-categories $\FFun_{\wbp, \wb}({-}, {-})$ and $\FFun_{w', w}({-}, {-})$ of \cref{functor-F-category} equip $\FCAT_{w', w}$ with left- and right-closed structure respectively.
\end{proposition}

\begin{proof}
	We must show that there are isomorphisms
	\begin{align}
		\FCAT_{w', w}(\fc S_1, \ldots, \fc S_n; \fc C) & \iso \FCAT_{w', w}(\fc S_2, \ldots, \fc S_n; \FFun_{\wbp, \wb}(\fc S_1, \fc C)) \\
		\label{right-closure} \FCAT_{w', w}(\fc S_1, \ldots, \fc S_n; \fc C) & \iso \FCAT_{w', w}(\fc S_1, \ldots, \fc S_{n - 1}; \FFun_{w', w}(\fc S_n, \fc C))
	\end{align}
	natural in $\fc S_1, \ldots, \fc S_n, \fc C$. The proof is analogous to its correspondent for 2-categories~\cite[Lemma~1.3.4]{verity1992enriched}. Consider a $(w', w)$-natural $\F$-multifunctor $M \colon (\fc S_2, \ldots, \fc S_n) \to \FFun_{\wbp, \wb}(\fc S_1, \fc C)$. Explicitly, $M$ comprises the following data, with \mbox{(1 -- 3)} corresponding to \cref{separate-functoriality}, and (4) corresponding to \cref{w-cell-data}.
	\begin{enumerate}
		\item For each $S_2 \in \fc S_2, \ldots, S_n \in \fc S_n$, an $\F$-functor $M(S_2, \ldots, S_n) \colon \fc S_1 \to \fc C$.
		\item For each loose morphism $s_i \colon S_i \lto S_i'$ in $\fc S_i$ for $1 < i \leq n$, a $\wb$-natural family,
		\[\{ M(S_2, \ldots, s_i, \ldots, S_n)_{S_1} \colon M(S_2, \ldots, S_i, \ldots, S_n)(S_1) \lto M(S_2, \ldots, S_i', \ldots, S_n)(S_1) \}_{S_1 \in \fc S_1}\]
		comprising a family of $\wb$-cells in $\fc C$, for each $s_1 \colon S_1 \lto S_1'$ in $\fc S_1$,
		\[\begin{tikzcd}[column sep=8em,row sep=large]
			{M(S_2, \ldots, S_i, \ldots, S_n)(S_1)} & {M(S_2, \ldots, S_i', \ldots, S_n)(S_1)} \\
			{M(S_2, \ldots, S_i, \ldots, S_n)(S_1')} & {M(S_2, \ldots, S_i', \ldots, S_n)(S_1')}
			\arrow["{M(S_2, \ldots, s_i, \ldots, S_n)_{S_1}}", squiggly, from=1-1, to=1-2]
			\arrow["{M(S_2, \ldots, s_i, \ldots, S_n)_{S_1'}}"', squiggly, from=2-1, to=2-2]
			\arrow[""{name=0, anchor=center, inner sep=0}, "{M(S_2, \ldots, S_i, \ldots, S_n)(s_1)}"', squiggly, from=1-1, to=2-1]
			\arrow[""{name=1, anchor=center, inner sep=0}, "{M(S_2, \ldots, S_i', \ldots, S_n)(s_1)}", squiggly, from=1-2, to=2-2]
			\arrow["{M(S_2, \ldots, s_i, \ldots, S_n)_{s_1}}"{description}, draw=none, from=0, to=1]
		\end{tikzcd}\]
		such that
		\begin{enumerate}
			\item $M(S_2, \ldots, s_i, \ldots, S_n)_{S_1}$ is tight if $s_i$ is tight;
			\item $M(S_2, \ldots, s_i, \ldots, S_n)_{s_1}$ is a $\wbp$-cell if either $s_1$ or $s_i$ is tight.
		\end{enumerate}
		\item For each 2-cell $\sigma_i \colon s_i \tto s_i'$, a modification $M(S_2, \ldots, \sigma_i, \ldots, S_n) \colon M(S_2, \ldots, s_i, \ldots, S_n) \tto M(S_2, \ldots, s_i', \ldots, S_n)$.
		\item For each $S_2 \in \fc S_2, \ldots, S_n \in \fc S_n$ and pair of loose morphisms $s_i \colon S_i \lto S_i'$ in $\fc S_i$ and $s_j \colon S_j \lto S_j'$ in $\fc S_j$ for $1 < i < j \leq n$, a $w$-cell in $\FFun_{\wbp, \wb}(\fc S_1, \fc C)$, hence a family of $w$-cells in $\fc C$, with $S_1 \in \fc S_1$, forming a modification
		\[\begin{tikzcd}[column sep=10em,row sep=large]
			{M(S_2, \ldots, S_i, \ldots, S_j, \ldots, S_n)(S_1)} & {M(S_2, \ldots, S_i', \ldots, S_j, \ldots, S_n)(S_1)} \\
			{M(S_2, \ldots, S_i, \ldots, S_j', \ldots, S_n)(S_1)} & {M(S_2, \ldots, S_i', \ldots, S_j', \ldots, S_n)(S_1)}
			\arrow["{M(S_2, \ldots, s_i, \ldots, S_j, \ldots, S_n)_{S_1}}", squiggly, from=1-1, to=1-2]
			\arrow[""{name=0, anchor=center, inner sep=0}, "{M(S_2, \ldots, S_i, \ldots, s_j, \ldots, S_n)(S_1)}"{description}, squiggly, from=1-1, to=2-1]
			\arrow[""{name=1, anchor=center, inner sep=0}, "{M(S_2, \ldots, S_i', \ldots, s_j, \ldots, S_n)(S_1)}"{description}, squiggly, from=1-2, to=2-2]
			\arrow["{M(S_2, \ldots, s_i, \ldots, S_j', \ldots, S_n)_{S_1}}"', squiggly, from=2-1, to=2-2]
			\arrow["{M(S_2, \ldots, s_i, \ldots, s_j, \ldots, S_n)_{S_1}}"{description}, draw=none, from=1, to=0]
		\end{tikzcd}\]
		such that
		\begin{enumerate}
			\item $M(S_2, \ldots, s_i, \ldots, s_j, \ldots, S_n)_{S_1}$ is a $w'$-cell if either $s_i$ or $s_j$ is tight;
			\item the assignment $M(S_2, \ldots, s_i, \ldots, {-}_j, \ldots, S_n)_{S_1}$ is equipped with the structure of a loose $w$-natural transformation
			\[M(S_2, \ldots, S_i, \ldots, {-}_j, \ldots, S_n)_{S_1} \tto M(S_2, \ldots, S_i', \ldots, {-}_j, \ldots, S_n)_{S_1}\]
			\item the assignment $M(S_2, \ldots, {-}_i, \ldots, s_j, \ldots, S_n)_{S_1}$ is equipped with the structure of a loose $\wb$-natural transformation
			\[M(S_2, \ldots, {-}_i, \ldots, S_j, \ldots, S_n)_{S_1} \tto M(S_2, \ldots, {-}_i, \ldots, S_j', \ldots, S_n)_{S_1}\]
			\item for each loose morphism $s_i \colon S_i \lto S_i'$ in $\fc S_i$, $s_j \colon S_j \lto S_j'$ in $\fc S_j$, and $s_k \colon S_k \lto S_k'$ in $\fc S_k$ for $1 < i < j < k \leq n$, the cubical compatibility condition holds.
		\end{enumerate}
	\end{enumerate}
	By inspection, (1 -- 3) above specifies the data of a $(w', w)$-natural $\F$-multifunctor $(\fc S_1, \ldots, \fc S_n) \to \fc C$ involving the first argument; whilst (4) specifies the data of such a multifunctor in the remaining arguments. Note that the $\wb$- and $\wbp$-cells arising in (2) become $w$- and $w'$-cells under this specification, because the order of arguments is swapped.
	We hence obtain a $(w', w)$-natural $\F$-multifunctor $M' \colon (\fc S_1, \ldots, \fc S_n) \to \fc C$ by the following definitions.
	\begin{align*}
		M'({-}, S_2, \ldots, S_n) & \defeq M(S_1, \ldots, S_n)\ph \\
		M'(S_1, S_2, \ldots, {-}_i, \ldots, S_n) & \defeq M(S_2, \ldots, {-}_i, \ldots, S_n)_{S_1}
	\end{align*}
	It is clear that the assignment $M \mapsto M'$ is bijective and natural in $\fc S_1, \ldots, \fc S_n, \fc C$. The isomorphism \eqref{right-closure} follows by symmetric reasoning.
\end{proof}

Observe that the left-closed structure on $\FCAT_{w', w}$ coincides with the right-closed structure on $\FCAT_{\wbp, \wb}$. Next, we extend \citeauthor{gray1974formal}'s $w$-tensor product of 2-categories to a $(w', w)$-tensor product of $\F$-categories.

\begin{definition}[Tensor product of $\F$-categories]
	\label{tensor-product}
	Let $w' \leq w \in \W$ be weaknesses, and let $\fc S$ and $\fc T$ be small $\F$-categories. Define a small $\F$-category $\fc S \otimes_{w', w} \fc T$ by the following presentation.
	\begin{enumerate}
		\item An object is a pair $(S, T)$ where $S \in \fc S$ and $T \in \fc T$.
		\item Loose morphisms are freely generated by the following, subject to equations expressing functoriality in each argument separately.
		\begin{enumerate}
			\item Loose morphisms $(s, 1) \colon (S, T) \lto (S', T)$ for $s \colon S \lto S'$ in $\fc S$.
			\item Loose morphisms $(1, t) \colon (S, T') \lto (S, T')$ for $t \colon T \lto T'$ in $\fc T$.
		\end{enumerate}
		\item Tightness is determined inductively by the generators: a loose morphism $(s, 1) \colon (S, T) \lto (S', T)$ is tight if $s$ is tight in $\fc S$; a loose morphism $(1, t) \colon (S, T) \lto (S, T')$ is tight if $t$ is tight in $\fc T$.
		\item 2-cells are freely generated by the following, subject to the same equations as in \cite[Theorem~{I,4.9}]{gray1974formal}.
		\begin{enumerate}
			\item 2-cells $(\sigma, 1) \colon (s, 1) \tto (s', 1) \colon (S, T) \lto (S', T)$ for $\sigma \colon s \tto s'$ in $\fc S$.
			\item 2-cells $(1, \tau) \colon (1, t) \tto (1, t') \colon (S, T) \lto (S, T')$ for $\tau \colon t \tto t'$ in $\fc T$.
			\item $w$-cells
			\[\begin{tikzcd}
				{(S, T)} & {(S, T')} \\
				{(S', T)} & {(S', T')}
				\arrow["{(1, t)}", squiggly, from=1-1, to=1-2]
				\arrow[""{name=0, anchor=center, inner sep=0}, "{(s, 1)}"', squiggly, from=1-1, to=2-1]
				\arrow[""{name=1, anchor=center, inner sep=0}, "{(s, 1)}", squiggly, from=1-2, to=2-2]
				\arrow["{(1, t)}"', squiggly, from=2-1, to=2-2]
				\arrow["{\upsilon_{s, t}}"{description}, draw=none, from=0, to=1]
			\end{tikzcd}\]
			for $s \colon S \lto S'$ in $\fc S$ and $t \colon T \lto T'$ in $\fc T$, such that $\upsilon_{s, t}$ is a $w'$-cell if either $s$ or $t$ is tight.
		\end{enumerate}
	\end{enumerate}
	Define $\fc S \otimes_w \fc T \defeq \fc S \otimes_{s, w} \fc T$.
\end{definition}

\begin{theorem}
	\label{FCAT-is-representable}
	Let $w' \leq w \in \W$ be weaknesses. The tensor product $\otimes_{w', w}$ of $\F$-categories exhibits the multicategory $\FCAT_{w', w}$ as representable, and so equips the category $\FCAT$ with closed semicartesian monoidal structure.
\end{theorem}

\begin{proof}
	Recall from \cref{representability-simplified} that, since $\FCAT_{w', w}$ is closed by \cref{FCAT-is-closed}, it suffices to check \cref{representable} in the case $n = 0$. It is immediate from inspection of \cref{F-multifunctor} that there is an isomorphism as follows, thus exhibiting unitality.
	\[\FCAT_{w', w}(\bbn 1; \fc C) \iso \FCAT_{w', w}({}; \fc C)\]
	Next, we shall show that $\otimes_{w', w}$ exhibits a binary tensor product, for which we must exhibit an isomorphism as follows.
	\[\FCAT_{w', w}(\fc S \otimes_{w', w} \fc T; \fc C) \iso \FCAT_{w', w}(\fc S, \fc T; \fc C)\]
	By construction, our $(w', w)$-weak tensor product of $\F$-categories restricts on loose morphisms to the usual $w$-weak tensor product of 2-categories~\cite{gray1974formal}. Therefore, it is enough to show that the tensor product behaves appropriately with respect to tight morphisms.

	For this, observe that there are two places tightness plays a role in \cref{F-multifunctor}: the first in the separate $\F$-functoriality of \cref{separate-functoriality}; the second in the $w$-cell data of \cref{w-cell-data}. The tightness condition in the former corresponds in \cref{tensor-product} to the tightness of $(s, 1)$ and $(1, t)$ when $s$ and $t$ respectively are tight. The latter corresponds to the condition that $\upsilon_{s, t}$ is a $w'$-cell when either $s$ or $t$ is tight.
\end{proof}

\begin{remark}
	Just as $\FFun_s$ is the usual enriched category of enriched functors, $\otimes_s$ is the usual tensor product of enriched categories~\cite[\S1.4]{kelly1982basic}. Consequently, since $\F$ is cartesian, $\otimes_s$ is the cartesian product of $\F$-categories. The description of \cref{tensor-product} can be simplified in this case: $(\fc S \times \fc T)_\lambda = \fc S_\lambda \times \fc T_\lambda$ and a morphism $(s, t)$ is tight just when both $s$ and $t$ are tight.
\end{remark}

\subsection{The multicategory of enhanced 2-sketches}

We now refine the notion of \mbox{$\F$-multifunctor} to the case in which the objects in the domain of the multimorphism are $\F$-sketches, and in which each argument is required to preserve weighted cones.

\begin{definition}
	Let $w' \leq w \in \W$ be weaknesses, and let $\sk S_1, \ldots, \sk S_n, \sk C$ be $\F$-sketches. A \emph{$(w', w)$-natural $\F$-multimodel} $M$ from $\sk S_1, \ldots, \sk S_n$ to $\sk C$, for $n > 0$, is a $(w', w)$-natural \mbox{$\F$-multifunctor} $M \colon (\fc S_1, \ldots, \fc S_n) \to \fc C$ such that, for each $S_1 \in \fc S_1, \ldots, S_n \in \fc S_n$, the $\F$-functor $M(S_1, \ldots, {-}_i, \ldots, S_n) \colon \fc S_i \to \fc C$ is an $\sk S_i$-model. A nullary $(w', w)$-natural \mbox{$\F$-multimodel} to $\sk C$ is simply an object of $\fc C$.
\end{definition}

Binary multimorphisms generalise the classic notions of bimodel for sketches (\cf{}~\cite{freyd1966algebra,borceux1994notion}).

\begin{proposition}
	\label{FCAT-FSK-coreflection}
	$\F$-sketches and $(w', w)$-natural $\F$-multimodels form a multicategory structure $\FSK_{w', w}$ on $\FSK$, for each pair of weaknesses $w' \leq w \in \W$. Furthermore, the forgetful functor $\FSK_{w', w} \to \FCAT_{w', w}$ of multicategories has a \ff{} left adjoint.
	\[\begin{tikzcd}
		{\FCAT_{w', w}} & {\FSK_{w', w}}
		\arrow[""{name=0, anchor=center, inner sep=0}, "{\ph^{\sharp}}", shift left=2, hook, from=1-1, to=1-2]
		\arrow[""{name=1, anchor=center, inner sep=0}, shift left=2, from=1-2, to=1-1]
		\arrow["\dashv"{anchor=center, rotate=-90}, draw=none, from=0, to=1]
	\end{tikzcd}\]
\end{proposition}

\begin{proof}
	The multicategory structure of $\FCAT_{w', w}$ trivially restricts to multimodels, since weighted cone preservation is closed under identities and composition. The assignment sending each $\F$-sketch to its underlying $\F$-category is trivially functorial and faithful, since we defined the multicategory structure on $\FSK_{w', w}$ in terms of that of $\FCAT_{w', w}$. (Note, though, that this functor does not preserve the closed structure.)

	The fully faithful left adjoint $\ph^{\sharp}$ sends an $\F$-category to the trivial $\F$-sketch structure on the $\F$-category, \ie{} with an empty set of cones.
\end{proof}

\begin{proposition}
	$\FSK_{w', w}$ is symmetric for $w' \leq w \leq p$.
\end{proposition}

\begin{proof}
	Follows from the symmetry of $\FCAT_{w', w}$ in \cref{multicategory-symmetry}, preservation of weighted cones being independent of argument order.
\end{proof}

We shall now show that $\FSK_{w', w}$ is closed. Again, since $\FSK_{w', w}$ is not symmetric in general, we must exhibit both left- and right-closed structure. First, we define the suitable internal hom of $\F$-sketches.

\begin{definition}[Internal hom of $\F$-sketches]
	Let $w' \leq w \in \W$ be weaknesses, and let $\sk S$ and $\sk C$ be $\F$-sketches. The $\F$-category $\FMod_{w', w}(\sk S, \sk C)$ of \cref{F-category-of-models} may be equipped with the structure of an $\F$-sketch $\SMod_{w', w}(\sk S, \sk C)$, whose weighted cones $W \tto \FMod_{w', w}(\sk S, \sk C)(X,D-)$ are precisely those whose composite with each evaluation $\F$-functor $\ph(S) \colon \FMod_{w', w}(\sk S, \sk C) \to \fc C$, for $S \in \fc S$, is a weighted cone in the $\F$-sketch $\sk C$. Define $\SMod_w(\sk S, \sk C) \defeq \SMod_{s, w}(\sk S, \sk C)$.
\end{definition}

\begin{proposition}
	\label{FSK-is-closed}
	For each pair of weaknesses $w' \leq w \in \W$, the $\F$-sketches $\SMod_{\wbp, \wb}({-}, {-})$ and $\SMod_{w', w}({-}, {-})$ equip $\FSK_{w', w}$ with left- and right-closed structure respectively.
\end{proposition}

\begin{proof}
	By definition, to give a $(w', w)$-weak $\F$-multifunctor $M \colon (\fc S_2, \ldots, \fc S_n) \to \FMod_{\wbp, \wb}(\sk S_1, \sk C)$ is to give a  $(w', w)$-weak $\F$-multifunctor $M \colon (\fc S_2, \ldots, \fc S_n) \to \FFun_{\wbp, \wb}(\fc S_1, \fc C)$ such that each $M(S_2, \ldots, S_n) \colon \fc S_1 \to \fc C$ is an $\sk S_1$-model. Moreover, $M$ is furthermore a $(w', w)$-weak $(\sk S_2, \ldots, \sk S_n)$-multimodel if and only if each $M(S_2, \ldots, {-}_i, \ldots, S_n)(S_1) \colon \fc S_i \to \fc C$ is a model of $\sk S_i$. Via left-closure of $\FCAT$, this is precisely the definition of a  $(w', w)$-weak $\F$-multimodel. The isomorphism of \cref{FCAT-is-closed} defining the left-closed structure of $\FCAT$ trivially restricts on morphisms and 2-cells, because $\F$-categories of models are full sub-$\F$-categories of functor $\F$-categories, thus establishing left-closure. Right-closure follows by symmetric reasoning.
\end{proof}

Next, we show that the $(w', w)$-weak tensor product of $\F$-categories lifts to $\F$-sketches. To define the lifting, it will be helpful to observe that, for $\F$-categories $\fc S$ and $\fc T$, and for each object $T \in \fc T$, there is an $\F$-functor,
\begin{equation}\label{T-constant}
	\fc S \xto\iso \fc S \otimes_{w', w} \bbn 1 \xto{\fc S \otimes_{w', w} T} \fc S \otimes_{w', w} \fc T
\end{equation}
constant in the right-hand argument; and similarly, for each object $S \in \fc S$, an $\F$-functor
\begin{equation}\label{S-constant}
	\fc T \xto\iso \bbn 1 \otimes_{w', w} \fc T \xto{S \otimes_{w', w} \fc T} \fc S \otimes_{w', w} \fc T
\end{equation}
constant in the left-hand argument.

\begin{definition}[Tensor product of $\F$-sketches]
	\label{tensor-of-sketches}
	Let $w' \leq w \in \W$ be weaknesses, and let $\sk S$ and $\sk T$ be $\F$-sketches. Define an $\F$-sketch $\sk S \otimes_{w', w} \sk T$ by equipping $\fc S \otimes_{w', w} \fc T$ with \mbox{$\F$-sketch} structure. For each weighted cone $(W \colon \fc J \to \fc F, D \colon \fc J \to \fc S, X \in \fc S, \gamma \colon W \tto \fc S(X, D{-}))$ in $\sk S$, and each object $T \in \fc T$, we define a $W$-weighted cone in $\fc S \otimes_{w', w} \fc T$ by postcomposing \eqref{T-constant}; similarly, for each weighted cone $(W \colon \fc J \to \fc F, D \colon \fc J \to \fc T, X \in \fc T, \gamma \colon W \tto \fc T(X, D{-}))$ in $\sk T$, and each object $S \in \fc S$, we define a $W$-weighted cone in $\fc S \otimes_w \fc T$ by postcomposing \eqref{S-constant}. Define $\sk S \otimes_w \sk T \defeq \sk S \otimes_{s,w} \sk T$.
\end{definition}

\begin{theorem}
	Let $w' \leq w\ \in \W$ be weaknesses.
	\begin{enumerate}
		\item The tensor product $\otimes_{w', w}$ of $\F$-sketches exhibits the multicategory $\FSK_{w', w}$ as representable, and so equips the category $\FSK$ with closed monoidal structure $(\FSK, \otimes_{w', w}, \bbn 1^\sharp)$.
		\item Furthermore, the coreflection of \cref{FCAT-FSK-coreflection} is monoidal.
	\end{enumerate}
\end{theorem}

\begin{proof}
	For the first part, we proceed analogously to \cref{FCAT-is-representable}. It is evident that the $\F$-sketch $\bbn 1^\sharp$ forms a unit for $\FSK_{w', w}$.
	For the binary product, we observe that a $(w', w)$-natural $\F$-multimodel
	$(\sk S \otimes_{w', w} \sk T) \to \sk C$
	is a $(w', w)$-natural $\F$-multifunctor
	$(\fc S \otimes_{w', w} \fc T) \to \fc C$
	that restricts to a model in each argument. By \cref{FCAT-is-representable} and examination of the $\F$-sketch structure on $\fc S \otimes_{w', w} \fc T$, this is equivalently a $(w', w)$-natural $\F$-multifunctor
	$(\fc S, \fc T) \to \fc C$
	that restricts to a model in each argument, which is precisely a $(w', w)$-natural $\F$-multimodel
	$(\sk S, \sk T) \to \sk C$, from which representability follows.

	For the second part, observe that both functors are strict monoidal: the $\F$-category underlying an $\F$-sketch $\sk S \otimes_{w', w} \sk T$ is simply $\fc S \otimes_{w', w} \fc T$; whilst the $(w', w)$-weak tensor product of two free $\F$-sketches is also free by definition; that the counit is monoidal is trivial.
\end{proof}

\section{The symmetry of internalisation}
\label{multiple-perspectives}

In this section we use the theory of enhanced 2-sketches developed in the previous section to explain the multiple perspectives on two-dimensional structures observed in \cref{F-categorical-structures}. For instance, monoidal double categories can be viewed both as pseudomonoids in $\FDblCat s p$ and as pseudocategories in $\FMonCat s p$. Correspondences such as this one are a special case of the symmetry isomorphism \[\FMod_w(\sk S, \FMod_\wb(\sk T, \fc C)) \iso \FMod_\wb(\sk T, \FMod_w(\sk S, \fc C))\]
of \cref{Mod-skew-symmetry-tight} below, which will be the key to our applications, described in \cref{applications}.

\subsection{The main symmetry results}

At its heart, the phenomenon described above arises from the following consequence of the closed structure on the multicategory $\FSK_{w', w}$ of $\F$-sketches.

\begin{proposition}
	\label{Mod-skew-symmetry}
	Let $w' \leq w \in \W$ be weaknesses, and let $\sk S$, $\sk T$ and $\sk C$ be $\F$-sketches. There is an
	isomorphism of $\F$-sketches:
	\[\SMod_{w', w}(\sk S, \SMod_{\wbp, \wb}(\sk T, \sk C)) \iso \SMod_{\wbp, \wb}(\sk T, \SMod_{w', w}(\sk S, \sk C))\]
\end{proposition}

\begin{proof}
	This follows immediately from \cref{closed-multicategory-skew-symmetry}, since $\FSK_{w', w}$ is closed by \cref{FSK-is-closed}.
\end{proof}

However, there is a gap between this result and our intended applications, due to the fact that, so far, we lack an understanding of when the weighted cones in the $\F$-sketches appearing above, such as $\SMod_{w', w}(\sk S, \sk C)$, are actually weighted limit cones.
To understand the nature of limits in $\F$-categories of models, the following concept will be key.

\begin{definition}[Creation]
	Let $w' \leq w \in \W$ be weaknesses, and let $\Psi$ be a class of $\F$-weights. $\Psi$ is \emph{$(w', w)$-created} if, for each $\Psi$-limit $\F$-sketch $\sk S$ and each $\Psi$-complete $\F$-category $\fc C$, the $\F$-category $\FMod_{w', w}(\sk S, \fc C)$ is also $\Psi$-complete and the evaluation $\F$-functors $\ph(S) \colon \FMod_{w', w}(\sk S, \fc C) \to \fc C$, for each $S \in \fc S$, preserve and jointly reflect $\Psi$-limits.
\end{definition}

\begin{proposition}
	\label{realised}
	Let $w' \leq w \in \W$ be weaknesses, let $\sk S$ be a $\Psi$-limit $\F$-sketch, and let $\fc C$ be a small $\Psi$-complete $\F$-category, viewed as a $\Psi$-limit $\F$-theory. If $\Psi$ is $(w', w)$-created, then the $\F$-category $\FMod_{w', w}(\sk S, \fc C)$ is $\Psi$-complete and its associated $\Psi$-limit $\F$-sketch is $\SMod_{w', w}(\sk S, \fc C)$.
\end{proposition}

\begin{proof}
	The first claim is by definition. The second part follows from the fact that, by definition, the weighted cones in the $\F$-sketch $\SMod_{w', w}(\sk S, \fc C)$ are exactly those which are pointwise weighted cones in $\fc C$, which is to say pointwise $\Psi$-limiting cones.  Since $\Psi$ is $(w', w)$-created, these are exactly the $\Psi$-limit cones in $\FMod_{w', w}(\sk S, \fc C)$.
\end{proof}

We may consequently refine \cref{Mod-skew-symmetry} as follows.

\begin{theorem}
	\label{Mod-skew-symmetry-refined}
	Let $w' \leq w \in \W$ be weaknesses, and let $\Psi$ be a class of weights that is both $(w', w)$- and $(\wbp, \wb)$-created. Further let $\sk S$ and $\sk T$ be $\Psi$-limit $\F$-sketches, and let $\fc C$ be a $\Psi$-complete $\F$-category. There is an isomorphism of $\Psi$-complete $\F$-categories,
	\[\FMod_{w', w}(\sk S, \FMod_{\wbp, \wb}(\sk T, \fc C)) \iso \FMod_{\wbp, \wb}(\sk T, \FMod_{w', w}(\sk S, \fc C))\]
	where the objects on either side are the $\F$-functors sending the specified weighted cones to weighted limit cones.
\end{theorem}

\begin{proof}
	Immediate from \cref{Mod-skew-symmetry,realised}.
\end{proof}

Identifying the classes of $\F$-weights that are $(w', w)$-created is nontrivial in general and an interesting problem, which we expect to be closely connected to the main results of \cite{lack2012enhanced}. However our main applications concern the relatively simple case in which $w' = s$ and the limits are tight limits: namely, products, pullbacks and comma objects of tight morphisms. We show in \cref{created} that tight limits are $(s, w)$-created for all $w \in \W$. Therefore, we may immediately specialise \cref{Mod-skew-symmetry-refined} into the following form, which is the result that we will use in our applications.

\begin{theorem}
	\label{Mod-skew-symmetry-tight}
	Let $w \in \W$ be a weakness, and let $\Psi$ be a class of weights for tight limits. Further let $\sk S$ and $\sk T$ be $\Psi$-limit $\F$-sketches and let $\fc C$ be a $\Psi$-complete $\F$-category $\fc C$. Then there is an isomorphism of $\Psi$-complete $\F$-categories,
	\[\FMod_w(\sk S, \FMod_\wb(\sk T, \fc C)) \iso \FMod_\wb(\sk T, \FMod_w(\sk S, \fc C))\]
	where the objects on either side are the $\F$-functors sending the specified weighted cones to weighted limit cones.
\end{theorem}

\begin{proof}
	Immediate from \cref{Mod-skew-symmetry-refined,created}.
\end{proof}

We briefly note a similar result concerning tensor products of $\F$-sketches.

\begin{theorem}
	\label{Mod-tensor-hom-tight}
	Let $w \in \W$ be a weakness, and let $\Psi$ be a class of weights for tight limits. Further let $\sk S$ and $\sk T$ be $\Psi$-limit $\F$-sketches and let $\fc C$ be a $\Psi$-complete $\F$-category $\fc C$. Then there are isomorphisms of $\Psi$-complete $\F$-categories,
	\begin{align*}
		\FMod_w(\sk S \otimes_w \sk T, \fc C) & \iso \FMod_w(\sk S, \FMod_w(\sk T, \fc C)) \\
		\FMod_\wb(\sk S \otimes_w \sk T, \fc C) & \iso \FMod_\wb(\sk T, \FMod_\wb(\sk S, \fc C))
	\end{align*}
	where the objects on either side are the $\F$-functors sending the specified weighted cones to weighted limit cones. In particular, when $w \in \{ s, p \}$, there is an isomorphism:
	\[\FMod_w(\sk S, \FMod_w(\sk T, \fc C)) \iso \FMod_w(\sk S \otimes_w \sk T, \fc C) \iso \FMod_w(\sk T, \FMod_w(\sk S, \fc C))\]
\end{theorem}

\begin{proof}
	By \cref{closed-multicategory-tensor}, we have isomorphisms of $\F$-sketches as follows.
	\begin{align*}
		\SMod_w(\sk S \otimes_w \sk T, \fc C) & \iso \SMod_w(\sk S, \SMod_w(\sk T, \fc C)) \\
		\SMod_\wb(\sk S \otimes_w \sk T, \fc C) & \iso \SMod_\wb(\sk T, \SMod_\wb(\sk S, \fc C))
	\end{align*}
	Observe that $\sk S \otimes_w \sk T$ is also a $\Psi$-limit $\F$-sketch, since the weights for $\sk S \otimes_w \sk T$ are specified by the weights for $\sk S$ and $\sk T$. Since $\Psi$-limits are $(s, w)$-created by \cref{created}, by \cref{realised} the above isomorphisms become exactly the claimed isomorphisms.
\end{proof}

\subsection{Applications}
\label{applications}

By instantiating \cref{Mod-skew-symmetry-tight} with different $\F$-sketches and weaknesses, we obtain various symmetries of structure, both known and new. In each case, we take $\fc C$ to be the chordate limit $\F$-theory $\chord\Cat$ of categories, functors, and natural transformations.

\begin{example}
	\label{examples-of-symmetries}
	\begin{enumerate}
		\item Taking $\sk S$ to be the $\F$-sketch for pseudomonoids, $\sk T$ to be the \mbox{$\F$-sketch} for pseudocategories, and $w = p$, we recover the observation in \cref{pseudomonoid-examples,pseudocategory-examples} that a monoidal double category is equivalently a pseudomonoid in $\FDblCat s p$ and a pseudocategory in $\FMonCat s p$. By taking $\sk S$ to be the $\F$-sketch for \emph{left-skew} monoids, we obtain two new equivalent definitions of left-skew monoidal double categories.
		\item Taking $\sk S$ to be the $\F$-sketch for pseudomonoids, $\sk T$ to be the $\F$-sketch for pseudocategories, and $w = c$, we observe that a colax monoidal double category is equivalently a pseudomonoid in $\FDblCat s c$ and a pseudocategory in $\FMonCat s l$.
		\item Taking $\sk S$ to be the $\F$-sketch for cartesian pseudomonoids, $\sk T$ to be the $\F$-sketch for pseudocategories, and $w = p$, we observe that a cartesian double category is equivalently a cartesian pseudomonoid in $\FDblCat s p$ and a pseudocategory in $\FCartCat s p$. Taking $w = l$, we observe that a precartesian double category is equivalently a cartesian pseudomonoid in $\FDblCat s l$ and a pseudocategory in $\FCartCat s c$.
		\item Taking $\sk S$ to be the $\F$-sketch for fibrations, $\sk T$ to be the $\F$-sketch for pseudomonoids, and $w = p$, we recover the observation in \cref{pseudocategory-examples,fibration-examples} that a monoidal fibration is equivalently a pseudomonoid in $\FFib s p$ and a fibration in $\FMonCat s p$. By taking $\sk S$ to be the $\F$-sketch for \emph{left-skew} pseudomonoids, we obtain two new equivalent definitions of left-skew monoidal fibrations. Similar comments apply taking $\sk S$ to be the $\F$-sketch for discrete fibrations.
		\item Taking $\sk S$ to be the $\F$-sketch for fibrations, $\sk T$ to be the $\F$-sketch for pseudocategories, and $w = p$, we recover the observation in \cref{pseudocategory-examples,fibration-examples} that a double fibration is equivalently a pseudocategory in $\FFib s p$ and a fibration in $\FDblCat s p$. Similar comments apply taking $\sk S$ to be the $\F$-sketch for discrete fibrations.
		\item Taking $\sk S = \sk T$ to be the $\F$-sketch for pseudomonoids, and $w = l$, we recover the observation in \cref{pseudomonoid-examples} that a duoidal category is equivalently a pseudomonoid in $\FMonCat s l$ and a pseudomonoid in $\FMonCat s c$.
		\item Taking $\fc S$ to be an $\F$-category, and $\sk T$ to be the sketch for pseudomonoids, we observe that $\F$-functors into $\FMonCat{s}{w}$ are equivalently pseudomonoids in $\FFun_\wb(\fc S, \chord\Cat)$, recovering \cite[Proposition~4.4]{moeller2020monoidal}.
		\item \label{h-and-v-intercategories} Taking $\sk S = \sk T$ to be the $\F$-sketch for pseudocategories, and $w = l$, we recover \cite[Theorem~4.1(b \& c)]{grandis2015intercategories}, which states that horizontal intercategories are equivalent to vertical intercategories. We may also characterise the morphisms that arise.
		\begin{itemize}
			\item Loose lax natural transformations between models of $\sk S$ in $\FMod_c(\sk S, \chord\Cat)$ are equivalent to loose colax natural transformations between models of $\sk S$ in $\FMod_l(\sk S, \chord\Cat)$, which are the \emph{colax--lax} morphisms of intercategories~\cite[1241]{grandis2015intercategories}. We thus recover \cite[Theorem~6.2]{grandis2015intercategories}.
			\item Loose colax natural transformations between models of $\sk S$ in $\FMod_c(\sk S, \chord\Cat)$ are the \emph{colax--colax} morphisms \ibid.
			\item Loose lax natural transformations between models of $\sk S$ in $\FMod_l(\sk S, \chord\Cat)$ are the \emph{lax--lax} morphisms \ibid.
			\qedhere
		\end{itemize}
	\end{enumerate}
\end{example}

The final example above provides a conceptual explanation for the observation in \cite[\S7]{grandis2015intercategories} that there are three natural notions of morphism for intercategories. A similar phenomenon occurs for any structure obtained by considering $\sk S = \sk T$.


Finally, we give a couple of examples involving tensor products of $\F$-sketches (\cref{Mod-tensor-hom-tight}), leaving further combinations to the reader.

\begin{example}
	\begin{enumerate}
		\item Taking $\sk S$ to be the $\F$-sketch for pseudomonoids and $\sk T$ to be the $\F$-sketch for pseudocategories, we obtain an $\F$-sketch $\sk S \otimes_p \sk T$ for \emph{monoidal pseudocategories}, whose models in $\chord{\Cat}$ are monoidal double categories.
		\item Taking $\sk S = \sk T$ to be the $\F$-sketch for pseudocategories, and $w = l$, we obtain \mbox{$\F$-sketches} $\sk S \otimes_l \sk S$ and $\sk S \otimes_c \sk S$ for \emph{double pseudocategory objects} in the sense of \cite{grandis2015intercategories}. In both cases, the models in $\chord{\Cat}$ are intercategories, which, in conjunction with \cref{h-and-v-intercategories}, recovers \cite[Theorem~4.1]{grandis2015intercategories}. We may also characterise the morphisms that arise, with respect to \cref{h-and-v-intercategories}.
		\begin{itemize}
			\item Loose lax natural transformations between models of $\sk S \otimes_l \sk S$ or $\sk S \otimes_c \sk S$ are lax--lax morphisms.
			\item Loose colax natural transformations between models of $\sk S \otimes_l \sk S$ or $\sk S \otimes_c \sk S$ are colax--colax morphisms.
			\qedhere
		\end{itemize}
	\end{enumerate}
\end{example}

\begin{remark}
	For pseudo double categories with companions and conjoints, it is sometimes the case that structure on the double category induces structure on the underlying bicategory. For instance, a (braided/symmetric) monoidal double category with companions and conjoints has an associated (symmetric/braided) monoidal bicategory~\cites[Theorem~5.1]{shulman2010constructing}[Theorem~5.12]{hansen2019constructing}. We expect that the theory of $\F$-sketches may provide a way to generally prove such results.
\end{remark}

\section{From 2-sketches to enhanced 2-sketches}
\label{2-sketches}

In this section, we investigate the relationship between enhanced 2-sketches (\ie{} \mbox{$\F$-sketches}) and 2-sketches (\ie $\CAT$-sketches). Our first main result, \cref{free-F-sketch}, shows that it is possible to freely construct an $\F$-sketch from a 2-sketch, relative to a fixed class of $\F$-categorical weights. In particular, this construction recovers our three guiding examples of tight limit $\F$-sketches in \cref{examples-of-F-sketches} -- those for pseudomonoids, pseudocategories, and fibrations -- from their underlying 2-sketches.

We then specialise in \cref{algebraic-2-theories} to the algebraic setting, in which the defining class of limits is the powers. In particular, we revisit earlier work on algebraic 2-theories by \textcite{lack2007lawvere} and explain how it fits into the $\F$-categorical setting. Then, in \cref{flexible-limit-2-sketches}, we turn to the more general but subtler setting of flexible limit 2-sketches, showing that each cloven flexible limit 2-sketch naturally induces an $\F$-sketch with the same $\F$-category of models; this realises a proposal in \cite[\S9]{bourke2021accessible}.

\subsection{Universal enhancements of 2-sketches}
\label{(co)free-F-sketches}

\begin{definition}
	\label{2-sketch}
	A \emph{2-sketch} is a $\CAT$-sketch in the sense of \cref{sketches}. Denote by $\twoSK$ the 2-category of 2-sketches.
\end{definition}

Via change of base for enriched sketches (\cref{change-of-base}), every enhanced 2-sketch $\sk S$ has an underlying 2-sketch $\sk S_\lambda$, and this defines a 2-functor $\ph_\lambda \colon \FSK \to \twoSK$. Explicitly, $\ph_\lambda \colon \FSK \to \twoSK$ sends each $\F$-sketch $\sk S$ to its underlying 2-category $\fc S_\lambda$ of loose morphisms, together with the loose parts of the weights, diagrams, and weighted cones $\gamma_\lambda \colon W_\lambda \tto \fc S_\lambda(X,D_\lambda{-})$ appearing in $\sk S$ (see \cref{enhanced-2-limits-as-2-limits}).
Furthermore, for a fixed class of $\F$-weights $\Psi$, this forgetful 2-functor restricts by \cref{restricted-change-of-base} to a 2-functor $\ph_\lambda \colon \Psi\SK \to \Psi_\lambda\SK$, where $\Psi_\lambda$ comprises the loose parts of the $\F$-weights in $\Psi$.

For instance, if $\Psi$ is the class of tight limits, then $\ph_\lambda$ sends the tight limit $\F$-sketches for pseudomonoids, pseudocategories, and fibrations respectively to the 2-sketches for pseudomonoids, pseudocategories, and fibrations. This process forgets which morphisms are tight. We now describe a general construction which, in the above cases, recovers the tight morphisms from the underlying 2-sketches.

\begin{theorem}
	\label{free-F-sketch}
	Let $\Psi$ be a class of $\F$-weights. The forgetful functor ${\ph_\lambda \colon \Psi\SK_0 \to \Psi_\lambda\SK_0}$ admits a coreflective left adjoint $\ph^\Psi \colon \Psi_\lambda\SK_0 \to \Psi\SK_0$.
\end{theorem}

\begin{proof}
	Let $\sk S$ be a $\Psi_\lambda$-limit 2-sketch. We construct a $\Psi$-limit $\F$-sketch $\sk S^\Psi$ as follows. The underlying 2-category of loose morphisms $(\fc S^\Psi)_\lambda$ is given by $\tc S$. The tight morphisms are generated under identities and composition by the following.

	For each weight $W \colon \fc J \to \fc F$ in $\Psi$, diagram $D \colon \fc J_\lambda \to \tc S$, and weighted cone ${\gamma \colon W_\lambda \tto \tc S(X, D{-})}$ in $\sk S$,
	\begin{enumerate}
		\item if a morphism $f \colon A \lto B \in \fc J$ is tight, the morphism $D(f) \colon D(A) \lto D(B)$ is tight in $\fc S^\Psi$;
		\item given an object $J \in \fc J$ and an object $Y \in W_\tau(J)$, the morphism $\gamma_{J, Y} \colon X \lto D(J)$ is tight in $\fc S^\Psi$.
	\end{enumerate}
	With this choice of tight morphisms for $\fc S^\Psi$, $D$ becomes an $\F$-functor $\fc J \to \fc S^\Psi$ and $\gamma$ becomes a $W$-weighted cone. This determines the choice of cones for $\sk S^\Psi$. Note that we have $(\sk S^\Psi)_\lambda = \sk S$ by definition.

	Let $S \colon \sk S \to \sk S'$ be a morphism of $\Psi_\lambda$-limit 2-sketches, defining a 2-functor ${S^\Psi \colon (\fc S^\Psi)_\lambda \to ((\fc S')^\Psi)_\lambda}$. We must verify that the generating classes of tight morphisms (1) \& (2) in $\sk S^\Psi$ are sent to tight morphisms in $(\sk S')^\Psi$. As above, let $W \colon \fc J \to \fc F$ in $\Psi$, $D \colon \fc J_\lambda \to \tc S$, and $\gamma \colon W_\lambda \tto \tc S(X, D{-})$ in $\sk S$. Since $S$ is a morphism of sketches, $S_{X, D{-}} \c \gamma \colon W_\lambda \tto \ct S'(SX, SD{-})$ is a cone in $\sk S'$. This implies that
	\begin{enumerate}
		\item for each tight morphism $f \colon A \to B$ in $\fc J$, the morphism $SD(f) \colon SD(A) \lto SD(B)$ is tight in $(\sk S')^\Psi$;
		\item for each $J \in \fc J$ and $Y \in W_\tau(J)$, the morphism $S(\gamma_{J, Y}) \colon SX \lto SD(J)$ is tight in $(\sk S')^\Psi$.
	\end{enumerate}

	It remains to show that the functor $\ph^\Psi \colon \Psi_\lambda\SK_0 \to \Psi\SK_0$ defines a left adjoint to the forgetful functor. Let $\sk T$ be a $\Psi$-limit $\F$-sketch. The identity 2-functor on $\fc T_\lambda$ defines an $\F$-functor $(\fc T_\lambda)^\Psi \to \fc T$, since the $\CAT$-enriched cones in $\sk T_\lambda$ were induced by $\F$-enriched cones in $\sk T$, and so the morphisms determined to be tight in $(\fc T_\lambda)^\Psi$ by the procedure above must have been tight in $\fc T$. For the same reason, this $\F$-functor preserves the $\F$-enriched cones in $(\fc T_\lambda)^\Psi$, and consequently defines a morphism of sketches $\varepsilon_{\sk T} \colon (\sk T_\lambda)^\Psi \to \sk T$. The triangle identities follow from the fact that $\varepsilon_{\sk T}$ is induced by the identity functor.
\end{proof}

\begin{remark}[Models for a free $\F$-sketch]
	\label{models-of-free-F-sketches}
	Given the explicit construction of the $\F$-sketch $\sk S^\Psi$ from a 2-sketch $\sk S$, we may give an explicit description of its $\F$-categories of models. To this end, let us refer to the generating classes of tight morphisms (1) and (2) in $\sk S^\Psi$ described in the proof of \cref{free-F-sketch} as \emph{tight diagram morphisms} and \emph{tight cone projections} respectively.

	Given a $\Psi$-limit $\F$-sketch $\sk C$ and $w' \leq w \in \W$ a pair of weaknesses, the $\F$-category $\FMod_{w', w}(\sk S^\Psi, \sk C)$ is described as follows. Its objects are models of $\sk S$ in $\sk C_\lambda$ that send tight diagram morphisms and tight cone projections to tight morphisms in $\fc C$. Given a pair of such models $M, N \colon \sk S \to \sk C_\lambda$, a loose morphism $M \lto N$ is a $w$-natural transformation $\phi \colon M \tto N$ whose components at tight diagram morphisms and tight cone projections are $w'$-natural. A morphism is tight when it is $w'$-natural and has tight components. Finally, its 2-cells are simply modifications.
\end{remark}

\begin{example}
	\label{examples-of-cofree-F-sketches}
	Let $\Psi$ be a class of $\CAT$-weights and let $\chord\Psi$ be the corresponding class of tight $\F$-weights (\cref{tight-limits}). Then $(\chord\Psi)_\lambda = \Psi$ and so we obtain an adjunction $\ph^{\chord\Psi} \colon \Psi\SK_0 \rightleftarrows \chord\Psi\SK_0 \cocolon \ph_\lambda$.
	\begin{itemize}
		\item The free tight product $\F$-sketch on the product 2-sketch for pseudomonoids is precisely the $\F$-sketch for pseudomonoids (\cref{sketch-pseudomonoid}).
		\item The free tight pullback $\F$-sketch on the pullback 2-sketch for pseudocategories is precisely the $\F$-sketch for pseudocategories (\cref{sketch-pseudocategory}).
		\item The free tight comma object $\F$-sketch on the comma object 2-sketch for fibrations is precisely the $\F$-sketch for fibrations (\cref{sketch-fibration}).
		\qedhere
	\end{itemize}
\end{example}

\begin{remark}
	The choice of tight morphisms in \cref{free-F-sketch} is reminiscent of \citeauthor{bastiani1974multiple}'s~\cite{bastiani1974multiple} approach to lax morphisms of sketches mentioned in \cref{lax-morphisms-via-2-categories}, since the morphisms in $\sk S^\Psi$ over which limits are taken (\ie the tight diagram morphisms) are required to be tight. (However, in general, these are not the only morphisms required to be tight, as the class of weights $\Psi$ also determines tightness.)
\end{remark}

To obtain a right adjoint to the forgetful 2-functor $\ph_\lambda \colon \Psi\SK \to \Psi_\lambda\SK$, we must impose the mild additional assumption that the class $\Psi$ of weights contains the induced class of tight weights.

\begin{theorem}
	\label{cofree-F-sketch}
	Let $\Psi$ be a class of $\F$-weights for which $\chord{(\Psi_\lambda)} \subseteq \Psi$. The forgetful 2-functor $\ph_\lambda \colon \Psi\SK \to \Psi_\lambda\SK$ admits a reflective right adjoint $\chord\ph \colon \Psi_\lambda\SK \to \Psi\SK$.
\end{theorem}

\begin{proof}
	By change of base for enriched sketches (\cref{change-of-base}), the reflective 2-adjunction $\ph_\lambda \colon \FCAT \rightleftarrows \twoCAT \cocolon \chord\ph$ lifts to a reflective 2-adjunction $\ph_\lambda \colon \FSK \rightleftarrows \twoSK \cocolon \chord\ph$ between the 2-category of $\F$-sketches and the 2-category of 2-sketches. Explicitly, the right adjoint $\chord\ph \colon \twoCAT \to \FCAT$ assigns to each $\CAT$-weighted cone $\gamma \colon W \tto {\tc S}(X, D{-})$ in a 2-sketch $\sk S$, the $\F$-weighted cone $\chord{\gamma} \colon \chord W \tto \chord{\tc S}(X, \chord{D}{-})$ with the same components as~$\gamma$.

	Consequently, by \cref{restricted-change-of-base}, the left adjoint restricts to $\Psi\SK \to \Psi_\lambda\SK$, and the right adjoint restricts to $\Psi_\lambda\SK \to \chord{(\Psi_\lambda)}\SK$. Since $\chord{(\Psi_\lambda)} \subseteq \Psi$ by assumption, $\chord{(\Psi_\lambda)}\SK$ is a full sub-2-category of $\Psi\SK$, and so the 2-adjunction above restricts as required.
\end{proof}

\begin{remark}[Models for a cofree $\F$-sketch]
	\label{models-of-cofree-F-sketches}
	Given a $\Psi$-limit $\F$-sketch $\sk C$ and $w' \leq w \in \W$ a pair of weaknesses, the $\F$-category $\FMod_{w', w}(\chord{\sk S}, \sk C)$ is described as follows. Its objects are models of $\sk S$ in $\sk C_\lambda$ that factor through $\fc C_\tau$. Given a pair of such models $M, N \colon \sk S \to \sk C_\lambda$, a loose morphism $M \lto N$ is a $w'$-natural transformation $\phi \colon M \tto N$. A morphism is tight when it has tight components. Finally, its 2-cells are simply modifications.
\end{remark}

\begin{example}
	Taking $\Psi = \varnothing$, we obtain from \cref{free-F-sketch,cofree-F-sketch} the adjoint triple between $\FCAT_0$ and $\twoCAT_0$ described in \cref{chord-and-inchord}.
\end{example}

\begin{example}
	\label{chordate-F-sketch}
	Taking $\Psi$ to be the class of all $\F$-weights, we have that $\Psi_\lambda$ is the class of all $\CAT$-weights, since for every $\CAT$-weight $W \colon \tc J \to \CAT$, we have $W = (\chord W)_\lambda$. Consequently, we obtain from \cref{free-F-sketch,cofree-F-sketch} the adjoint triple below.
	\[\begin{tikzcd}
		{\FSK_0} && {\twoSK_0}
		\arrow[""{name=0, anchor=center, inner sep=0}, "{(-)_\lambda}"{description}, from=1-1, to=1-3]
		\arrow[""{name=1, anchor=center, inner sep=0}, "{\ph^\allweights}"', shift right=3, curve={height=6pt}, from=1-3, to=1-1]
		\arrow[""{name=2, anchor=center, inner sep=0}, "{\chord\ph}", shift left=3, curve={height=-6pt}, from=1-3, to=1-1]
		\arrow["\dashv"{anchor=center, rotate=-90}, draw=none, from=0, to=2]
		\arrow["\dashv"{anchor=center, rotate=-90}, draw=none, from=1, to=0]
	\end{tikzcd}\]

	Note that, given a 2-sketch $\sk S$, the underlying $\F$-category of $\sk S^\allweights$ need not be inchordate, since the left adjoint $\ph^\allweights$ typically introduces new tight morphisms.
	On the other hand, the underlying $\F$-category of $\chord{\sk S}$ is the chordate $\F$-category $\chord{\tc S}$; our notation for the right adjoint is chosen to reflect this.
\end{example}

Given 2-sketches $\sk S$ and $\sk C$, the following clarifies the relationship between models of $\sk S$ in $\sk C$ and models of the free and cofree $\F$-sketches induced by \cref{free-F-sketch,cofree-F-sketch}.

\begin{proposition}
	\label{models-for-(co)free-sketches}
	Let $\Psi$ be a class of $\F$-weights and let $\sk S$ and $\sk C$ be $\Psi_\lambda$-limit 2-sketches. For each pair $w' \leq w \in \W$ of weaknesses, there is a 2-natural isomorphism of 2-categories:
	\begin{equation}\label{models-for-cofree-sketches}
		(\FMod_{w', w}(\sk S^\Psi, \sk C^\Psi))_\lambda \iso (\FMod_{w', w}(\sk S^\Psi, \chord{\sk C}))_\lambda
	\end{equation}
	Furthermore, in the case $w = w'$, there are 2-natural isomorphisms of 2-categories:
	\begin{equation}\label{models-for-2-sketches}
		(\FMod_{w, w}(\sk S^\Psi, \sk C^\Psi))_\lambda \iso (\FMod_{w, w}(\sk S^\Psi, \chord{\sk C}))_\lambda \iso (\FMod_{w, w}(\chord{\sk S}, \chord{\sk C}))_\lambda
	\end{equation}
	in which the objects are in bijection with 2-sketch morphisms $\sk S \to \sk C$, the 1-cells are in bijection with $w$-natural transformations between underlying 2-functors, and the 2-cells are modifications.
\end{proposition}

\begin{proof}
	Each isomorphism is immediate from the explicit description of the $\F$-categories of models for free and cofree $\F$-sketches in \cref{models-of-free-F-sketches,models-of-cofree-F-sketches}.
\end{proof}

\subsection{Algebraic 2-theories}
\label{algebraic-2-theories}

A particular setting of interest is that of two-dimensional universal algebra, in which the relevant class of limits are the finite powers (or, more generally, the $\kappa$-ary powers for a regular cardinal $\kappa$). Power 2-sketches capture many examples of categorical structure, such as monoidal categories and categories with limits of a given shape.

Our goal in this section is to show that our $\F$-categorical approach to finite power 2-sketches subsumes earlier work on the topic by \textcite{lack2007lawvere}, and to revisit a curious result of \textcite{power1999enriched}.

Two-dimensional algebraic theories, which we shall call \emph{algebraic 2-theories} (\cref{algebraic-2-theory}), were introduced and studied by \textcite{power1999enriched}. To recall the definition, let $\FPCat$ denote a skeleton of the full sub-2-category of $\Cat$ spanned by the small finitely presentable categories. Since each object of $\FPCat$ is a finite copower of $1$, its opposite 2-category $\FPCat\op$ admits finite powers.

\begin{definition}
	\label{algebraic-2-theory}
	An \emph{algebraic 2-theory} $S$ comprises a small 2-category $\tc S$ equipped with a finite power-preserving \ioo{} 2-functor $S \colon \FPCat\op \to \tc S$.
\end{definition}

A \emph{model} of $S$ in a 2-category $\tc C$ is a finite power-preserving 2-functor from $\tc S$ to $\tc C$. \textcite[\S5]{power1999enriched} defined a pseudo morphism of models of an algebraic 2-theory to simply be a pseudo natural transformation. However, in later unpublished work with Lack~\cite{lack2007lawvere}, the notion of pseudo morphism was refined, as we now recall.

\begin{definition}
	\label{model-for-an-algebraic-2-theory}
	Let $S \colon \FPCat\op \to \tc S$ be an algebraic 2-theory and let $M, N \colon \tc S \to \tc C$ be models thereof. Let $w \in \W$ be a weakness. A \emph{$w$-weak morphism} from $M$ to $N$ is a $w$-natural transformation $\phi \colon M \tto T$ for which $\phi_S \colon M S \to N S \colon \FPCat\op \to \tc C$ is 2-natural. Denote by $\Mod_w(S, \tc C)$ the 2-category whose objects are models of $S$ in $\tc C$, whose 1-cells are $w$-weak morphisms, and whose 2-cells are modifications.
\end{definition}

We explain how this fits within our approach. Each algebraic theory $S \colon \FPCat\op \to \tc S$ equips $\tc S$ with specified finite powers of $\b1$ and thus describes a finite power 2-sketch structure $\sk S$ on $\tc S$. Denoting by $\Pi$ the class of $\F$-weights for tight finite powers, from \cref{free-F-sketch} we obtain a tight finite power $\F$-sketch $\sk S^\Pi$. We show that $\sk S^\Pi$ is a faithful representation of $S$, in that its models and their morphisms are precisely those described in \cref{model-for-an-algebraic-2-theory}.

\begin{proposition}
	\label{models-of-algebraic-2-theory-are-models-of-F-sketch}
	Let $S \colon \FPCat\op \to \tc S$ be an algebraic 2-theory and let $\tc C$ be a 2-category. For each weakness $w \in \W$, there is an isomorphism of 2-categories:
	\[\Mod_w(S, \tc C) \iso (\FMod_w(\sk S^\Pi, \chord{\sk C}))_\lambda\]
	in which $\sk C$ is the finite power 2-sketch whose cones are given by all finite power cones in~$\tc C$.
\end{proposition}

\begin{proof}
	We start by fixing some notation. Each finitely presentable category $\ct X \in \FPCat$ is the copower $\b 1\copow \ct X$ of $\b 1$ by $\ct X$, whose copower cocone $\copi \colon \ct X \tto \FPCat(\b 1, \ct X)$ sends each object $X \in \ct X$ to the functor $\copi_X \colon \b 1 \to \ct X$ identifying it. Consequently, it is the power cones $S_{\tc X, \b 1} \circ \copi\op \colon \tc X \tto \tc S(S(\tc X), S(\b 1))$ that define the finite power 2-sketch $\sk S$.

	We treat the case $w = l$, the others being identical in form. By \cref{models-of-free-F-sketches}, the 2-category $(\FMod_l(\sk S^\Pi, \chord{\sk C}))_\lambda$ has as objects the models of $S$ in $\tc C$. A 1-cell $\phi \colon M \rightsquigarrow N$ therein is a lax natural transformation whose components are identities at each power projection $\pi_X \colon S(\tc X) \to S(\b 1)$ for $X \in \tc X$. We must show that every such transformation defines a morphism of $\Mod_l(S, \tc C)$, \ie{} that each component $\phi_{S(F)}$ is the identity for each functor $F \colon \tc X \to \tc Y \in \FPCat$; the converse follows trivially, since the power cones are in the image of $S$.

	Since $\phi$ is a lax natural transformation, we have the following equality of 2-cells, for each $X \in \ct X$, so that the 2-cell on the left is an identity.
	\[
	\begin{tikzcd}[column sep=large]
		{M(S(\ct Y))} & {N(S(\ct Y))} \\
		{M(S(\ct X))} & {N(S(\ct X))} \\
		{M(S(\b1))} & {N(S(\b1))}
		\arrow["{\phi_{S(\ct Y)}}", from=1-1, to=1-2]
		\arrow["{M(S(F))}"', from=1-1, to=2-1]
		\arrow["{\phi_{S(F)}}"', shorten <=9pt, shorten >=9pt, Rightarrow, from=1-2, to=2-1]
		\arrow["{N(S(F))}", from=1-2, to=2-2]
		\arrow["{\phi_{S(\ct X)}}"{description}, from=2-1, to=2-2]
		\arrow[""{name=0, anchor=center, inner sep=0}, "{M(S(\pi_X))}"', from=2-1, to=3-1]
		\arrow[""{name=1, anchor=center, inner sep=0}, "{N(S(\pi_X))}", from=2-2, to=3-2]
		\arrow["{\phi_{S(\b1)}}"', from=3-1, to=3-2]
		\arrow["{=}"{description}, draw=none, from=0, to=1]
	\end{tikzcd}
	\quad
	=
	\quad
	\begin{tikzcd}[column sep=large]
		{M(S(\ct Y))} & {N(S(\ct Y))} \\
		\\
		{M(S(\b1))} & {N(S(\b1))}
		\arrow["{\phi_{S(\ct Y)}}", from=1-1, to=1-2]
		\arrow[""{name=0, anchor=center, inner sep=0}, "{M(S(\pi_{F(X)}))}"', from=1-1, to=3-1]
		\arrow[""{name=1, anchor=center, inner sep=0}, "{N(S(\pi_{F(X)}))}", from=1-2, to=3-2]
		\arrow["{\phi_{S(\b1)}}"', from=3-1, to=3-2]
		\arrow["{=}"{description}, draw=none, from=0, to=1]
	\end{tikzcd}
	\]
	It follows from the 2-dimensional universal property of powers that the projections $M(S(\pi_X))$, for each $X \in \ct X$, jointly reflect identity 2-cells. Consequently, $\phi_{S(F)}$ is also an identity, as required. Thus, there is a bijection between the 1-cells in both 2-categories. Finally, the 2-cells in both 2-categories coincide by definition.
\end{proof}

As mentioned above, in \citeauthor{power1999enriched}'s earlier work~\cite[\S5]{power1999enriched}, the notion of pseudo morphism of models was simply that of a pseudonatural transformation, with no further conditions.
In light of our earlier discussion, one might suppose this weaker notion of morphism to be inadequate. However, there is in fact a biequivalence between $\Mod_w(S, \tc C)$ and the 2-category of models, $w$-natural transformations, and modifications, justifying \citeauthor{power1999enriched}'s weaker choice; this observation essentially appears as \cite[Theorem~5.3]{power1999enriched}.

We now explain this biequivalence directly and show that it is peculiar to algebraic 2-theories: for most $\F$-sketches, it is not the case that loose $(s, p)$-natural transformations are equivalent to loose $(p, p)$-natural transformations. The following lemma is the key to establishing the biequivalence for algebraic 2-theories.

\begin{lemma}
	\label{strictification-for-morphisms-of-power-theories}
	Let $S \colon \FPCat\op \to \tc S$ be an algebraic 2-theory, let $\fc C$ be an $\F$-category, and let $M, N \colon \sk S^\Pi \to \fc C$ be models. For each weakness $w \ge p$, every loose $(p, w)$-natural transformation $\phi \colon M \tto N$ is isomorphic to a loose $(s, w)$-natural transformation $M \tto N$.
\end{lemma}

\begin{proof}
	We use the same notation as in the proof of \cref{models-of-algebraic-2-theory-are-models-of-F-sketch}, but elide the action of $S$. Define $\varphi_{\b1} \defeq \phi_{\b1} \colon M(\b1) \lto N(\b1)$. For each $\ct X \in \FPCat$, there exists a unique loose morphism $\varphi_{\ct X} \colon M(\ct X) \to N(\ct X)$ making the leftmost and rightmost squares in the following diagram commute for each $f \colon A \to B$ in $\ct X$, by the universal property of the power. (Note in particular that this choice for $\varphi_{\ct X}$ agrees with our choice for $\varphi_{\b1}$ when $\ct X = \b1$, since in that case the unique projection is the identity.)
	\[\begin{tikzcd}
		{M(\ct X)} &&& {N(\ct X)} \\
		\\
		{M(\b1)} &&& {N(\b1)}
		\arrow["{\varphi_{\ct X}}", squiggly, from=1-1, to=1-4]
		\arrow[""{name=0, anchor=center, inner sep=0}, "{M(\pi_A)}"', shift right=3, curve={height=6pt}, from=1-1, to=3-1]
		\arrow[""{name=1, anchor=center, inner sep=0}, "{M(\pi_B)}", shift left=3, curve={height=-6pt}, from=1-1, to=3-1]
		\arrow[""{name=2, anchor=center, inner sep=0}, "{N(\pi_A)}"', shift right=3, curve={height=6pt}, from=1-4, to=3-4]
		\arrow[""{name=3, anchor=center, inner sep=0}, "{N(\pi_B)}", shift left=3, curve={height=-6pt}, from=1-4, to=3-4]
		\arrow["{\varphi_{\b1}}"', squiggly, from=3-1, to=3-4]
		\arrow["{M(\pi_f)}", shorten <=5pt, shorten >=5pt, Rightarrow, from=0, to=1]
		\arrow["{N(\pi_f)}", shorten <=5pt, shorten >=5pt, Rightarrow, from=2, to=3]
	\end{tikzcd}\]
	We shall show that $\varphi_{\ct X} \colon M(\ct X) \to N(\ct X)$ extends to a loose $(s, w)$-natural transformation isomorphic to $\phi$.
	We define an invertible 2-cell as follows (in which $\phi_{\pi_A}$ is invertible since $\pi_A$ is tight).
	\begin{equation}
	\label{upsilon-def}
	\upsilon_A^{\ct X} \defeq
	\begin{tikzcd}
		& {M(\ct X)} \\
		{N(\ct X)} & {M(\b1)} & {N(\ct X)} \\
		& {N(\b1)}
		\arrow["{\varphi_{\ct X}}"', squiggly, from=1-2, to=2-1]
		\arrow["{M(\pi_A)}"{description}, from=1-2, to=2-2]
		\arrow["{\phi_{\ct X}}", squiggly, from=1-2, to=2-3]
		\arrow["{N(\pi_A)}"', from=2-1, to=3-2]
		\arrow["{=}"{description}, draw=none, from=2-2, to=2-1]
		\arrow["{\phi\inv_{\pi_A}}", shorten <=6pt, shorten >=6pt, Rightarrow, from=2-2, to=2-3]
		\arrow["{\varphi_{\b1}}"{description}, squiggly, from=2-2, to=3-2]
		\arrow["{N(\pi_A)}", from=2-3, to=3-2]
	\end{tikzcd}
	\end{equation}
	This satisfies
	\[
	\begin{tikzcd}[column sep=small]
		& {M(\ct X)} \\
		{N(\ct X)} && {N(\ct X)} \\
		& {N(\b1)}
		\arrow["{\varphi_{\ct X}}"', squiggly, from=1-2, to=2-1]
		\arrow["{\phi_{\ct X}}", squiggly, from=1-2, to=2-3]
		\arrow["{\upsilon_B^{\ct X}}", shorten <=24pt, shorten >=24pt, Rightarrow, from=2-1, to=2-3]
		\arrow[""{name=0, anchor=center, inner sep=0}, "{N(\pi_B)}", from=2-1, to=3-2]
		\arrow[""{name=1, anchor=center, inner sep=0}, "{N(\pi_A)}"', curve={height=45pt,pos=.2}, from=2-1, to=3-2]
		\arrow["{N(\pi_B)}", from=2-3, to=3-2]
		\arrow["{N(\pi_f)}"{description}, shorten <=4pt, shorten >=4pt, Rightarrow, from=1, to=0]
	\end{tikzcd}
	\qquad = \qquad
	\begin{tikzcd}[column sep=small]
		& {M(\ct X)} \\
		{N(\ct X)} && {N(\ct X)} \\
		& {N(\b1)}
		\arrow["{\varphi_{\ct X}}"', squiggly, from=1-2, to=2-1]
		\arrow["{\phi_{\ct X}}", squiggly, from=1-2, to=2-3]
		\arrow["{\upsilon_A^{\ct X}}", shorten <=24pt, shorten >=24pt, Rightarrow, from=2-1, to=2-3]
		\arrow["{N(\pi_A)}"', from=2-1, to=3-2]
		\arrow[""{name=0, anchor=center, inner sep=0}, "{N(\pi_A)}"', from=2-3, to=3-2]
		\arrow[""{name=1, anchor=center, inner sep=0}, "{N(\pi_B)}", curve={height=-45pt,pos=.2}, from=2-3, to=3-2]
		\arrow["{N(\pi_f)}"{description}, shorten <=4pt, shorten >=4pt, Rightarrow, from=0, to=1]
	\end{tikzcd}
	\]
	so, by the two-dimensional universal property of the power $N(\ct X)$, we obtain a unique 2-cell $\upsilon^{\ct X} \colon \varphi_{\ct X} \tto \phi_{\ct X}$ for which the following equation holds, and which is furthermore invertible.
	\begin{equation}
	\label{upsilon-prop}
	\upsilon_A^{\ct X} \quad =
	\begin{tikzcd}[column sep=huge]
		{M(\ct X)} & {N(\ct X)} & {N(\b1)}
		\arrow[""{name=0, anchor=center, inner sep=0}, "{\varphi_{\ct X}}", curve={height=-18pt}, squiggly, from=1-1, to=1-2]
		\arrow[""{name=1, anchor=center, inner sep=0}, "{\phi_{\ct X}}"', curve={height=18pt}, squiggly, from=1-1, to=1-2]
		\arrow["{N(\pi_A)}"{description}, from=1-2, to=1-3]
		\arrow["{\upsilon^{\ct X}}"', shorten <=5pt, shorten >=5pt, Rightarrow, dashed, from=0, to=1]
	\end{tikzcd}
	\end{equation}
	By transport of structure along the isomorphism, $\varphi$ uniquely obtains the structure of a loose $(p, w)$-natural transformation for which $\upsilon \colon \varphi \iso \phi$ is a modification.

	Finally, we must verify that $\varphi$ is actually $(s, w)$-natural. For each object $A \in \ct X$, we have that
	\[
	\varphi_{\pi_A} \quad = \quad
	\begin{tikzcd}
		{M(\ct X)} && {N(\ct X)} \\
		\\
		{M(\b1)} && {N(\b1)}
		\arrow[""{name=0, anchor=center, inner sep=0}, "{\varphi_{\ct X}}", curve={height=-12pt}, squiggly, from=1-1, to=1-3]
		\arrow[""{name=1, anchor=center, inner sep=0}, "{\phi_{\ct X}}"{description}, curve={height=12pt}, squiggly, from=1-1, to=1-3]
		\arrow["{M(\pi_A)}"', from=1-1, to=3-1]
		\arrow["{\phi_{\pi_A}}"', shorten <=25pt, shorten >=25pt, Rightarrow, from=1-3, to=3-1]
		\arrow["{N(\pi_A)}", from=1-3, to=3-3]
		\arrow[""{name=2, anchor=center, inner sep=0}, "{\phi_{\b1}}", curve={height=-12pt}, squiggly, from=3-1, to=3-3]
		\arrow[""{name=3, anchor=center, inner sep=0}, "{\varphi_{\b1}}"', curve={height=12pt}, squiggly, from=3-1, to=3-3]
		\arrow["{\upsilon^{\ct X}}"', shorten <=3pt, shorten >=3pt, Rightarrow, from=0, to=1]
		\arrow["{=}"{description}, draw=none, from=2, to=3]
	\end{tikzcd}
	\]
	which is equal to the identity 2-cell by \eqref{upsilon-def} and \eqref{upsilon-prop}.
\end{proof}

\begin{theorem}
	\label{power-biequivalence}
	Let $S \colon \FPCat\op \to \tc C$ be an algebraic 2-theory, and let $\fc C$ be an $\F$-category. For each weakness $w \ge p$, the inclusion
	\[\FMod_{s, w}(\sk S^\Pi, \fc C) \to \FMod_{p, w}(\sk S^\Pi, \fc C)\]
	induces a biequivalence between the underlying 2-categories of loose morphisms, where $\sk S^\Pi$ is the tight finite power $\F$-sketch induced by $S$.
\end{theorem}

\begin{proof}
	By \cref{strictification-for-morphisms-of-power-theories}, the inclusion is essentially surjective on loose morphisms. Furthermore, if a loose $(p, w)$-natural transformation $\phi \colon M \tto N$ is tight, then the isomorphic $(s, w)$-natural transformation $\varphi \colon M \tto N$ is also tight, because the tight morphisms in $\sk S^\Psi$ are generated by the cone projections for each power, and we know from \cref{strictification-for-morphisms-of-power-theories} that $\varphi$ restricts to the identity on these.
\end{proof}

The situation of \cref{power-biequivalence} is far from being the norm. For instance, if $\sk S$ is an $\F$-sketch that contains nontrivial tight morphisms, there is no reason to expect that every loose $(p, p)$-natural transformation is equivalent to a loose $(s, p)$-natural transformation, because the tight morphisms in $\sk S$ may not be respected.

\begin{example}
	\label{involution-counterexample}
	Let $\fc U$ be the $\F$-category generated from a single object $*$ with a tight involution, \ie an endomorphism $i \colon * \to *$ such that $i \c i = 1_*$. We may equip $\fc U$ with the structure of an $\F$-sketch $\sk U$ in which $i$ exhibits a cone for the diagram $\b1 \to \fc U$ picking out the unique object $*$ (\ie a unary product cone). A model for $\fc U$ in $\chord\Cat$ is a category equipped with an involution (note that every involution is a limiting cone in $\chord\Cat$).

	Now let $M \colon \fc U \to \chord\Cat$ be the model picking out the identity morphism on the terminal 2-category $\b1$ and let $N \colon \fc U \to \chord\Cat$ be the model picking out the involution $\tx{swap} \colon \b I \to \b I$, on the free-standing isomorphism $\b I \defeq \{ 0 \iso 1 \}$, that swaps $0$ and $1$. There are two functors $\b1 \to \b I$, which pick out $0$ and $1$ respectively, and these define $(p, p)$-morphisms from $M$ to $N$, since $0 \iso 1$.
	\[\begin{tikzcd}
		{\b 1} & {\b I} \\
		\b1 & {\b I}
		\arrow["{\phi_*}", from=1-1, to=1-2]
		\arrow[""{name=0, anchor=center, inner sep=0}, "{M(i) = 1_{\b 1}}"', from=1-1, to=2-1]
		\arrow[""{name=1, anchor=center, inner sep=0}, "{\tx{swap} = N(i)}", from=1-2, to=2-2]
		\arrow["{\phi_*}"', from=2-1, to=2-2]
		\arrow["\iso"{description}, draw=none, from=0, to=1]
	\end{tikzcd}\]
	However, there exist no $(s, p)$-morphisms from $M$ to $N$: such a morphism would necessarily make the square above commute, but there exist no such functors $\phi_*$. Consequently, the inclusion
	\[\FMod_{s, p}(\sk U, \chord\Cat) \to \FMod_{p, p}(\sk U, \chord\Cat)\]
	is not a biequivalence on 2-categories of loose morphisms.

	Note that $\sk U$ is a free $\F$-sketch with respect to a class of tight weights (\cref{examples-of-cofree-F-sketches}), so this shows that being free is not enough to guarantee that pseudo morphisms can be strictified in the same way as \cref{power-biequivalence}.
\end{example}

\begin{remark}
	\Cref{involution-counterexample} also demonstrates that the free $\F$-sketch construction of \cref{free-F-sketch} does not preserve Morita equivalence of sketches, in the following sense. Let $\sk U_1 \defeq \sk U_\lambda$, and let $\sk U_2$ be the 2-sketch with the same underlying 2-category as $\sk U_1$ but with no cones. These two 2-sketches have the same 2-category of models in $\Cat$, since the cone preservation condition for $\sk U_1$ is trivial. However, letting $\Psi \defeq \{ \Delta(\b1) \colon \b1 \to \fc F \}$ be the class of weights for  tight unary products, we have that $(\sk U_1)^\Psi$ and $(\sk U_2)^\Psi$ have inequivalent $\F$-categories of models in $\fc F$, since in the former $i \colon * \lto *$ is tight, whereas it is not in the latter. Explicitly, a model of $(\sk U_1)^\Psi$ in $\fc F$ comprises a full embedding $\ct M_\tau \to \ct M_\lambda$ and involutions on $\ct M_\tau$ and $\ct M_\lambda$ rendering the following square commutative; whereas a model of $(\sk U_2)^\Psi$ in $\fc F$ comprises a full embedding $\ct M_\tau \to \ct M_\lambda$ and an involution on $\ct M_\lambda$ satisfying no other conditions.
	\[\begin{tikzcd}
		{\ct M_\tau} & {\ct M_\lambda} \\
		{\ct M_\tau} & {\ct M_\lambda}
		\arrow["{\ct M}", from=1-1, to=1-2]
		\arrow["{I_\tau}"', from=1-1, to=2-1]
		\arrow["{I_\lambda}", from=1-2, to=2-2]
		\arrow["{\ct M}"', from=2-1, to=2-2]
	\end{tikzcd}\qedshift\]
\end{remark}

\begin{example}
	Let $\dc M$ and $\dc N$ be the strict double categories generated from the following double graphs; note that neither double category has any non-identity cells.
	\begin{align*}
		\dc M & \defeq 
		\begin{tikzcd}[ampersand replacement=\&]
			0 \& 1
			\arrow["m", "\shortmid"{marking}, shift left=2, from=1-1, to=1-2]
			\arrow["{m'}"', "\shortmid"{marking}, shift right=2, from=1-1, to=1-2]
		\end{tikzcd} &
		\dc N & \defeq 
		\begin{tikzcd}[ampersand replacement=\&]
			X \& Y \\
			{X'} \& {Y'}
			\arrow["n", "\shortmid"{marking}, from=1-1, to=1-2]
			\arrow["\iso"', from=1-1, to=2-1]
			\arrow["\iso", from=1-2, to=2-2]
			\arrow["{n'}"', "\shortmid"{marking}, from=2-1, to=2-2]
		\end{tikzcd}
	\end{align*}
	We define a strict wobbly pseudo functor $F \colon \dc M \to \dc N$ (in the sense of \cref{sketch-pseudocategory}) as follows, where the isomorphisms mediating the compatibility between the object assignments and the domain and codomain of $F(m)$ and $F(n)$ are uniquely determined.
	\begin{align*}
		F(0) & \defeq X &
		F(1) & \defeq Y &
		F(m) & \defeq n &
		F(m') & \defeq n'
	\end{align*}
	Now observe that there exists no pseudo functor $F' \colon \dc M \to \dc N$ that is isomorphic (as a wobbly pseudo functor) to $F$. Without loss of generality, we could take $F'(0) = X$ and $F'(1) = Y$, and $F'(m) = n$. $F'(m')$ is then forced to have domain $X$ and codomain $Y$, for which the only possible choice is $n$. However, to have an isomorphism of wobbly pseudo functors, we must have an invertible cell $n = F(m') \iso F'(m') = n'$, but none such exists. (Note that, by \cite[Theorem~4.5]{pare2015wobbly}, wobbly double functors may be strictified when their codomain has companions and conjoints of isomorphisms, so it is crucial for our example that $\dc N$ does not satisfy this property.)

	Consequently, taking the $\F$-sketch $\sk C$ for pseudocategories (\cref{sketch-pseudocategory}), the inclusion
	\[\FMod_{s, p}(\sk C, \chord\Cat) \to \FMod_{p, p}(\sk C, \chord\Cat)\]
	is not a biequivalence on 2-categories of loose morphisms.
\end{example}

\subsection{Flexible limit 2-sketches}
\label{flexible-limit-2-sketches}

To capture structures such as regular or exact categories~\cite{barr1971exact}, one needs to move beyond the algebraic setting of finite powers to richer classes of limits, such as PIE limits or flexible limits. In this section, we show how the \mbox{$\F$-categorical} approach clarifies the existing literature on PIE-limit 2-theories. In particular, we introduce \emph{cloven flexible limit 2-sketches} (\cref{cloven-sketch}), which generalise the PIE-limit 2-theories of \cite{bourke2021accessible}, and show that they are subsumed by $\F$-sketches (\cref{model-comparison-for-cloven-sketches}).

We start by giving an example that explains the motivation for the introduction of flexible limit 2-sketches as a generalisation of algebraic 2-theories.

\begin{example}
	\label{regular-category}
	The definition of a regular category involves kernel pairs and their coequalisers. Whilst the existence of kernel pairs can be captured using finite powers of categories and adjunctions between them (both of which are algebraic in nature), coequalisers of kernel pairs cannot be. Rather, to express the existence of such coequalisers in a category $\ct X$, one needs a \emph{partial} left adjoint to the diagonal $\Delta \colon \ct X \to \ct X^\rightrightarrows$, defined only at the kernel pairs. In other words, one requires a left adjoint to the diagonal, relative to the functor sending each morphism to its kernel pair, as depicted on the left below.
	\[\begin{tikzcd}
		& {\ct X} \\
		{\ct X^\to} && {\ct X^\rightrightarrows}
		\arrow[""{name=0, anchor=center, inner sep=0}, "\Delta", from=1-2, to=2-3]
		\arrow[""{name=1, anchor=center, inner sep=0}, "\coeq", from=2-1, to=1-2]
		\arrow["\kerp"', hook, from=2-1, to=2-3]
		\arrow["\dashv"{anchor=center}, shift right=2, draw=none, from=1, to=0]
	\end{tikzcd}
	\hspace{4cm}
	\begin{tikzcd}
		{\kerp \comma \Delta} & {\ct X} \\
		{\ct X^\to} & {\ct X^\rightrightarrows}
		\arrow["{{\pi_2}}", from=1-1, to=1-2]
		\arrow[""{name=0, anchor=center, inner sep=0}, "{{\pi_1}}", from=1-1, to=2-1]
		\arrow["\Delta", from=1-2, to=2-2]
		\arrow[""{name=1, anchor=center, inner sep=0}, "\ell", shift left=5, from=2-1, to=1-1]
		\arrow[shorten <=10pt, shorten >=10pt, Rightarrow, from=2-1, to=1-2]
		\arrow["\kerp"', from=2-1, to=2-2]
		\arrow["\dashv"{anchor=center}, draw=none, from=1, to=0]
	\end{tikzcd}\]
	Unlike ordinary adjoints, relative adjoints are not an equational concept, in that it is not possible to define the data of a relative adjunction in a 2-category purely in terms of objects, 1-cells, and 2-cells. However, in some settings, relative adjunctions coincide with absolute left lifts\footnotemark{}, which may be captured using comma objects.
	\footnotetext{In particular, this is true for functors between ordinary categories (though not, in general, for enriched functors~\cite[Remark~5.9]{arkor2024formal}).}%
	In particular, the left relative adjoint $\coeq$ above exists if and only if the first projection from the comma object on the right above admits a left adjoint with identity unit ~\cite[Theorem~1.17]{street1980semitopological}, from which we obtain $\coeq \defeq \pi_2 \circ \ell$.
	Accordingly, we may capture the existence of coequalisers of kernel pairs in a 2-sketch using comma objects. To express the pullback stability of regular epimorphisms, one further needs inverters. For more details regarding the example of regular categories, see \cite[\S6.5]{bourke2021accessible}.
\end{example}

The kinds of two-dimensional limits involved in \cref{regular-category} includes powers, comma objects, and inverters. These are all examples of the well-behaved class of \emph{PIE limits}~\cite{power1991characterization}, which include many other important examples of limits, such as algebra objects for monads, pseudo limits, and lax limits. Accordingly \citeauthor{bourke2021accessible} defined in \cite[\S9]{bourke2021accessible} the notion of \emph{PIE-limit 2-theory}, as well as the 2-categories of models for such theories. In this setting, a lax morphism of models is a lax natural transformation which has identity components at certain (but not all) cone projections; the precise definition is fairly subtle, and we shall recall it below. A particular subtlety is that the choice of cone projections is not intrinsic to the weight, but rather additional structure. \citeauthor{bourke2021accessible} suggested that the situation could be clarified by working $\F$-categorically, and our main aim in this section is to show how this is done.

In fact, we will work with the slightly larger class of \emph{flexible limits}, which are generated by PIE limits together with splittings of idempotents. Whilst flexible limits are not essential for the development of \cite{bourke2021accessible}, they simplify certain constructions: for instance, flexible limits include pullbacks of normal isofibrations, which are convenient in sketching structures such as exact categories~\cite[\S6.5]{bourke2021accessible}.

We begin by recalling the definition of flexible limits and one of their characterisations.

\begin{theorem}
	\label{flexible-weight-characterisation}
	The following are equivalent for a weight $W \colon \tc J \to \CAT$.
	\begin{enumerate}
		\item $W$ is flexible (\aka PIES)~\cite{bird1989flexible}, \ie $W$ lies in the saturation of the class of weights for \underline products, \underline inserters, \underline equifiers, and \underline splittings of idempotents.
		\item Each connected component of the category of elements $\ob \circ W_0 \colon \tc J_0 \to \SET$ admits a naturally weakly initial object, \ie an object $I$ equipped with a morphism $i_A \colon I \to A$ for each object $A$ such that, for each $f \colon A \to B$, we have $f \c i_A = i_B$.
	\end{enumerate}
\end{theorem}

Here $W_0 \colon \tc J_0 \to \CAT_0$ is the underlying functor of $W$, whilst $\ob \colon \CAT_0 \to \SET$ is the functor sending each small category to its set of objects. Therefore, objects of $\el(\ob \circ W_0)$ are pairs $(J \in \tc J, Y \in W(J))$, and morphisms $f \colon (J, Y) \to (J', Y')$ are 1-cells $f \colon J \to J'$ in $\tc J$ satisfying $W(f)(Y) = Y'$.

\begin{proof}
	A characterisation of flexible weights in terms of their double categories of elements is given in \cites[Theorem~3.2]{grandis2019persistent}[Corollary~2.7.2]{verity1992enriched}: in the former, a weight is flexible if and only if each connected component of the `horizontal category' underlying the double category of elements of the weight admits a naturally weak initial object. Unwinding the definition of the double category of elements~\cite[\S1.2]{grandis2019persistent} reveals its horizontal category to be precisely the category of elements of $\ob \circ W_0$.
\end{proof}

\begin{remark}
	A weight $W \colon \tc J \to \CAT$ is furthermore PIE, \ie in the saturation of the class of weights for \underline products, \underline inserters, and \underline equifiers, if and only if, for each of the weakly initial objects $I$, we have $i_I = 1_I$, which is equivalent to saying that $I$ is actually initial in the connected component~\cites[Corollary~3.3]{power1991characterization}[Corollary~2.7.3]{verity1992enriched}.
\end{remark}

\begin{definition}
	A \emph{cleavage} for a flexible weight $W \colon \tc J \to \CAT$ comprises, for each connected component of $\el(\ob \c W_0)$, a cocone over the identity functor, exhibiting a naturally weak initial object. A \emph{cloven flexible weight} is a flexible weight equipped with a cleavage.
\end{definition}

A cleavage for $W$ thus comprises a family $\Omega_W$ of triples $(J \in \tc J, Y \in W(J), i \colon \Delta((J, Y)) \tto 1)$.
Given a weighted cone $\gamma \colon W \tto \tc C(X, D{-})$, we refer to those cone projections $\gamma_{J, Y} \colon X \to D(J)$ with $(J, Y) \in \Omega_W$ as \emph{initial cone projections}.

\begin{remark}
	\label{groupoid}
	For many flexible weights $W$, each connected component of $\el(\ob \circ W_0)$ admits a \emph{unique} initial object, so that $W$ admits a unique cleavage. For instance, this is the case for products, powers, comma objects, and, more generally, for any PIE weight whose domain $\tc J$ contains no invertible 1-cells. An example of a flexible weight (in fact, a PIE weight) which does not have this property is the weight corresponding to the pseudo or lax limit of a diagram $\tc G \to \CAT$, where $\tc G$ is a groupoid viewed as locally discrete 2-category (since pseudo limits and lax limits are PIE limits, and the category of elements in this case will itself be a groupoid).
\end{remark}

The following definition extends \cite[Definition~9.3]{bourke2021accessible} from PIE-limit 2-theories to flexible-limit 2-sketches. The subtlety in the definition concerns the weak homomorphisms of models.

\begin{definition}
	\label{cloven-sketch}
	A \emph{cloven flexible-limit 2-sketch} $\sk S$ is a limit 2-sketch each of whose weights is flexible, together with a cleavage for each weight. Models of $\sk S$ in a 2-sketch $\sk C$ are simply the usual notion of model, \ie 2-functors $M \colon \tc S \to \tc C$ preserving the chosen weighted cones. For each weakness $w \in \W$, a \emph{$w$-morphism} of models $\phi \colon M \to N$ is a $w$-natural transformation such that, for each initial cone projection $\gamma_{J, Y}$, the component $\phi_{\gamma_{J, Y}}$ is the identity. Together with modifications, these form a 2-category $\Mod_w(\sk S, \sk C)$.
\end{definition}

\begin{example}
	Note that it would not be appropriate in \cref{cloven-sketch} to restrict the $w$-morphisms of models to have identity components for \emph{every} cone projection (rather than the specified initial cone projections).

	For instance, consider the colax limit $B \comma g$ of a 1-cell $g \colon A \to B$ in a 2-category $\tc C$ as recalled in \cref{2-categorical-limits}. Denoting by $\b 1 = \{ * \}$ the terminal category, the weight for a colax limit is given by the 2-functor $W \colon \b 2 \to \CAT$ whose image is the functor $(* \mapsto 1) \colon \b 1 \to \b 2$. There are three cone projections: $\pi_{0, *} \colon B \comma g \to A$ -- corresponding to the single object of $\b 1$ -- and $\pi_{1, 0} \colon B \comma g \to B$ and $\pi_{1, 1} \colon B \comma g \to B$ -- corresponding to the two objects of $\b 2$. Now suppose that $\tc C = \fc C_\lambda$ for some $\F$-category $\fc C$. Then certainly we wish $\pi_{0, *}$ and $\pi_{1, 0}$ to be tight, in accordance with \cref{enhanced-2-limits-as-2-limits} (as below left). However, note that the diagram below right commutes.
	\[
	\begin{tikzcd}
		& {B \comma g} \\
		A && B
		\arrow["{\pi_{0, *}}"', from=1-2, to=2-1]
		\arrow[""{name=0, anchor=center, inner sep=0}, "{\pi_{1, 0}}", from=1-2, to=2-3]
		\arrow["f"', squiggly, from=2-1, to=2-3]
		\arrow[shift right, shorten <=18pt, shorten >=29pt, Rightarrow, from=0, to=2-1]
	\end{tikzcd}
	\hspace{2cm}
	\begin{tikzcd}
		& {B \comma g} \\
		A && B
		\arrow["{\pi_{0, *}}"', from=1-2, to=2-1]
		\arrow["{\pi_{1, 1}}", squiggly, from=1-2, to=2-3]
		\arrow["f"', squiggly, from=2-1, to=2-3]
	\end{tikzcd}
	\]
	Consequently, we do not wish for the cone projection $\pi_{1, 1}$ to be tight, since there is no reason that the precomposition of a loose morphism by a tight morphism should be tight. This is the reason we work with cloven weights, which intrinsically identify those projections that must be tight, rather than arbitrary weights, for which it is not clear in general which projections should be tight.
\end{example}

\begin{example}
	Note that it would also not be appropriate in \cref{cloven-sketch} to restrict the $w$-morphisms of models to have identity components for the initial cone projections associated to \emph{every} cleavage simultaneously.

	For instance, let $\tc G$ be the \emph{free-standing involution}, \ie the 2-category with a single object $G$ and a single non-identity 1-cell $g \colon G \to G$ satisfying $g \c g = 1_G$. Suppose that we have a 2-functor $D \colon \tc G \to \tc C$. The pseudo limit of $D$, if it exists, comprises an object $L$ equipped with a projection $\pi \colon L \to D(G)$ and an isomorphism $\varpi \colon D(g) \c \pi \iso \pi$ such that pasting $\varpi$ with itself is the identity 2-cell on $\pi$.

	Suppose that $\tc C = \fc C_\lambda$ for some $\F$-category $\fc C$. Then we wish the unique cone projection $\pi$ to be tight, so that we have the following situation.
	\[\begin{tikzcd}
		& L \\
		{D(G)} && {D(G)}
		\arrow[""{name=0, anchor=center, inner sep=0}, "\pi"', from=1-2, to=2-1]
		\arrow[""{name=1, anchor=center, inner sep=0}, "\pi", from=1-2, to=2-3]
		\arrow["{D(g)}"', squiggly, from=2-1, to=2-3]
		\arrow["\iso"{description}, shift right=2, draw=none, from=0, to=1]
	\end{tikzcd}\]

	Since $L$ is a conical pseudo limit, it is given by the $(\Delta\b1 \colon \tc G \to \CAT)$-weighted pseudo limit of $D$. Consequently, it is equivalently the $(\Delta\b1)'$-weighted 2-limit of $D$, where $(\Delta\b1)'$ is the pseudo morphism classifier of \cite[\S1.7]{bird1984limits} (see also \cite[\S2.6]{grandis2019persistent}). Explicitly, $(\Delta\b1)'(G)$ is the codiscrete category with two objects: $(1_G, *)$ and $(g, *)$; and $(\Delta\b1)'(g)$ swaps the two objects. Recall from \cref{groupoid} that, since $G$ is a groupoid, $\el(\ob \c (\Delta\b1)'_0)$ is also a groupoid: it is the codiscrete category on two objects: $(G, (1_G, *))$ and $(G, (g, *))$.

	Explicitly, the $(\Delta\b1)'$-weighted 2-limit of $D$ comprises an object $L'$ equipped with two cone projections $\pi_1 \colon L' \to D(G)$ and $\pi_2 \colon L' \to D(G)$ (corresponding to the two objects of $(\Delta\b1)'$), and an isomorphism $\varpi' \colon \pi_1 \iso \pi_2$ (corresponding to the isomorphism between the two aforementioned objects). Furthermore, 2-naturality imposes $D(g) \c \pi_1 = \pi_2$ and $D(g) \c \pi_2 = \pi_1$ together with a pasting condition for $\varpi'$. Consequently, taking $\pi$ to be equal either to $\pi_1$ or to $\pi_2$ (the two choices are equivalent), it is evident that the universal property for $L'$ is the same as that for $L$. Accordingly, we wish to view either $\pi_1$ or $\pi_2$ as the `canonical projection' (which we take to be tight), and the other to be a `derived projection' (which we take to be loose, since it is the composite of a tight morphism with a loose morphism).
\end{example}

An obvious shortcoming with \cref{cloven-sketch} is that a cloven flexible-limit 2-sketch is not simply a 2-sketch satisfying a certain property, but rather a 2-sketch equipped with additional structure, namely the cleavages for each weight.

We shall show that each cloven flexible-limit 2-sketch may be viewed naturally as an $\F$-sketch (without additional structure), and that the 2-categories of models for a cloven flexible-limit 2-sketch are subsumed by the $\F$-categories of models for the corresponding $\F$-sketch. The main ingredient is the following construction, which shows that each cloven flexible weight gives rise to an $\F$-weight, which encodes the requirement that the chosen initial cone projections are required to be tight.

\begin{lemma}
	\label{F-weight-from-cloven-weight}
	Every cleavage $\Omega$ for a flexible weight $W \colon \tc J \to \CAT$ induces an $\F$-weight $W^\Omega \colon \inchord{\tc J} \to \fc F$ such that $(W^\Omega)_\lambda = W$. An $\F$-category $\fc S$ admits $W^\Omega$-weighted limits if and only if $\fc S_\lambda$ admits $W$-weighted limits and the initial cone projections are tight and jointly detect tightness.
\end{lemma}

\begin{proof}
	The construction of $W^\Omega$ for $W$ a PIE weight is described abstractly in the proof of \cite[Theorem~6.12]{lack2012enhanced}. The proof extends with minor modifications to flexible weights: we spell out the construction. Since the domain of $W^\Omega$ is inchordate, it suffices to define a 2-functor from $\tc J$ to the underlying 2-category of $\fc F$. For $J \in \tc J$, the full inclusion $W^\Omega(J) \in \fc F$ is given by the full subcategory inclusion of $W(J)$ spanned by those objects $Y$ such that $(J, Y) \in \Omega_W$. The action of $W^\Omega$ on 1-cells and 2-cells is given by that of $W$. By definition, this choice satisfies $(W^\Omega)_\lambda = W$. The characterisation of completeness under $W^\Omega$-weighted limits follows directly from the characterisation of general $\F$-limits in \cref{enhanced-2-limits-as-2-limits}.
\end{proof}

\begin{remark}
	\label{non-isomorphic-naturally-initial-objects}
	In the case that a flexible weight admits two cleavages, the two corresponding  $\F$-weights need not be isomorphic.  For instance, let $\tc J$ be the free-standing section--retraction pair, as below. Observe that $A$ is naturally weakly initial and $B$ is initial, so that the constant weight $\Delta\b1 \colon \tc J \to \CAT$ is flexible by \cref{flexible-weight-characterisation}, since $\el(\ob \c W_0) \iso \tc J_0$. The $\Delta\b1$-weighted limit of a 2-functor $D$ is simply given by evaluation at the initial object $B$.
	\[\begin{tikzcd}
		A & B
		\arrow["sr", from=1-1, to=1-1, loop, in=145, out=215, distance=10mm]
		\arrow["r", shift left=2, from=1-1, to=1-2]
		\arrow["s", shift left=2, from=1-2, to=1-1]
	\end{tikzcd}\]
	There are two choices of cleavage associated with $\Delta\b1$, given by $(A, *)$ and $(B, *)$ respectively. Furthermore, the two associated $\F$-weights induced by \cref{F-weight-from-cloven-weight} are clearly not isomorphic: a limit weighted by the former has a strict projection $DB \to DA$, whereas the latter does not.
\end{remark}

With \cref{F-weight-from-cloven-weight} in place, we can now describe the $\F$-sketch associated to a cloven flexible-limit 2-sketch. Just as for algebraic 2-theories (which are particular instances of cloven flexible-limit 2-sketches), we make use of the free limit $\F$-sketch construction of \cref{free-F-sketch} (though, in contrast to the settings of \cref{examples-of-cofree-F-sketches} and \cref{algebraic-2-theories}, the $\F$-weights induced by \cref{F-weight-from-cloven-weight} will not in general be tight.)

\begin{theorem}
	\label{model-comparison-for-cloven-sketches}
	Let $\sk S$ be a cloven flexible-limit 2-sketch, and denote by $\Psi \defeq \{ W^\Omega \mid W \in \sk S \}$ the class of $\F$-weights induced by the cleavages via \cref{F-weight-from-cloven-weight}. For each weakness $w \in \W$ and each limit 2-sketch $\sk C$, there is a 2-natural isomorphism of 2-categories:
	\[\Mod_w(\sk S, \sk C) \iso (\FMod_w(\sk S^\Psi, \chord{\sk C}))_\lambda\]
	where we denote by $\sk S^\Psi$ the free $\Psi$-limit $\F$-sketch associated to the $\Psi_\lambda$-limit 2-sketch underlying $\sk S$.
\end{theorem}

\begin{proof}
	Follows immediately from the description of the models for $\sk S^\Psi$ in \cref{models-of-free-F-sketches}, noting that (1) there are no nontrivial tight diagram morphisms, since every weight has inchordate domain; and that (2) by construction of $\Psi$, the tight cone projections are precisely the initial cone projections associated to $\sk S$.
\end{proof}

Consequently, the study of cloven flexible-limit 2-sketches and their 2-categories of models, whose definitions are rather complex, are naturally subsumed by the study of $\F$-sketches and their $\F$-categories of models.

We conclude by discussing the example of the cloven flexible-limit 2-sketch for regular categories anticipated in \cref{regular-category}.

\begin{example}
	\label{regular-category-sketch}
	As indicated in \cref{regular-category}, regular categories are the models for a PIE-limit 2-sketch $\sk R$, the full details for which may be extracted straightforwardly from the description in \cite[\S6.5]{bourke2021accessible}. The weighted limits involved are powers, comma objects and inverters, and by \cref{groupoid}, each of these admits a unique cleavage.  Therefore by \cref{model-comparison-for-cloven-sketches} we obtain a canonical enhancement of $\sk R$ to an $\F$-sketch $\sk R^\Psi$. Our intention here is to indicate which morphisms in $\sk R$ become tight in the associated $\F$-sketch $\sk R^\Psi$, to illustrate a particular subtlety.

	Before giving the details, we may make an educated guess as to which morphisms should be tight.  Intuitively, in the diagram below,
	\[\begin{tikzcd}
		& {\ct X} \\
		{\ct X^\to} && {\ct X^\rightrightarrows}
		\arrow[""{name=0, anchor=center, inner sep=0}, "\Delta", from=1-2, to=2-3]
		\arrow[""{name=1, anchor=center, inner sep=0}, "\coeq", squiggly, from=2-1, to=1-2]
		\arrow["\kerp"', squiggly, hook, from=2-1, to=2-3]
		\arrow["\dashv"{anchor=center}, shift right=2, draw=none, from=1, to=0]
	\end{tikzcd}\]
	one might expect the morphisms $\coeq$ and $\kerp$ to be loose, since regular functors commute with these only up to coherent isomorphism, and the diagonal $\Delta$ to be tight, since regular functors commute with diagonals strictly.

	However this does not quite match the reality.  Consider the part of the $\F$-sketch $\sk R^\Psi$ depicted below.
	\[\begin{tikzcd}
		{\kerp \comma \Delta} & {\ct X} \\
		{\ct X^\to} & {\ct X^\rightrightarrows}
		\arrow["{{\pi_2}}", from=1-1, to=1-2]
		\arrow[""{name=0, anchor=center, inner sep=0}, "{{\pi_1}}", from=1-1, to=2-1]
		\arrow["\Delta", squiggly, from=1-2, to=2-2]
		\arrow[""{name=1, anchor=center, inner sep=0}, "\ell", shift left=5, squiggly, from=2-1, to=1-1]
		\arrow[shorten <=10pt, shorten >=10pt, Rightarrow, from=2-1, to=1-2]
		\arrow["\kerp"', squiggly, from=2-1, to=2-2]
		\arrow["\dashv"{anchor=center}, draw=none, from=1, to=0]
	\end{tikzcd}\]
	Here the comma object projections $\pi_1$ and $\pi_2$ are tight.  As anticipated, the kernel pair functor $\kerp \colon \ct X^\to \rightsquigarrow \ct X^\rightrightarrows$ is loose. The left adjoint $\ell$ is also loose and hence, as anticipated, the composite $\coeq \defeq \pi_2 \circ \ell \colon \ct X^\to \rightsquigarrow X$ is loose.

	However, the diagonal $\Delta$ is also loose. Whilst at first this may appear unexpected, note that, with respect to models in $\Psi$-complete $\F$-categories, $\Delta$ is treated as though it were tight. In particular, if $M$ is a model of $\sk R^\Psi$ in a $\Psi$-complete $\F$-category, the morphism $M(\Delta)$ will in fact be tight. This follows from the fact that each of its composites with the two tight power projections $X^\rightrightarrows \to X$ are tight (see \cref{enhanced-2-limits-as-2-limits}). Similarly, if $\phi \colon M \to N$ is a $w$-morphism of $\sk R^\Psi$-models in a $\Psi$-complete $\F$-category, the 2-cell component $\phi_\Delta$ will be an identity 2-cell. Abstractly, this behaviour amounts to the fact that $\Delta$ will be tight in the free $\F$-theory on the $\F$-sketch $\sk R^\Psi$ (\cf~\cite[\S6.4]{kelly1982basic}). This is an instance of a general pattern: morphisms that intuitively correspond to mediating morphisms into a limit are not required to be tight in an $\F$-sketch, even though they may be forced to be tight in models.
\end{example}

\section{Future directions}
\label{future-directions}

There are many potential directions in which the theory of enhanced 2-sketches could be taken. We mention a few.
\begin{itemize}
	\item In the one-dimensional setting, the category of models of a limit sketch is locally presentable~\cite[Proposition~1.52]{adamek1994locally}; a similar fact holds for enriched limit sketches~\cite[Proposition~10.4]{kelly1982structures}. However, the theory of local presentability for enriched limit sketches concerns enriched categories of models and their \emph{strict} morphisms; in the two-dimensional setting, we are typically interested in the weaker notions of morphism between models. It remains to be seen whether one may analogously characterise the $\F$-categories of models of $\F$-sketches and their pseudo/lax/colax morphisms. This would refine the theory of locally bipresentable 2-categories developed by \textcite{di2022biaccessible}.
	\item Along similar lines, in \cite[Theorem~9.4]{bourke2021accessible}, it was conjectured that the 2-category of models and pseudo morphisms of any appropriately nice PIE-limit 2-theory would be accessible as a 2-category, and admit certain limits and colimits (\cf~\cite[Definitions~4.1 \& 4.5]{bourke2021accessible}). The proof of this conjecture was deferred until a theory of limit $\F$-theories had been developed.
	\item We have not investigated the construction of the free $\F$-theory generated by an $\F$-sketch: such a construction exists by \cite[\S6.4]{kelly1982basic}, but it remains to be seen how the construction interacts with weak morphisms. This is closely related to the theory of local presentability and duality for $\F$-categories mentioned above.
	\item An open problem is to obtain a deeper understanding of limits in $\F$-categories of models and their $(w', w)$-morphisms, which we expect to be closely connected to the main results of \cite{lack2012enhanced}. The dotted limits of \cite{ko2023dotted} may be a useful tool for this pursuit.
	\item It is likely that there are interesting examples of mixed $\F$-sketches, combining \mbox{$\F$-categorical} limits and colimits, whose models would correspond to a two-dimensional analogue of accessible categories (\cf~\cite[Corollary~2.61]{adamek1994locally}).
	\item The theory of enhanced 2-sketches extends in many cases the theory of \mbox{2-monads} (and, more generally, of $\F$-monads). As such, many of the topics studied in two-dimensional monad theory may be studied in the context of enhanced \mbox{2-sketches}, for instance (co)lax models, weak morphism classifiers, flexibility, and doctrinal adjunction.
	\item The double categorical development of \cite{lambert2024cartesian} likely admits a refinement to a notion of enhanced double category and enhanced double sketch.
	\item $\F$-categories capture 2-categorical structures in which there are two natural notions of morphism, one of which is stricter than the other. However, in many examples, the structure formed by the morphisms is more complex: for instance, monoidal categories have at least four natural notions of morphism: strict, pseudo, lax, and colax (more if one includes notions such as normal monoidal functors). By considering this extra structure, it is possible to sharpen some of our results, such as allowing \cref{Mod-skew-symmetry} to be parameterised by four weaknesses rather than two. This would permit, for instance, the consideration of lax monoidal pseudo double functors between monoidal double categories. This requires generalisation from $\F$-categories to more expressive structures, as also suggested in \cite[\S1]{lack2012enhanced}.
\end{itemize}

\appendix

\section{Creation of certain enhanced 2-limits}
\label{creation-of-certain-F-limits}

In this appendix, we prove a technical result regarding the creation of certain \mbox{$\F$-categorical} limits needed in \cref{multiple-perspectives}. To do so, we make use of the formalism of \emph{marked lax limits}. These are of the same expressive power as $\CAT$-weighted limits and can be more convenient. In particular, this is the case in the proof of \cref{created}.

We begin by recalling the notion of marked lax limit. Let $\fc J$ be an $\F$-category and $\tc C$ be a 2-category. Denote by $\Delta \colon \tc C \to \FFun_{s, l}(\fc J, \chord{\tc C})_\lambda$ the diagonal 2-functor into the 2-category of $\F$-functors $\fc J \to \chord{\tc C}$ (these are equivalently 2-functors $\fc J_\lambda \to \tc C$, as recalled in \cref{chord-and-inchord}) and loose $(s, l)$-natural transformations (\cref{functor-F-category}). The \emph{marked lax limit} of a 2-functor $D \colon \fc J_\lambda \to \tc C$ is given by a partial right adjoint to $\Delta$, \ie an object $\mlim D \in \tc C$ equipped with a universal loose $(s, l)$-natural transformation $\pi \colon \Delta(\mlim D) \tto D$, exhibiting a 2-natural isomorphism as follows.
\begin{equation}
	\tc C(C, \mlim D) \iso \FFun_{s, l}(\fc J, \chord{\tc C})_\lambda(\Delta C, D)
\end{equation}

In the case that $W \colon \tc J \to \CAT$ is a weight, one can construct an $\F$-category of elements $\el(W)$ over which marked lax limits have the same universal property as $W$-weighted limits: in particular, existence and preservation of weighted limits can be understood in terms of marked lax limits. For further details, we refer the reader to \cite{ko2023dotted} and the references therein.

\begin{remark}
	The concept of a marked lax limit has appeared under different names; see \cite{ko2023dotted} for an account of the history of the concept. In particular, in this context \mbox{$\F$-categories} are often referred to as \emph{marked 2-categories}; and the morphisms of $\FFun_{s, l}(\fc J, \chord{\tc C})_\lambda$ are referred to as \emph{marked lax natural transformations}.
\end{remark}

\begin{proposition}\label{created}
	All limits are $(s, s)$-created. Furthermore, tight limits are $(s, w)$-created for all weaknesses $w \in \W$.
\end{proposition}

\begin{proof}
	The first part is straightforward and holds over a general base of enrichment. Indeed, $\FMod_s(\sk S, \sk C) \hookrightarrow [\fc S, \fc C]$ is the full sub-$\F$-category of the functor $\F$-category containing the models. Now, by the pointwise nature of limits in enriched functor categories~\cite[\S3.3]{kelly1982basic}, for each $\F$-weight $W$, the $\F$-category $[\fc S, \fc C]$ admits $W$-weighted limits if $\fc C$ does and the evaluation $\F$-functors $\ph(S) \colon [\fc S, \fc C] \to \fc C$ preserve and jointly reflect them. Therefore, it suffices to show that $\FMod_s(\sk S, \sk C) \hookrightarrow [\fc S, \fc C]$ is closed under $W$-weighted limits, which holds since weighted limits commute with weighted limits~\cite[(3.22)]{kelly1982basic}.

	For the second part, observe that the following triangle of $\F$-functors commutes, for each $S \in \fc S$.
	\[\begin{tikzcd}
		{\FMod_s(\sk S, \sk C)} && {\FMod_w(\sk S, \sk C)} \\
		& {\fc C}
		\arrow[hook, from=1-1, to=1-3]
		\arrow["{\ph(S)}"', from=1-1, to=2-2]
		\arrow["{\ph(S)}", from=1-3, to=2-2]
	\end{tikzcd}\]
	Let $W \colon \tc J \to \CAT$ be a weight, and denote by $\chord{W} \colon \chord{\tc J} \to \fc F$ the associated $\F$-weight (\cref{tight-limits}). Since $\chord{\tc J}$ is chordate, each diagram $D \colon \chord{\tc J} \to \FMod_w(\sk S, \sk C)$ factors through $\FMod_s(\sk S, \sk C) \hookrightarrow \FMod_w(\sk S, \sk C)$. Since all limits are $(s, s)$-created, to show that a \mbox{$\chord{W}$-weighted} limit of $D$ is $(s, w)$-created, it is necessary and sufficient to show that its limit in $\FMod_s(\sk S, \sk C)$ is preserved by the inclusion $\FMod_s(\sk S, \sk C) \ffto \FMod_w(\sk S, \sk C)$. Since the inclusion is the identity on tight morphisms, this is equally to say that the inclusion \mbox{2-functor} $\FMod_s(\sk S, \sk C)_{\tau} = \FMod_w(\sk S, \sk C)_{\tau} \hookrightarrow \FMod_w(\sk S, \sk C)_{\lambda}$ preserves $W$-weighted limits, which will follow immediately if we can show that $\FFun_w(\fc S,\fc C)_\tau \hookrightarrow \FFun_w(\fc S,\fc C)_\lambda$ preserves them. In summary, it suffices to show that if $W$ is a weight such that $\fc C_{\tau}$ admits $W$-weighted limits and these are preserved by $\fc C_{\tau} \hookrightarrow C_{\lambda}$, then $\FFun_s(\fc S, \fc C)_\tau \hookrightarrow \FFun_w(\fc S,\fc C)_\lambda$ also preserves them.

	We first treat the case $w = l$, before dealing with the other two cases. Let $\fc J \defeq \el(W)$ be the $\F$-category of elements corresponding to $W$ and let $D \colon \fc J_\lambda \to \FFun_s(\fc S, \fc C)_\tau$ be a 2-functor admitting a marked lax limit $\pi \colon\Delta \mlim D \tto D$.
	Denoting by $J \colon \FFun_s(\fc S, \fc C)_\tau \hookrightarrow \FFun_l(\fc S,\fc C)_\lambda$ the inclusion, consider a loose $(s, l)$-natural transformation $\theta \colon \Delta F \tto JD$, so that its components are morphisms of $\FFun_l(\sk S, \sk C)_\lambda$.
	Applying the evaluation $\F$-functor at each $S \in \fc S$ gives the components of a loose \mbox{$(s, l)$-natural} transformation as below left. Therefore, the universal property of the marked lax limit $(\mlim F)(S)$ in $\fc C_{\lambda}$ induces a unique loose morphism $\tp\theta(S) \colon F(S) \rightsquigarrow (\mlim D)(S)$ satisfying the following equation.
	\[
	\begin{tikzcd}
		&& {DA(S)} \\
		{F(S)} \\
		&& {DB(S)}
		\arrow["{Df(S)}", from=1-3, to=3-3]
		\arrow["{\theta_A(S)}", squiggly, from=2-1, to=1-3]
		\arrow[""{name=0, anchor=center, inner sep=0}, "{\theta_B(S)}"', squiggly, from=2-1, to=3-3]
		\arrow["{\theta_f(S)}"{description}, shorten <=7pt, shorten >=7pt, Rightarrow, from=1-3, to=0]
	\end{tikzcd}
	\quad = \quad
	\begin{tikzcd}
		&&& {DA(S)} \\
		{F(S)} & {(\mlim D)(S)} \\
		&&& {DB(S)}
		\arrow["{Df(S)}", from=1-4, to=3-4]
		\arrow["{\tp\theta(S)}", squiggly, from=2-1, to=2-2]
		\arrow["{\pi_A(S)}", from=2-2, to=1-4]
		\arrow[""{name=0, anchor=center, inner sep=0}, "{\pi_B(S)}"', from=2-2, to=3-4]
		\arrow["{\pi_f(S)}"{description}, shorten <=8pt, shorten >=8pt, Rightarrow, from=1-4, to=0]
	\end{tikzcd}
	\]
	For each loose morphism $s \colon S \lto S'$ in $\fc S$, the two-dimensional universal property of $(\mlim D)(S')$ induces a unique 2-cell $\tp\theta(s)$ satisfying the following equation.
	\[
	\begin{tikzcd}[sep=large]
		{F(S)} & {DA(S)} \\
		{F(S')} & {DA(S')}
		\arrow["{\theta_A(S)}", squiggly, from=1-1, to=1-2]
		\arrow["{F(s)}"', squiggly, from=1-1, to=2-1]
		\arrow["{\theta_A(s)}"{description}, shorten <=4pt, shorten >=4pt, Rightarrow, from=1-2, to=2-1]
		\arrow["{DA(s)}", squiggly, from=1-2, to=2-2]
		\arrow["{\theta_A(S')}"', squiggly, from=2-1, to=2-2]
	\end{tikzcd}
	\quad = \quad
	\begin{tikzcd}[sep=large]
		{F(S)} & {(\mlim D)(S)} & {DA(S)} \\
		{F(S')} & {(\mlim D)(S')} & {DA(S')}
		\arrow["{\tp\theta(S)}", squiggly, from=1-1, to=1-2]
		\arrow["{F(s)}"', squiggly, from=1-1, to=2-1]
		\arrow["{\pi_A(S)}", from=1-2, to=1-3]
		\arrow["{\tp\theta(s)}"{description}, shorten <=8pt, shorten >=8pt, Rightarrow, from=1-2, to=2-1]
		\arrow["{(\mlim D)(s)}"{description}, squiggly, from=1-2, to=2-2]
		\arrow["{DA(s)}", from=1-3, to=2-3]
		\arrow["{\tp\theta(S')}"', squiggly, from=2-1, to=2-2]
		\arrow["{\pi_A(S')}"', from=2-2, to=2-3]
	\end{tikzcd}
	\]
	The lax naturality of $\tp\theta(s)$ follows from the defining equality above. Moreover, if $s$ is tight, then, since $\theta_A$ is a loose lax natural transformation, each 2-cell $\theta_A(s)$ is the identity. Now, since the cone projections $\pi_A(S)$ jointly reflect identity 2-cells, $\tp\theta(s)$ is also the identity, so that $\tp\theta \colon F \rightsquigarrow \mlim D$ is a morphism of $\FMod_l(\sk S, \sk C)$ such that $\pi \circ \Delta\tp\theta = \theta$. Moreover, by construction, $\tp\theta$ is the unique morphism giving such a factorisation. The 2-dimensional universal property of the marked lax limit is easily verified.

	For the case $w = p$, one simply observes that, since the cone projections $\pi_A(S)$ jointly reflect invertibility of 2-cells, $\tp\theta(s)$ will be invertible if each $\theta_A(s)$ is. In the case $w = c$, the directions of $\tp\theta(s)$ and $\theta_A(s)$ are reversed, but everything else remains the same.
\end{proof}

\printbibliography

\end{document}